\documentclass[11pt,a4paper]{article}
\usepackage[latin,english]{babel}
\usepackage[T1]{fontenc}
\usepackage{eufrak}
\usepackage{geometry}
\usepackage{amsmath}
\usepackage{bbm}
\usepackage{amssymb}
\usepackage{amsbsy}
\usepackage{amsthm}
\usepackage{makeidx}
\usepackage{mathtools}
\usepackage{mathabx, mathrsfs, dsfont}
\usepackage{cases}
\usepackage{braket}
\usepackage[toc,page]{appendix}
\usepackage{dirtytalk}
\usepackage{pdfsync}
\usepackage{hyperref}
\usepackage{graphicx}
\usepackage{epstopdf}
\usepackage{todonotes}
\usepackage{comment}
\usepackage{tikz}
\usepackage[normalem]{ulem}
\usetikzlibrary{shapes,decorations,calc,fit}
\usepackage{graphicx}
\usepackage{fancyhdr}
\usepackage{math}
\usepackage{cite}
\usepackage{color}
\usepackage{subcaption}
\usepackage{multicol}
\mathtoolsset{showonlyrefs}
\usepackage[shortlabels]{enumitem}

\renewcommand{\T}{\mathbb{T}}
\renewcommand{\meanval}[1]{\mathbb{E}\left[#1\right]}

\renewcommand{\mod}{{\,{\rm mod}\,}}

\newcommand{\Tr}[1]{\mathrm{Tr}\left(#1\right)}

\newcommand{\Volt}{\textrm{Volt}}
\newcommand{\dAL}{\textrm{dAL}}

\newcommand{\dSc}{\textrm{dS}}

\newtheorem{lemma}{Lemma}[section]
\newtheorem{assumption}{Hypotheses}[section]
\newtheorem{theorem}[lemma]{Theorem}
\newtheorem{proposition}[lemma]{Proposition}
\newtheorem{corollary}[lemma]{Corollary}
\newtheorem{remark}[lemma]{Remark}

\newtheorem{definition}[lemma]{Definition}

\newcommand{\diag}[1]{\textbf{diag}\left( #1\right)}

\numberwithin{equation}{section}

\author{G. Mazzuca\footnote{Department of Mathematics, Tulane University, New Orleans, Louisiana, USA. \newline
\textit{email:} gmazzuca@tulane.edu}, R. Memin\footnote{IMT, 
Toulouse, France. \newline
\textit{email:} ronan.memin@math.univ-toulouse.fr}}
\title{CLT for $\beta$-ensembles at high-temperature, and for integrable systems: a transfer operator approach.}

\begin{document}
\maketitle

\begin{abstract}
In this paper, we prove a polynomial Central Limit Theorem for several integrable models, and for the $\beta$-ensembles at high-temperature with polynomial potential. Furthermore, we connect the mean values, the variances and the correlations of the moments of the Lax matrices of these integrable systems with the ones of the $\beta$-ensembles. Moreover, we show that the local functions' space-correlations decay exponentially fast for the considered integrable systems. For these models, we also established a Berry--Esseen type bound.
\end{abstract}

\tableofcontents

\section{Introduction}
In this paper, we establish a Central Limit Theorem for eigenvalue fluctuations of the Lax matrices of several integrable models, and for the classical $\beta$-ensembles of random matrix theory in the so-called high-temperature regime. Motivated by the recent observation that both topics are closely related, we connect, for the considered integrable systems, the limiting variance in the CLT to the one of the associated $\beta$-ensemble in the high-temperature regime, by establishing a relation between their free energies. Furthermore, we establish exponential decay of space correlation for the considered integrable models and a Berry--Esseen type bound.

We start by describing the case of the Toda chain, which will serve both as a motivation for the results and as a way to introduce the key objects of the present work.
\subsection{The Toda chain: an introductory case}
\label{subsection Toda}

The classical Toda chain (or Toda lattice) \cite{Toda1970} is the  dynamical system described by the following Hamiltonian: 
	\begin{equation}
	H_{T}(\bp,\bq):=\frac{1}{2} \sum_{j=1}^{N}{p_j^2} + \sum_{j=1}^{N}V_T(q_{j+1}-q_{j} )\ , \quad  V_T(x) =  e^{- x} +  x - 1\, ,
	\label{toda}
	\end{equation}
	
	with periodic boundary conditions  $ \, q_{j+N} = q_{j} + \Omega \quad \forall \, j \in \mathbb{Z}$, where $\Omega > 0$ is fixed. Its equations of motion take the form
	\begin{equation}
	\label{Todaeq}
	\dot{q}_j=\dfrac{\partial H_T}{\partial p_j}=p_j,\quad \dot{p}_j=-\dfrac{\partial H_T}{\partial q_j}=V_T'(q_{j+1}-q_{j}) -V_T'(q_{j}-q_{j-1}), \;\;j=1,\dots,N\, .
	\end{equation}
	
	It is well known that the Toda chain is an integrable system \cite{Toda1970,Henon1974}. One way to prove it is to put the Toda equations  in Lax pair form. This was obtained by Flaschka \cite{Flaschka1974}, and Manakov \cite{Manakov1975} through the following {\em non-canonical} change of coordinates:
	
	\begin{equation}
	a_j := -p_j \, , \qquad b_j:=  e^{\frac{1}{2}(q_j-q_{j+1})} \equiv e^{-\frac{1}{2}r_j}, \qquad 1 \leq j \leq N\,,
	\label{bavariable}
	\end{equation}
	where $r_j=q_{j+1}-q_j$ is the relative distance between particles $j$ and $j+1$.

	Defining the Lax operator $L$ as the periodic Jacobi matrix
	\cite{VanMoerbeke1976}
	\begin{equation} \label{eq:L_toda}
	L := \left( \begin{array}{ccccc}
	a_{1} & b_{1} & 0 & \ldots &  b_{N} \\
	b_{1} & a_{2} & b_{2} & \ddots & \vdots \\
	0 & b_{2} & a_{3} & \ddots & 0 \\
	\vdots & \ddots & \ddots & \ddots & b_{N-1} \\
	b_{N} & \ldots & 0 & b_{N-1} & a_{N} \\
	\end{array} \right) \, ,
	\end{equation}
	and the antisymmetric matrix $B$ 
	\begin{equation}
	B :=\frac{1}{2} \left( \begin{array}{ccccc}
	0 & b_{1} & 0 & \ldots &  -b_{N} \\
	-b_{1} & 0 & b_{2} & \ddots & \vdots \\
	0 & -b_{2} & 0 & \ddots & 0 \\
	\vdots & \ddots & \ddots & \ddots & b_{N-1} \\
	b_{N} & \ldots & 0 & -b_{N-1} & 0 \\
	\end{array} \right) \, ,
	\end{equation}
	a straightforward calculation shows that the equations of motions \eqref{Todaeq} are equivalent to
	\begin{equation}
        \label{laxToda}
	\dot{L}= \left[B;L\right]\, ,
	\end{equation}
	where $[B;L] = BL-LB$ is the commutator of two matrices.
	This form implies that the eigenvalues of $L$ are conserved, and therefore $\Tr{L^k}$, $k=1,\ldots,N$ are constants of motions for the Toda lattice, so the system is integrable \cite{Babelon}. We call these quantities \textit{conserved fields}. Also, note that the sum of the streches $\sum_{j=1}^N r_j = \Omega$
    is conserved along the dynamic.

    Let $\alpha>0$ and $V:\R \to \R$ be a Lebesgue measurable function such that for some $C\in \R$, $\lim_{x\to \pm \infty} \frac{V(x)}{x^2} > 0$ and $V(x)>C$ for all $x\in \R$.
    We introduce, after \cite{Spohn1}, the probability measure on $\R^{N}\times\R^N_+$
	
	\begin{equation}\label{eq:Gibbs_todaIntro}
	\di \mu_{N,\text{Toda}}^{\alpha,V}(\ba,\bb) :=    \frac{1}{Z_{N}^\text{Toda}(\alpha,V)} \  \prod_{j=1}^N b_j^{2\alpha -1}\mathbbm{1}_{b_j>0}e^{-\Tr{V(L)} } \di \ba \,  \di \bb \,,
	\end{equation}
    where $L$ is a periodic tridiagonal matrix of the form \eqref{eq:L_toda} and
    $$ \Tr{V(L)}=\sum_{j=1}^N V(\lambda_j(L))\,,$$
    denoting by $\lambda_j(L)$, $1\leq j \leq N$ the eigenvalues of $L$.
    
	The measure $\mu_{N,\text{Toda}}^{\alpha,V}$ is called \textit{Generalized Gibbs Ensemble} (GGE) of the Toda chain with pressure parameter $\alpha >0$ and potential $V$. Because both quantities $\Tr{V(L)}$ and
    $$ \prod_{j=1}^N b_j = \exp\left\{-\frac{1}{2}\sum_{j=1}^N r_j\right\}$$
    are constants of motion, $\mu_{N,T}^{\alpha,V}$ is invariant under the Toda flow, this is a consequence of a more general result, see \cite{Guionnet2002,mazzuca_int_sys}.
    
 Note that the conditions on $V(x)$ are enough to ensure that the normalizing constant $Z^\text{Toda}_N(\alpha,V)$ is finite. It is called the \textit{partition function} of the system, and it reads
 $$ Z^\text{Toda}_N(\alpha,V) = \int_{\R^N\times \R_+^N}\prod_{j=1}^N b_j^{2\alpha -1}\mathbbm{1}_{b_j>0}e^{-\Tr{V(L)} } \di \ba \,  \di \bb \,.$$

In the framework of generalized hydrodynamics, it is expected that for generic initial data, the long-time distribution of the Toda chain converges towards a GGE whose pressure $\alpha$ and potential $V$ are prescribed by the initial condition. Formally, it is expected that for sufficiently regular observables $f:\R^N\times \R_+^N \to \R$ we have as $t\to +\infty$
$$ \frac{1}{t}\int_0^t f(\ba(s)),\bb(s)) \di s \to \E_N^{\alpha,V}[f(\ba,\bb)]\,,$$
where the expectation is taken under $\mu^{\alpha,V}_{N,T}$, and where $\alpha$, $V$ depend on $(\ba(0),\bb(0))$. From there, understanding the GGE with arbitrary potential and pressure parameter is key to understanding the Toda lattice's long-time behaviour. 

Under $\mu^{\alpha,V}_{N,\text{Toda}}$, $L$ is a random matrix whose law is invariant under the Toda flow. It turns out, as was observed by Spohn in \cite{Spohn1}, that its spectral properties can be linked with the ones of the real $\beta$-ensemble in the high-temperature regime.

Fix $N\in \N$, $\beta>0$ and $V:\R\to \R$ Lebesgue measurable such that for some $\delta>0$, $c\in \R$, we have for all $x\in \R$
$$ V(x) \geq (N\beta + 1 + \delta) \log|x| + c\,.$$
The real $\beta$-ensemble of size $N$ with potential $V$ is the probability measure on $\R^N$ given by
\begin{equation}
    \label{def:real_beta_ens}
    d\mathbf{P}^{\beta,V}_N(\lambda_1,\ldots, \lambda_N)=\frac{1}{\mathbf{Z}^{\beta,V}_N}\prod_{1\leq i < j \leq N}|\lambda_i-\lambda_j|^\beta e^{-\sum_{j=1}^N V(\lambda_j)}\di \boldsymbol{\lambda}\,.
\end{equation}
When $\beta \in \{1,2,4\}$ and $V$ is quadratic, this measure is the joint law of the (unordered) spectrum of respectively the GOE, GUE and GSE (see \cite[Section 2.5]{AGZ} for example), giving matrix representation for the $\beta$-ensemble in those specific cases. In \cite{dued}, Dumitriu and Edelman gave a tridiagonal matrix model for the $\beta$-ensemble of size $N$ with quadratic potential for \textbf{any choice of} $\boldsymbol{\beta>0}$: 

Let $T$ be the following random tridiagonal matrix with independent coefficients (up to the symmetry): 

\begin{itemize}
    \item The diagonal entries $T_{i,i}$ are standard Gaussians
    \item For $1\leq i \leq N-1$, the off-diagonal entry $T_{i,i+1}=T_{i+1,i}$ is distributed as $\frac{1}{\sqrt{2}}\chi_{\beta(N-i)}$\,,
\end{itemize}
Where for $a>0$, the $\chi_a$ law is given by the density 
$$f_a(x)=\frac{2^{1-a/2}}{\Gamma(a/2)}x^{a-1}e^{-x^2/2}\,.$$
Then, \cite[Theorem 2.12]{dued} asserts that the distribution of the eigenvalues of $T$ is given by the $\beta$-ensemble of size $N$ with potential $V(x)=x^2/2$.

With this last point point in mind, let 
\begin{equation} 
\label{tridiagBeta}
	T := \left( \begin{array}{ccccc}
	a_{1} & b_{1} & 0 & \ldots &  0 \\
	b_{1} & a_{2} & b_{2} & \ddots & \vdots \\
	0 & b_{2} & a_{3} & \ddots & 0 \\
	\vdots & \ddots & \ddots & \ddots & b_{N-1} \\
	0 & \ldots & 0 & b_{N-1} & a_{N} \\
	\end{array} \right) \, 
\end{equation}
be the (nonperiodic) tridiagonal matrix whose entries are distributed according to the measure 
\begin{equation}
    \di \mu^{\beta,V}_{N,\text{beta}} = \frac{1}{Z^\text{beta}_N(\beta,V)}\prod_{j=1}^{N-1} b_j^{\beta(N-j)-1}\mathbbm{1}_{b_j > 0} e^{ - \Tr{V(T)}}\di\ba\di\bb\,,
\end{equation}
with $\ba=(a_1,\ldots,a_N)$ and $\bb=(b_1,\ldots,b_{N-1})$. Then, the eigenvalues of $T$ are distributed according to the $\beta$-ensemble \eqref{def:real_beta_ens} with potential $V$, the case $V=\frac{x^2}{2}$ being the case where $T$ has independent entries described above.

The parameter $\beta$ is allowed to be $N$-dependent and can be interpreted as an inverse temperature. In this paper, we call the \textbf{high-temperature regime} the framework where $\beta=\frac{2\alpha}{N}$ for some fixed $\alpha>0$. We then define the high temperature version of $\mu^{\beta,V}_{N,\text{beta}}$ as 
\begin{equation}
    \label{def:HTbeta}
    \di\mu^{\alpha,V}_{N,\text{HT}} := \di\mu^{2\alpha/N,V}_{N,\text{beta}} = \frac{1}{Z^\text{HT}_N(\alpha,V)}\prod_{j=1}^{N-1} b_j^{2\alpha\frac{N-j}{N}-1}\mathbbm{1}_{b_j> 0} e^{-\Tr{V(T)}}\di\ba\di\bb\,.
\end{equation}

We can now sketch the link between the Toda chain distributed according to a GGE and the high temperature real $\beta$-ensemble. Consider the GGE \eqref{eq:Gibbs_todaIntro} with quadratic potential $V(x)=x^2/2$. Then, under $\mu^{\alpha, x^2/2}_{N,\text{Toda}}$, the entries of $L$ become independent (up to the symmetry) and their law is given by
\begin{itemize}
    \item the diagonal entries $L_{i,i}$ are standard Gaussian
    \item the off diagonal entries $L_{i,i+1}=L_{i+1,i}$, $1\leq i \leq N$ (where the indices are taken modulo $N$) are distributed as $\frac{1}{\sqrt{2}}\chi_{2\alpha}$.
\end{itemize}
With the choice $\beta=2\alpha/N$ in \eqref{def:real_beta_ens}, the entries near the top left corner of Dumitriu-Edelman representation (\textit{i.e.} the matrix $T$ distributed according to $\mu^{\alpha,x^2/2}_{N,\text{HT}}$) resemble the ones of the top left corner of $L$ when $(\ba,\bb)$ is distributed according to $\mu^{\alpha,x^2/2}_{N,\text{Toda}}$. One can then hope to pull this link between both models for more general choices of $V$. As already mentioned, this fundamental remark was made by Spohn in \cite{Spohn1} where the author was able to derive the limit of the empirical measure of eigenvalues of the Toda Lax matrix $L$ under the GGE in terms of the limiting measure of the $\beta$-ensemble in the high-temperature regime, both with polynomial potential $V$. Namely, denoting the empirical measure of eigenvalues of any $N\times N$ complex matrix $A$
\begin{equation}
    \label{emp.measure}
    \nu_N(A) = \frac{1}{N}\sum_{j=1}^N \delta_{\lambda_j(A)},
\end{equation}
Spohn argued that with $(\ba,\bb)$ distributed according to $\mu^{\alpha,V}_{N,\text{Toda}}$, the empirical measure $\nu_N(L)$ converges almost surely as $N\to \infty$ towards some deterministic measure $\nu^V_\alpha$, which is given in terms of the limiting empirical measure $\mu^V_\alpha$ of the high temperature $\beta$-ensemble, $\beta=\frac{2\alpha}{N}$, by
\begin{equation}
    \label{eq:link emp.measures Toda}
    \nu^V_\alpha = \partial_\alpha\left( \alpha \mu^V_\alpha \right)\,,
\end{equation}
where the last equality is understood in the weak sense, i.e. tested against bounded continuous functions. This was proved for the quadratic potential in \cite{mazzuca2021mean} and generalized to potentials with polynomial growth in \cite{GMToda}, where the authors established large deviations for the empirical measure of the Toda Lax matrix and related the associated rate function with the one of the $\beta$-ensemble at high temperature.

In the present paper, as an application of Theorem \ref{THM:FINAL_CLT_INTRO}, we push the comparison between both models further by linking the \textit{fluctuations} of the empirical measure of the Toda chain to the ones of the high temperature $\beta$-ensemble. 

Namely, we show that taking a \textbf{polynomial potential} $V$, considering
\begin{itemize}
    \item[\textbf{case 1:}] the periodic tridiagonal matrix $L$ distributed according to $\mu_{N,\text{Toda}}^{\alpha,V}$, eq. \eqref{eq:Gibbs_todaIntro}
    \item[\textbf{case 2:}] the tridiagonal matrix $T$ distributed according to $\mu^{\alpha,V}_{N,\text{HT}}$, eq. \eqref{def:HTbeta},
\end{itemize}
we have for some $\sigma^2_\text{Toda}(\alpha,V,g)>0$ and $\sigma^2_\text{HT}(\alpha,V,g)>0$ the following convergences in distribution: 

\textbf{in case 1},
$$ \sqrt{N}\left( \int_{\R} g\di \nu_N(L) - \lim_{N\to +\infty} \int_{\R} g\di \nu_N(L) \right) \xrightarrow{N\to\infty} \mathcal{N}(0,\sigma^2_\text{Toda}(\alpha,V,g))$$
where $\nu_N(L)$ \eqref{emp.measure} is the empirical measure of $L$.

\textbf{in case 2},
$$ \sqrt{N}\left( \int_{\R} g\di \nu_N(T) - \lim_{N\to +\infty} \int_{\R} g\di \nu_N(T) \right) \xrightarrow{N\to\infty} \mathcal{N}(0,\sigma^2_\text{HT}(\alpha,V,g))$$
where $\nu_N(T)$ \eqref{emp.measure} is the empirical measure of $T$.

Furthermore, we establish that both limiting variances are linked through 
\begin{equation}
    \label{eq:linkvariancesToda}
    \sigma^2_\text{Toda}(\alpha,V,g)=\partial_\alpha\left(\alpha \sigma^2_\text{HT}(\alpha,V,g)\right)\,,
\end{equation}
which is the analogue of \ref{eq:link emp.measures Toda} at the level of fluctuations of the empirical measures of eigenvalues.

This result is particularly relevant in the context of the so-called \textit{Generalized Hydrodynamics}, a recent physical theory that allows computing the correlation functions for classical integrable models, for an introduction to the subject see \cite{Doyon_notes,Spohn4}. According to this theory, one of the main ingredients to compute the correlation functions for the integrable model at hand is to be able to calculate the correlation functions for the conserved fields at time $0$. We are able to access these quantities as a consequence of Theorem \ref{THM:FINAL_CLT_INTRO}. We show how to do it in Section \ref{section:application}. We also mention the recent work \cite{mazzuca2023equilibrium}, where the authors made molecular dynamics simulations of the correlation functions of the Toda lattice, and they compared them with the predictions of linear Generalized Hydrodynamics, showing an astonishing agreement.

The relation between the Toda chain and the high-temperature regime of the real $\beta$-ensemble is not a unique case --- it is an instance of a broader picture linking two types of models:
\begin{itemize}
    \item[Type 1)] discrete space, continuous time integrable system with nearest neighbour interaction endowed with some GGE
    \item[Type 2)] high temperature $\beta$-ensemble on some given curve of the complex plane.
\end{itemize} 
Indeed, an analogue result to the one of Spohn for the Toda lattice holds true also for several other pairs integrable system/$\beta$-ensemble at high-temperature matrix model as it was described in \cite{mazzuca_int_sys,mazzuca2021generalized,mazzucamemin,Spohn4,Spohn3}.

In the present manuscript, we extend the aforementioned association. We establish a similar outcome analogous to the one delineated for the \textit{fluctuations} of the Toda lattice and the real $\beta$-ensemble at high temperature. This extension is applicable to other pairs of integrable systems and $\beta$-ensembles at high temperatures, as presented in Table \ref{tab:relationsIntro}.  The correspondence between the systems of this table was already described in \cite{mazzuca_int_sys} at the level of empirical measures under particular choices of potentials defining the GGEs and their associated $\beta$-ensemble. We introduce the systems at stake in Section \ref{section:application}.
\begin{table}[h]
    \centering

\begin{tabular}{|c|c|}

\hline
		 \textbf{Integrable System} (Type 1) & \textbf{$\beta$-ensemble at high-temperature } (Type 2)\\
		\hline
		
		      Toda lattice & Real  \\
		 \hline
		
		       Defocusing Ablowitz-Ladik lattice  & Circular \\
		\hline
	
		       Exponential Toda lattice	& 	Laguerre \\
		\hline
		
		  Defocusing Schur flow & Jacobi \\
		\hline
		
		     Volterra lattice & Antisymmetric \\
		\hline
	\end{tabular}
 \caption{Integrable systems and random matrix ensembles}
    \label{tab:relationsIntro}
\end{table} 

In the next section, we model the present picture by introducing two types of matrices and of distribution for their entries: the first type models are the Lax matrix of integrable systems, distributed according to a GGE, and the second type models the matrix representation of the associated $\beta$-ensemble in the high-temperature regime. We then state the main results of this paper.

\subsection{Main results of the paper}

We consider the following matrix models.

\begin{itemize}
    \item Type 1-\textit{i}) \textbf{Periodic Jacobi matrices}, which are periodic tridiagonal matrix of the form
    \begin{equation}
	\left( \begin{array}{ccccc}
	a_{1} & b_{1} & 0 & \ldots &  b_{N} \\
	b_{1} & a_{2} & b_{2} & \ddots & \vdots \\
	0 & b_{2} & a_{3} & \ddots & 0 \\
	\vdots & \ddots & \ddots & \ddots & b_{N-1} \\
	b_{N} & \ldots & 0 & b_{N-1} & a_{N} \\
	\end{array} \right) \,,
	\end{equation}
	for $\ba=(a_1,\ldots,a_N)\in \R^N$, $\bb=(b_1,\ldots,b_N)\in\R^N_+$.

   \item Type 1-\textit{ii}) \textbf{Periodic positive definite Jacobi matrices}, which are periodic tridiagonal matrix $L$ of the form

   \begin{equation}
       L = BB^\intercal\,,
   \end{equation}
   where $B\in \text{Mat}(\R,N)$ read 
    \begin{equation}
	B = \left( \begin{array}{ccccc}
	a_{1}  & b_{1} & 0 & \ldots &  0 \\
	   0 & a_{2} & b_{2} & \ddots & \vdots \\
	0 & 0 & a_{3} & \ddots & 0 \\
	\vdots & \ddots & \ddots & \ddots & b_{N-1} \\
	b_{N} & \ldots & 0 & 0 & a_{N} \\
	\end{array} \right) \,,
	\end{equation}
	for $\ba=(a_1,\ldots,a_N)\in \R^N_+$, $\bb=(b_1,\ldots,b_N)\in\R^N_+$.

        \item Type 1-\textit{iii}) \textbf{Antisymmetric tridiagonal periodic matrices}: 

               \begin{equation}
	 \left( \begin{array}{ccccc}
	0 & a_{1} & 0 & \ldots &  -a_{N} \\
	-a_{1} & 0& a_{2} & \ddots & \vdots \\
	0 & -a_{2} &0 & \ddots & 0 \\
	\vdots & \ddots & \ddots & \ddots & a_{N-1} \\
	a_{N} & \ldots & 0 & -a_{N-1} &0 \\
	\end{array} \right) \,,
	\end{equation}
	for $\ba=(a_1,\ldots,a_N)\in \R^N_+$.
	\item Type 1-\textit{iv}) \textbf{Periodic CMV (after Cantero, Moral and Velazquez) matrices}, which are $2N\times2N$ unitary matrices given by 
	\begin{equation}
		\cE  = \cL \cM\,,
	\end{equation}
	where we define $\cL$ and $\cM$ in the following way. Let $\ba = (a_1,\ldots a_{2N})$ be complex numbers of the unit disk $\D$. Define the  $2\times2$ unitary matrix  $\Xi_j$ 		
	\begin{equation}
	\label{def:Xi}
		\Xi_j = \begin{pmatrix}
			\wo a_j & \rho_j \\
			\rho_j & -a_j
		\end{pmatrix}\, ,\quad j=1,\dots, 2N\, ,\quad \rho_j = \sqrt{1-|a_j|^2}\,.
	\end{equation}
	Then, $\cL$ and $\cM$ are the $2N\times 2N$ matrices

	\begin{equation}
		\cL = \begin{pmatrix}
			\Xi_{1} \\
			& \Xi_3 \\
			&& \ddots \\
			&&&\Xi_{2N-1}
		\end{pmatrix} \, ,\qquad
		\cM= \begin{pmatrix}
			-a_{2N}&&&&& \rho_{2N} \\
			& \Xi_2 \\
			&&& \ddots \\
			&&&&\Xi_{2N-2}\\
			\rho_{2N} &&&&& \wo a_{2N}
		\end{pmatrix}\,.
	\end{equation}
	The matrix $\cE$ is a unitary pentadiagonal periodic matrix.
	\item Type 1-\textit{v}) \textbf{Two diagonals periodic matrices}  given by \, \\
	\begin{align}
	\begin{pmatrix}
	0&a_1&0&\cdots&\tikz[remember picture]\node[inner sep=0pt] (b) {$b_{N-r+1}$};&0&0&0\\
	0&0&a_2&\cdots&0&b_{N-r+2}&0&0\\
	\vdots&\ddots&\ddots&\ddots&\ddots&\ddots&\ddots&\\
	0&0&0&\cdots&a_{r-1}&0&0&b_N\\
	b_1&0&\cdots&\cdots&0&a_r&0&\tikz[remember picture]\node[inner sep=0pt] (a) {0};\\
	0&b_2&0\cdots&&\ddots&\ddots&\ddots&\\
	\vdots&\ddots&\ddots&\ddots&\ddots&0&0&a_{N-1}\\
	a_N&0&\cdots&b_{N-r}&\cdots&0&0&0\\
	\end{pmatrix}
		  \begin{tikzpicture}[overlay, remember picture]
		\draw[stealth-] (a.east)++(1,0) -- node[above] {$r+1$ row }++ (1,0);
		\draw[stealth-] (b.north)++(0,0.1) -- node[right] {$N-r$ column }++ (0,0.5);
	\end{tikzpicture}
	\end{align}
	Where $\ba,\bb \in \R^N_+$. In applications, we consider either $a_1=a_2=\ldots=a_N=1$ or $b_1=b_2=\ldots=b_N=1$.
\end{itemize}

We also consider the  non-periodic counterparts of the previous matrices. More specifically:

\begin{itemize}
    \item Type 2-\textit{i}) \textbf{Jacobi matrices}, which are symmetric tridiagonal matrices
    \begin{equation}
	    \begin{pmatrix}
	    a_1 & b_1 & & & \\
			b_1 & a_2 & b_2 & &\\
			& b_2 & \ddots & \ddots &\\
			&& \ddots & \ddots &  b_{N-1}  \\
			 &&&  b_{N-1} &a_N
	    \end{pmatrix},
	   \end{equation}
	   where $\ba \in \R^N$ and $\bb \in \R^{N-1}_+$.

       \item Type 2-\textit{ii}) \textbf{Positive definite Jacobi matrices}, tridiagonal matrix $T$ of the form

   \begin{equation}
       T = BB^\intercal\,,
   \end{equation}
   where $B\in \text{Mat}(\R,N)$ read 
    \begin{equation}
	B = \left( \begin{array}{ccccc}
	a_{1}  & b_{1} & 0 & \ldots &  0 \\
	   0 & a_{2} & b_{2} & \ddots & \vdots \\
	0 & 0 & a_{3} & \ddots & 0 \\
	\vdots & \ddots & \ddots & \ddots & b_{N-1} \\
	0 & \ldots & 0 & 0 & a_{N} \\
	\end{array} \right) \,,
	\end{equation}
	for $\ba=(a_1,\ldots,a_N)\in \R^N_+$, $\bb=(b_1,\ldots,b_{N-1})\in\R^N_+$.

    \item Type 2-\textit{iii}) \textbf{Tridiagonal Antisymmetric matrices}: 

               \begin{equation}
	\left( \begin{array}{ccccc}
	0 & a_{1} & 0 &  &   \\
	-a_{1} & 0& a_{2} & \ddots &  \\
	0 & -a_{2} &0 & \ddots & 0 \\
	    & \ddots & \ddots & \ddots & a_{N-1} \\
	    &  & 0 & -a_{N-1} &0 \\
	\end{array} \right) \,,
	\end{equation}
	for $\ba\in \R^{N-1}_+$.
	   \item Type 2-\textit{iv}) \textbf{CMV matrices}, $2N\times 2N$ unitary matrices of the form
	   \begin{equation}
	\fE = \fL\fM\,,
	\end{equation}
	where
	   \begin{equation}
		\fL = \mbox{diag}\left(\Xi_{0},\Xi_2,\Xi_4, \ldots,\Xi_{2N}\right)\qquad \text{and} \qquad \fM = \mbox{diag}\left(\Xi_1,\Xi_3,\Xi_{5} \ldots,\Xi_{2N-1}\right)\,,
	\end{equation}
	and the blocks $\Xi_j$, $j=1,\dots, 2N-1$ are defined as 	
	\begin{equation}
		\Xi_j = \begin{pmatrix}
			\wo a_j & \rho_j \\
			\rho_j & -a_j
		\end{pmatrix}\, ,\quad j=1,\dots, 2N-1\, ,\quad \rho_j = \sqrt{1-|a_j|^2}\,.
	\end{equation}
    for $\mathbf{a}=(a_1,\ldots,a_{2N})\in \mathbb{D}^{2N}$,  while $\Xi_{0} = (1)$ and $\Xi_{2N} = (\wo a_{2N})$ are $1\times 1$ matrices. 
    The matrix $\cE$ is a pentadiagonal matrix.
\end{itemize}

The periodic matrices (Type 1) that we consider are the \textit{Lax matrices} of some integrable models. These are particular dynamical systems that are Liouville integrable, and their integrability is proved by obtaining a Lax pair $(L,A)$ \cite{Lax1968} representation of the model, meaning that the equations of motions for each of these systems are equivalent to the following linear system of ordinary differential equations for some  matrices $L,A$

\begin{equation}
	\dot{L} \equiv\frac{\di L}{\di t} = [L;A] = LA-AL\,.
\end{equation}

This formulation is useful since it implies that $\{\Tr{L^k}\}_{k=1}^N$ are a system of independent constants of motion $\left(\dfrac{\di}{\di t}\Tr{L^k}=0)\right)$ for the system at hand, so the system is integrable in  classical sense. We call these quantities \textit{conserved fields}. For a comprehensive introduction to classical integrable systems, we refer to \cite{Babelon}.

We introduce each integrable model with its matrix representation in Section \ref{section:application}, but we anticipate that
\begin{itemize}
    \item the \textbf{Toda lattice} \cite{Toda1970} has as Lax matrix a periodic Jacobi matrix (type $1$-\textit{i}),
    \item the \textbf{Exponential Toda lattice} \cite{mazzuca_int_sys} has a positive definite periodic Jacobi matrix (type $1$-\textit{ii}),
    \item the \textbf{Volterra lattice} \cite{mazzuca_int_sys} has an antisymmetric periodic one (type $1$-\textit{iii}), 
    \item the \textbf{Ablowitz-Ladik lattice} \cite{Ablowitz1974} and the \textbf{Schur flow}  \cite{Golinskii} have a periodic CMV one (type $1$-\textit{iv}),
    \item the family of \textbf{Itoh--Narita--Bogoyavleskii} \cite{Bogoyavlensky1988} lattices have a bidiagonal periodic one (type $1$-\textit{v}).
\end{itemize}
We endow the \textbf{periodic matrices of type $\boldsymbol{1}$} with the so-called \textit{Generalized Gibbs Ensemble} of the corresponding dynamical system. An important property of these measures is that they are invariant with respect to the dynamics of the corresponding integrable system. These Generalized Gibbs Ensembles are probability measures $X^N$, where $X$ is a convex subset of $\R^d$, $d=1$ or $2$ (with $\C$ identified with $\R^2$) depending on the considered model, of the form
\begin{equation}
    \label{type1}
    \mu^{(1)}_N = \frac{1}{Z^{(1)}_N(\alpha,\Re P)}\left(\prod_{j=1}^N F(x_j,\alpha)\right) e^{-\Tr{\Re P(L)}}\di \bx\,.
\end{equation}
Here $L$ is a matrix of type 1, and the integration variable $\bx=(x_1,\ldots,x_N)$ is to be understood as the entries of $L$. In our applications, we will specify the set $X$ to be
\begin{itemize}
    \item model 1-\textit{i}) $\bx = (\ba,\bb)\in \R^N\times \R_+^N$, $X=\R\times \R_+$
    \item model 1-\textit{ii}) $\bx = (\ba,\bb)\in \R_+^{2N}$, $X=\R_+\times \R_+$
    \item model 1-\textit{iii}) $\bx = \ba\in \R_+^N$, $X=\R_+$
    \item model 1-\textit{iv}) $\bx=\ba \in \D^N$, $X=\D$
    \item model 1-\textit{v}) $\bx=\ba$ (resp. $\bx=\bb$) $\in \R_+^N$ if $b_1=\ldots=b_N=1$ (resp. if $a_1=\ldots=a_N=1$), $X=\R_+$.
\end{itemize}
For our applications, we consider polynomial potentials: $\Re P$ stands for the real part of a polynomial $P\in \C[X]$. In all models except type 1-\textit{iv}, we take directly $P\in \R[X]$. Model 1-\textit{iv} is the only one with complex entries, so in this case we consider just $\Re P$. Furthermore, $P$ has to be chosen in such a way that the measure \eqref{type1} is normalizable, \textit{i.e.} such that the associated partition function converges. Note that with this choice of potential, the trace $\Tr{\Re P(L)}= \Re \Tr{P(L)}$ is simply expressed either in terms of the diagonal entries of powers of $L$, or in terms of the eigenvalues of the matrix $L$.

We recover the GGE of the Toda lattice \eqref{eq:Gibbs_todaIntro} with polynomial potential taking for $x=(a,b)\in \R\times \R_+$ 
$$ F(x,\alpha)=F(a,b,\alpha)=b^{2\alpha -1}\,.$$
In this case, taking $P\in \R[X]$ polynomial with even degree $n\geq 2$ and positive leading coefficient ensures the convergence of the partition function $Z^{\text{Toda}}_N(\alpha,P)$ of the model.
\ \\ \ \\

The \textbf{non-periodic matrices of type $\boldsymbol{2}$} are related to the classical $\beta$-ensembles. More specifically
\begin{itemize}
    \item The \textbf{Real $\beta$-ensemble} has a Jacobi matrix representation \cite{dued1} (type 2-\textit{i})
    \item The \textbf{Laguerre $\beta$-ensemble} has a positive definite Jacobi matrix representation \cite{dued1} (type 2-\textit{ii})
    \item The \textbf{Antisymmetric $\beta$-ensemble} has a tridiagonal antisymmetric matrix representation \cite{DF10} (type 2-\textit{iii})
    \item The \textbf{Circular} and \textbf{Jacobi $\beta$-ensemble} have a CMV matrix representation \cite{Killip2004} (type 2-\textit{iv})
\end{itemize}
We introduce properly those ensembles in Section \ref{section:application}. As in \eqref{def:HTbeta} for the case of the real $\beta$-ensemble, we consider these ensembles in the high-temperature regime, meaning that the parameter $\beta$ scales with the matrix size $N$ as $\beta=\frac{2\alpha}{N}$, for some $\alpha\in\R_+$ independent of $N$.

In all the cases we consider, the joint density of the entries of the matrix representation of the $\beta$-ensembles with polynomial potential can be expressed as

\begin{equation}
    \label{type2}
    \mu^{(2)}_N =  \frac{1}{Z^{(2)}_N(\alpha,\Re P)}\left(\prod_{j=1}^{N-1} F\left(x_j,\alpha\left(1-\frac{j}{N}\right)\right)\right)R(x_N) e^{-\Tr{\Re P(T)}}\di \bx\,.
\end{equation}
As before, this probability measure is the entry-wise law of $T$;  furthermore, in our applications, the space is the same as in the corresponding type 1 model:

\begin{itemize}
    \item model 2-\textit{i}) $\bx=(\ba,\bb) \in \R^N\times \R_+^{N}$
    \item model 2-\textit{ii}) $\bx=(\ba,\bb)\in \R_+^{2N}$
    \item model 2-\textit{iii} $\bx = \ba \in \R^{N}_{+}$
    \item model 2-\textit{iv}) $\bx=\ba \in \D^{2N}$
\end{itemize}
In the previous equation, $R$ is a \textit{distribution} defining a probability measure supported on a subset of $X$: we denote, with an abuse of notation, $R(x)\di x$ the integration against $R$. Thus, for some $X'\subset X$, $\int_{X'} R(x)\di x =1$ and for any $\varphi \in C^{\infty}_c(X)$ with support disjoint from $X'$, $\int_X \varphi(x) R(x) \di x=0$.

In this notation, the type 2 models are defined on the same space as the type 1. As an example, considering  $x=(a,b)\in \R\times \R_+$ 
$$ F(x,\alpha) = F(a,b,\alpha) =  b^{2\alpha-1}\,, \qquad R(x) = R(a,b) = \delta_0(b)\,,$$
we recover the tridiagonal representation of the real $\beta$-ensemble in the high-temperature regime \eqref{def:HTbeta} with polynomial potential. As for the matrices of type 1, $P\in \R[X]$ except for the model 2-\textit{iv} associated with CMV matrices, which is the only model of type 2 with complex entries. We remark that the function $F$ specified for the real $\beta$-ensemble at high temperature is the same as the one for the Toda lattice \eqref{eq:Gibbs_todaIntro}.

As we already mentioned, we focus on the fluctuations (or \textit{linear statistics}) around the equilibrium measure of these general models, where we choose the functions $F$, $P$ and the distribution $R$ so that the partition functions
\begin{equation}
    \begin{split}
        & Z^{(1)}_N(\alpha,\Re P) = \int_{X^N} \left(\prod_{j=1}^N F(x_j,\alpha)\right) e^{-\Tr{\Re P(L)}}\di \bx\, \\
        & Z^{(2)}_N(\alpha,\Re P) = \int_{X^N} \left(\prod_{j=1}^{N-1}  F\left(x_j,\alpha\left(1-\frac{j}{N}\right)\right)\right) R(x_N)  e^{-\Tr{\Re P(T)}}\di \bx 
    \end{split}
\end{equation}
are finite for all $N$. Recall that $X$ is a convex subset of $\R^d$. In our applications, $d=1$ or $2$, and $\C$ is identified with $\R^2$.

Specifically, we focus on analysing the fluctuations

\begin{equation}
\label{eq:fluc}
    \int_\C Q\di\nu_N - \lim_{N\to\infty} \int_\C Q\di\nu_N \,,
\end{equation}
for \textbf{polynomial test functions} $Q$, where $\nu_N$ is the empirical measure of eigenvalues of the considered matrix $M$:

\begin{equation}
\label{eq:empirical_measure}
    \nu_N = \frac{1}{N}\sum_{j=1}^N\delta_{\lambda_j(M)}\,,
\end{equation}
with $\lambda_j(M)$ denoting the eigenvalues of $M$. In this equation, $\delta_x$ denotes the Dirac delta function centred at $x$. Note that the existence of the limit 
\begin{equation}
    \label{limitEmpirical}
    \lim_{N\to \infty} \int_\C Q\di\nu_N 
\end{equation}
is not straightforward: we establish it in Theorem \ref{THM:FINAL_CLT_INTRO}, see Remark \ref{remark moments}.

Our paper utilizes a transfer operator technique to analyze the random variable \eqref{eq:fluc} for polynomial potentials $\Re P$.

The study of statistical spectral properties of Lax matrices was initiated by Spohn in \cite{Spohn1}, establishing the link between the Lax matrix of the Toda chain distributed according to any of the GGEs,
and the real $\beta$-ensemble at high temperature, as sketched in Section \ref{subsection Toda}. See also\cite{mazzuca2021mean}. In \cite{Spohn1,Spohn2,Spohn3}, the author studies space-time correlation functions for the Toda lattice and, leveraging the theory of Generalized Hydrodynamics \cite{Doyon_notes}, he argues that they have a ballistic behaviour, characterized by symmetrically located peaks that travel in opposite directions at a constant speed and decay as $t^{-1}$ when $t$ approaches infinity. To achieve this result, Spohn computed the density of states of the Toda Lax matrix by connecting the Generalized Gibbs Ensemble of the Toda lattice to the real $\beta$-ensemble in the high-temperature regime \cite{Allez2012}. Later, one of the authors and T. Grava connected the Generalized Gibbs ensemble for the Ablowitz-Ladik lattice and the Schur flow to the Circular $\beta$-ensemble and the Jacobi $\beta$-ensemble, respectively, in the high-temperature regime \cite{Hardy2021,Forrester2021}, as did H. Spohn independently \cite{mazzuca2021generalized,Spohn3}. For all these models, a large deviation principle for their mean density of states was developed in \cite{mazzucamemin,GMToda}. Furthermore, in \cite{mazzuca_int_sys}, the authors established connections between the classical Gibbs ensemble for the Exponential Toda lattice and the Volterra lattice with the Laguerre ensemble \cite{dued1} and the Antisymmetric $\beta$-ensemble \cite{DF10}, respectively.

In addition to integrable systems, our study also involves the classical $\beta$-ensembles. Specifically, we investigate the random variable \eqref{eq:fluc} for these ensembles in the high-temperature regime. The study of these quantities was initiated by Johansson in \cite{johansson}, where the author derived a central limit theorem (CLT) for the Gaussian Unitary Ensemble with a polynomial potential. This result was subsequently generalized to other models and different values of $\beta$ in \cite{BoG2,Shc1,Dumitriu2012}, and more recently in \cite{Bourgade2022}, where the authors also obtained a rigidity result for the eigenvalues of the $\beta$-ensembles. See also \cite{lambert2021mesoscopic} for a refinement of the CLT for the circular $\beta$-ensemble at a mesoscopic scale. Other relevant works include \cite{Nakano2018}, where a CLT was derived for the $\beta$-ensemble in the high-temperature regime with a quadratic potential, \cite{Hardy2021}, where a CLT for the Circular $\beta$-ensemble at high temperatures was obtained using a normal approximation method, and the recent paper \cite{memin2023}, where the authors derived a CLT for the real $\beta$-ensemble in the high-temperature regime for general confining potentials. We also mention the work \cite{Serfaty2021}, where Coulomb gases in dimensions $\td \geq 2$ were investigated, and local laws were studied for various temperature regimes. Finally, we mention the recent works \cite{johansson2021multivariate} and \cite{Courteaut2021,Courteaut2022}, where the authors established super-exponential bounds for the convergence of moments of the CUE, COE, and CSE to a Gaussian vector.

\paragraph{Statement of the results.}
We come to precise statements of the main results of the present paper. To do that, we introduce the concept of \textit{circular functions}.      
\begin{definition}
\label{def:circular}
            Let $N\in \N$ and $s\geq 1$. Write $N = sM + \ell$ where $0\leq \ell < s$. We say that a function $Y: X^N \to \C$ is \textbf{$s$-circular} if there exist two functions $y\,:\, X^{2s} \to \C$, and  $\widetilde{y} \, : \, X^{2s + \ell} \to \C$ such that

            \begin{equation}
                Y(\bx) = \sum_{j=1}^{M-1} y(\bx_j, \bx_{j+1}) + \widetilde{y}(\bx_1,\bx_M,x_{sM + 1}, \ldots, x_{N})\,,
            \end{equation}
            where $\bx_j = (x_{(j-1)s+1}, \ldots, x_{js})$ and $s$ is independent of $N$ for $N$ big enough. We call $y$ a \textbf{seed} of $Y$, $\widetilde{y}$ a \textbf{weed} of $Y$, and $s$ the \textbf{circular index} of $Y$. When we are not interested in the specific circular index of the function $Y$, we just say that $Y$ is circular.
\end{definition}

\begin{remark}
    We notice that
    \begin{itemize}
        \item The seed of a function is not unique, indeed $\hat{y}(\bx_j,\bx_{j+1},\bx_{j+2},\bx_{j+3}) =  y(\bx_j, \bx_{j+1}) +  y(\bx_{j+1}, \bx_{j+2}) + y(\bx_{j+2}, \bx_{j+3})$ is another seed;
        \item the sum of a $s_1$-circular function $Y_1$ and a $s_2$-circular function $Y_2$ is a $s_3$-circular function $Y_3$, where $s_3$ is the smallest common multiple between $s_1,s_2$. In this situation, we say that $y_1, y_2$ are \textbf{compatible seeds} of $Y_1,Y_2$ if 
    \begin{itemize}
        \item $y_1+y_2$ is a seed of $Y_3$;
        \item the circular indexes of $y_1,y_2$ is $s_3$;
        \item $y_1$ is a seed of $Y_1$, and $y_2$ is a seed of $Y_2$
    \end{itemize}
    \end{itemize}
\end{remark}

Furthermore, the following lemma holds true

\begin{lemma}
\label{LEM:CIRCULAR_TRACE}
    Consider any type 1-2 matrix $M$. Then for any polynomial $P\in \C[X]$, $\Tr{P(M)}$ is circular.
\end{lemma}
We detail the proof of this Lemma in Appendix \ref{app:circ}. 

Given this definition and this lemma, we can proceed by stating our main result. First, we state the assumptions we make for the definition of our measures of interest $\mu^{(1)}_N$ \eqref{type1} and $\mu^{(2)}_N$ \eqref{type2}. \textbf{Let $M$ be a matrix of type $1$ or of type $2$, endowed with the distribution $\mu^{(1)}_N$ or $\mu^{(2)}_N$} on $X^N$, $X\subset \R^d$, $d=1$ or $2$.

Throughout the paper, we will consider polynomial potentials with the following properties, depending on the considered model.
\begin{assumption}
    \label{polynomial}
    \item Models 1-\textit{i} and 2-\textit{i}: $P\in \R[X]$, with degree $d\geq 2$ and positive leading coefficient
    \item Models 1-\textit{ii} and 2-\textit{ii}: $P\in \R[X]$, with degree $d\geq 1$ and positive leading coefficient
    \item Models 1-\textit{iii} and 2-\textit{iii}: $P\in \R[X]$, with $P$ of the form $P(x)=(-1)^{d} a_d x^{2d} + Q(x)$ with $a_d > 0$ and $Q\in \R[X]$ with degree smaller than $2d-1$. 
    \item Models 1-\textit{iv} and 2-\textit{iv}: $P\in \C[X]$
    \item Model 1-\textit{v}: $P\in \R[X]$ with degree $j(r+1)$ for some $j\geq 1$ (with $r$ appearing in the definition of the matrix model) and positive leading coefficient.
\end{assumption}

\begin{assumption}
\label{general_assumptions}
To prove our main result, we assume the following:
\begin{enumerate}
 \item[HP 1.] $X$ is a convex subset of $\R^d$, $d=1$ or $2$ ($\C$ being identified with $\R^2$).
 
    \item[HP 2.] $F(x,\alpha) \, : \, X \times (0,\infty)\to \R_+$ is such that for any $\alpha>0$ $F(\cdot,\alpha) \in C^1(X)$, and for any $x\in X, F(x,\cdot)\in C^\infty((0,+\infty))$ ;
        \item[HP 3.] The seed $W$ and weed $\wt W$ of $\Tr{\Re P(M)}$ are bounded away from $-\infty$
  \item[HP 4.] for all $\alpha\in (0,\infty)$, $F(\cdot,\alpha) \in L^{1}(X)$, and $[c,d]\subset (0,+\infty)$ one can find $g_{c,d}\in L^2(X)$ such that for all $c\leq \alpha\leq d$, $\left| \sqrt{F(\cdot,\alpha)}\right|,\ \left|\partial_\alpha \sqrt{F(\cdot,\alpha)}\right|\leq g_{c,d}$; moreover,  there exist a $\tc\in\N$ and $\varepsilon_0$ such that for all $\varepsilon<\varepsilon_0$ there exists a compact set $\mathcal{O}_\varepsilon\subseteq X$ and $d_1,d_2,d_3>0$, depending on $\varepsilon$, such that
    \begin{itemize}
        \item $||F(\cdot,\alpha)||_1 = d_1 \alpha^{-\tc}(1 + o(1))$ as $\alpha\to 0$
        \item  $||F(\cdot,\alpha)||_{1,\mathcal{O}_\varepsilon} = \int_{\mathcal{O}_\varepsilon }\vert F(x,\alpha) \vert \di x = d_2\alpha^{-\tc}(1 + o(1) )$ as $\alpha\to 0$
        \item  $||F(\cdot,\alpha)||_{1,\mathcal{O}^c_\varepsilon} \leq d_3$ here $\mathcal{O}_\varepsilon^c = X \setminus \mathcal{O}_\varepsilon$
        \item There exists a continuous function $w\, :\, X^k \to \R$ such that for $\bx,\by \in \mathcal{O}_\varepsilon$
        \begin{equation}
            W(\bx,\by) = w(\bx) + w(\by) + o(\varepsilon)\,.
        \end{equation} 
        \item If $\varepsilon_1<\varepsilon_2$ then $\mathcal{O}_{\varepsilon_1}\subseteq \mathcal{O}_{\varepsilon_2}$
    \end{itemize}

    \item[HP 5.] $R$ is a distribution defining a probability measure on a subset of $X$.
\end{enumerate}
Here $L^p(X,\cB)$ is the usual $L^p$ space.
\end{assumption}

\begin{remark}
The definition of circular function and seeds was introduced in this context in \cite{Giorgilli2014}. 
\end{remark}

Let $s\in \N$ and $t\in \R$. Then, for $P$, $F$ satisfying the previous assumptions, both
$$ \left( \prod_{j=1}^N F(x_j,\alpha) \right)e^{-\mathrm{Tr} \big(\Re P(L) + it\Re L^s\big)} $$
and 
$$ \left( \prod_{j=1}^{N-1} F\left(x_j,\alpha\left(1-\frac{j}{N}\right)\right) \right)e^{-\mathrm{Tr} \big(\Re P(L) + it\Re L^s} \big)R(x_N)$$
are integrable on $X^N$. We denote by 
$$ Z^{(1)}_{N}(\alpha, \Re P + it \Re z^s)\text{ and }Z^{(2)}_{N}(\alpha, \Re P + it \Re z^s) $$
the associated integrals.

Under these assumptions, we  prove our main theorem: 

\begin{theorem}
\label{THM:FINAL_CLT_INTRO}
Under Hypotheses \ref{polynomial}-\ref{general_assumptions}. Consider $\mu^{(1)}_{N},\, \mu^{(2)}_{N}$ \eqref{type1}-\eqref{type2},  and let $s\in \N$,
 $W,h$ be compatible seeds of $\Tr{\Re P (L)}$ and $\Tr{\Re L^s}$, and assume that $|h|^a e^{-W}$ is bounded, for $1\leq a \leq 3$. Then, there exist two continuous functions 

\begin{align}
    &A(x)\,:\, \R_+ \longrightarrow \R\,, \\
    &\sigma^2(x)\, : \, \R_+ \longrightarrow \R_+\,,
\end{align}
such that under $\mu^{(1)}_{N}$ \eqref{type1}
$$\left(\Tr{\Re L^s} -NA(\alpha)\right)/\sqrt{N}$$ 
converges to a Gaussian distribution $\cN(0,\sigma^2(\alpha))$ as $N$ tends to infinity. Similarly, under $\mu^{(2)}_{N}$ \eqref{type2},  
$$\left(\Tr{\Re T^s} -N\int_0^1 A(\alpha x)\di x\right)/\sqrt{N}$$
converges to a Gaussian distribution $\cN(0,\int_0^1\sigma^2(\alpha x)\di x)$ as $N$ tends to infinity.
Furthermore, defining the \textit{free energies} $\cF^{(1)}(\alpha,\Re P ),\cF^{(2)}(\alpha,\Re P )$ as

\begin{align}
    \label{eq:free_periodic}
    &\cF^{(1)}(\alpha,\Re P ) = -\lim_{N\to\infty} \frac{1}{N}\ln\left(Z_N^{(1)}(\alpha,\Re P )\right)\,, \\
    \label{eq:free_fix}
    & \cF^{(2)}(\alpha,\Re P) = - \lim_{N\to\infty} \frac{1}{N}\ln\left(Z_N^{(2)}(\alpha,\Re P)\right)\,,
\end{align}
then
\begin{multicols}{2}
\begin{enumerate}
    \item[i.] $ \cF^{(1)}(\alpha,\Re P ) =  \partial_\alpha \left( \alpha \cF^{(2)}(\alpha,\Re P)\right)$
    \item[ii.] $A(\alpha) = i\partial_{t}\cF^{(1)}(\alpha,\Re P +it \Re z^s)_{\vert_{t=0}}$
    \item[iii.] $\int_0^1 A(\alpha x)\di x = i\partial_{t}\cF^{(2)}(\alpha,\Re P +it\Re z^s)_{\vert_{t=0}}$ 
    \item[iv.] $\sigma^2(\alpha) = \partial^2_{t}\cF^{(1)}(\alpha,\Re P +it\Re z^s)_{\vert_{t=0}}$
    \item[v.] $\int_0^1\sigma^2(\alpha x)\di x= \partial^2_{t}\cF^{(2)}(\alpha,\Re P +it\Re z^s)_{\vert_{t=0}}$
\end{enumerate}
\end{multicols}

\end{theorem}

\begin{remark}\label{remark moments}Let us notice that 
\begin{itemize}
    \item The previous theorem also holds for $\Im z^s$ in place of  $\Re z^s$. 
    \item As a consequence of Theorem \ref{THM:FINAL_CLT_INTRO}, we deduce the convergence in law of both 
    $$\frac{1}{N}\Re\Tr{L^s}\text{ and }\frac{1}{N}\Im\Tr{M^s}$$
    as $N\to +\infty$, for any $L$ among the considered models. Furthermore, the limits are deterministic, therefore we deduce the convergence of the moments, as $N\to +\infty$
    $$ \lim_{N\to +\infty}\frac{1}{N}\Tr{L^s}=\int_\C z^s \di \nu_N(z)\,.$$
    In the case of the Ablowitz-Ladik lattice (matrix of type 1-\textit{iii}) or of the circular $\beta$-ensemble in the high-temperature regime (matrix of type 2-\textit{iii}), this is a consequence of the following facts
    \begin{itemize}
        \item[1)] The empirical measures are supported on the circle (as empirical measures of unitary matrices)
        \item[2)]  They converge (see for \cite{mazzucamemin} for Ablowitz-Ladik, \cite{Hardy2021} for circular $\beta$-ensemble for example)
        \item[3)] Polynomial functions are bounded on the circle.
    \end{itemize}
    In general, even in cases where we know the empirical measure to converge, the fact that the previous limit exists is not trivial, since polynomial functions are not bounded on non-compact sets.
\end{itemize}
\end{remark}

\begin{remark}
\label{rem:Toda_proof}
    Let us check here that Theorem \ref{THM:FINAL_CLT_INTRO} applies to the case of the Toda lattice, and its random matrix counterpart, \textit{i.e.} the real $\beta$-ensemble in the high-temperature regime. The points that need to be discussed are HP 3., HP 4. and that $|h|^s e^{-W}$ is bounded for $1\leq s \leq 3$.

    Let $L$ be a type $1$---\textit{i} matrix with coordinates $\bx=(\ba,\bb)\in X_0^N=(\R\times \R_+)^N$ and consider $\mu^{(1)}_N$ with $P$ satisfying Hypotheses \ref{polynomial} and 
    $$F(x,\alpha)=F(a,b,\alpha)=b^{2\alpha-1}\,.$$ 
    The latter function is not integrable on $X_0$, so it doesn't satisfy Hypotheses \ref{general_assumptions} HP3. Instead, we take advantage of the following fact: there are some $c>0$, $C\in \R$ such that for all $x\in \R$,
    $$P(x) \geq C + cx^2.$$ 
    Thus, $\Tr{P(L)} \geq CN + c\Tr{L^2} = CN + c\sum_{j=1}^N a_j^2 + 2b_j^2$, therefore setting 
    $$ \wt F(a,b,\alpha) = b^{2\alpha -1} e^{-\frac{c}{2}(a^2+ 2b^2)}\,,$$
    the measure $\mu^{(1)}_N$ reads
    $$ \frac{1}{Z^{(1)}_N(\alpha,P)} \left(\prod_{j=1}^N \wt F(x_j,\alpha)\right) e^{-(\Tr{P(L)}-\frac{c}{2}\Tr{L^2})}\di \bx\,.$$
    The modified potential $P(x)- \frac{c}{2}x^2$ still fulfils Hypotheses \ref{polynomial}, and the function $\wt F$ then satisfies the prescribed assumptions.
    Indeed, $\partial_\alpha\sqrt{F(a,b,\alpha)}=\ln(b)b^{\alpha-\frac{1}{2}}e^{-\frac{c}{4}(a^2+2b^2)}$ thus we can easily find, for any compact $K\subset \R_+^*$, a function $g_K\in L^2(\R_+\times \R)$ such that for any $\alpha\in K$, and for any $(a,b)\in \R_+\times \R$
    $$ \sqrt{F(a,b,\alpha)},\ \left|\partial_\alpha \sqrt{F(a,b,\alpha)} \right| \leq g_K(a,b)\,.$$
    The remaining points of HP3. hold for $\tc=1$. Indeed, $\| F(\cdot,\alpha) \|_{L^1(X)}= c^{-\alpha}2^{1-\alpha}\Gamma(\alpha)$, that as $\alpha\to 0$ is of order $\alpha^{-1}$, and with $\mathcal{O}_\varepsilon=[0,\varepsilon]^2$, we have
    $$ \int_{\mathcal{O}_\varepsilon} b^{2\alpha-1}e^{-\frac{c}{2}(a^2+2b^2)}\di a\di b\geq \int_0^\varepsilon e^{-\frac{ca^2}{2}}\di a  \int_0^\varepsilon b^{2\alpha -1}e^{-c}\di b =: \frac{d_\varepsilon}{\alpha}\left( 1 + o(1)\right)\,,$$
    as $\alpha\to 0$; moreover, $d_\varepsilon\to0$ as $\varepsilon\to 0$ and it is independent of $\alpha$.

    HP4. and the boundedness of $|h|^s e^{-W}$, $1\leq s \leq 3$ is a consequence of Theorem \ref{thm:struct}, see Corollary \ref{cor:struct}.
\end{remark}

\begin{remark}
\label{rem:operatorRegularity}
In the central part of the proof, we introduce a family of operators that act on $L^2(X^k)$ as follows:

$$ \cL_{\alpha,t}f(\by)=\int_{X^k}f(\bx)\prod_{j=1}^k \sqrt{F(x_j,\alpha)F(y_j,\alpha)}e^{-(W+ith)(\by,\bx)}\di \bx \di\by\,.$$
Based on our assumptions, each $\cL_{\alpha,t}$ is a Hilbert-Schmidt operator, meaning that the kernel
$$(\bx,\by)\mapsto\prod_{j=1}^k \sqrt{F(x_j,\alpha)F(y_j,\alpha)}e^{-(W+ith)(\by,\bx)}$$ 
belongs to $L^2(X^k\times X^k)$. Consequently, $\cL_{\alpha,t}$ is a compact operator. Then, we have the following:

\begin{itemize}
   
    \item Under Hypotheses \ref{general_assumptions} and the assumption that $|h|^a \exp(-W)$ is bounded for $1\leq a \leq 3$, we show in Lemma \ref{lem:differentiability} that $(\alpha,t)\mapsto \cL_{\alpha,t}$ is a regular, operator valued function: specifically, $(\alpha,t)\mapsto \cL_{\alpha,t}$ is differentiable with respect to $\alpha$ and three times differentiable with respect to $t$; 
     \item As a consequence of Jentzsh's theorem \cite[Theorem 137.4]{ZaanenBook} and an approximation argument, $\cL_{\alpha,0}$ possesses a simple dominant eigenvalue with a spectral gap uniform in $\alpha$.
\end{itemize}
These two properties ensure that for sufficiently small $t$, the operator $\cL_{\alpha,t}$ possesses a simple dominant eigenvalue $\wt \lambda(\alpha,t)$, and the function $t\mapsto \wt\lambda(\alpha,t)$ is three times differentiable. The function $\wt \lambda(\alpha,t)$ is the key to our analysis, because of the equality, established in Lemma \ref{lem:free_energy}
$$ \cF^{(1)}(\alpha,\Re P + it \Re z^s) = -\frac{1}{k}\ln\left( \wt \lambda(\alpha,t) \right)\,,$$
where $\cF^{(1)}$ is the free energy defined in \eqref{eq:free_periodic}.

Note that we only require the existence of a second derivative with respect to $t$ in the proof of the main theorem, while the existence of a third derivative is used in the proof of the Berry-Esseen bound, as stated in Theorem \ref{thm:BerryEsseen}. The differentiability with respect to $\alpha$ is utilized in the proof of Theorem \ref{thm:hard_asymptotic}.
\end{remark}

As a by-product, we can also compute the so-called \textit{susceptibility matrix} $\mathcal{C}$ for integrable models. This is the matrix of the space-correlation functions of the conserved fields, i.e.

\begin{equation}
\label{def:susceptibility}
    \mathcal{C}_{m,n} =\lim_{N\to\infty} \frac{1}{N}\left( \E_1\left[\Tr{\Re{L^m}}\Tr{\Re L^n}\right] - \E_1\left[\Tr{\Re L^m}\right]\E_1\left[\Tr{\Re L^n}\right]\right)\,,
\end{equation}
where $L$ is the Lax matrix of the integrable system at hand and the mean values are taken according to the corresponding Generalized Gibbs ensemble. The computation of such quantities is relevant to obtain the decay of the correlation functions for these integrable systems, as shown by Spohn in \cite{Spohn2}. In particular, we can prove the following:

\begin{theorem}
\label{THM:SUSCEPTIBILITY}
Under the same hypotheses as Theorem \ref{THM:FINAL_CLT_INTRO}. Consider $\mu^{(1)}_{N}$ \eqref{type1} and define the \textit{free energy} $\cF^{(1)}(\alpha,\Re P)$ as in \eqref{eq:free_periodic}, then

\begin{equation}
\label{eq:susceptibility}
    \mathcal{C}_{m,n} = \partial_{t_1}\partial_{t_2}\left(\cF^{(1)}(\alpha, \Re P +it_1\Re z^m+it_2 \Re z^n)\right)_{\vert_{t_1=t_2=0}}\,.
\end{equation}
\end{theorem}

\begin{remark}
    In view of Theorem \ref{THM:FINAL_CLT_INTRO}, we can rewrite  \eqref{eq:susceptibility} as 
    \begin{equation}
\label{eq:susceptibility_beta}
    \mathcal{C}_{m,n} = \partial_\alpha\partial_{t_1}\partial_{t_2}\left(\alpha \left(\cF^{(2)}(\alpha, \Re P +it_1\Re z^m+it_2 \Re z^n)\right)\right)_{\vert_{t_1=t_2=0}}\,.
\end{equation}
In our context, the previous equality implies that we can compute the susceptibility matrix of the integrable systems that we are considering in terms of just the free energy of the corresponding classical $\beta$-ensemble in the high-temperature regime. 
\end{remark}

Furthermore, considering the type 1 measures \eqref{type1},  we investigate the space-correlations for \textit{local} functions, defined as follows.
Given a differentiable function $F \colon X^N \to \C$, we define its {\em support} as the set 
        \begin{equation}
        	\label{def:supp}
        	{\rm supp }\, F := \left\{ \ell \in \{ 1, \ldots, N\} \colon \ \ \   \frac{\partial F}{\partial x_\ell}\notequiv 0 \right\}
        \end{equation}
        and its  {\em diameter}  as 
        \begin{equation}
        	\label{diameter}
        	{\rm diam} \left({\rm supp }\, F\right) := \sup_{i, j \in {\rm supp}\, F} \td_{N}(i,j) , 
        \end{equation}
        where $\td$ is the  {\em periodic distance}  
        \begin{equation}
        	\label{p.dist}
        	\td (i,j) := \min \left( |i-j|, \ N - |i-j| \right) . 
        \end{equation}
        Note that $0\leq \td (i,j) \leq N/2$.
        
        We say that a function $F$ is \textit{local} if ${\rm diam} \left({\rm supp }\, F\right)$ is uniformly bounded in $N$, i.e. there exists a constant $\tc\in \N$ such that ${\rm diam} \left({\rm supp }\, F\right) \leq \tc$, and $\tc$ is independent of $N$. We remark that these definitions were already introduced in \cite{Grava2020}.

Given this definition, we can prove the following lemma

\begin{theorem}[Decay of correlations]
\label{thm:decay}
    Consider a matrix $L$ of type 1 with polynomial potential $\Re P$ such that $\mu_N^{(1)}$ \eqref{type1} is well-defined, let $I,J : X^k \to \R $ two local functions such that   $\int_{X^k\times X^k}\left|I(\bx)\prod_{i=1}^k F(x_i,\alpha)e^{-W(\bx,\by)}\right\vert^2\di \bx \di \by < \infty $, where $W$ is one seed of $\Tr{\Re P(L)}$, and analogously for $J(\bx)$. Write $N=kM+\ell$, and let $j\in \{1,\ldots,M\}$. Then there exists some $0<\mu<1$ such that
    $$ \E_{1}\left[I(\bx_1)J(\bx_j)\right]-\E_{1}\left[I(\bx_1)\right]\E_{1}\left[J(\bx_j)\right]=O(\mu^{M-j}+\mu^j)\,.$$

\end{theorem}
This result establishes the exponential decay of space correlations between two local functions acting on distinct segments of the chain, with the decay rate determined by the distance between the corresponding particle sets. Thanks to this Theorem, we rigorously justify the assumption made by H. Spohn regarding the decay of space correlations between the \textit{local conserved fields} and their currents in Section \ref{sec:toda_gaussian}, where we apply the aforementioned result. We remark that the same reasoning can be applied to other integrable systems considered in our study.

Finally, we also prove a Berry-Essen type bound for the aforementioned integrable models. Specifically, we prove the following
\begin{theorem}
\label{thm:BerryEsseen_intro}
Fix $s\in \N$. Under Hypotheses \ref{general_assumptions}. Consider a matrix $L$ of type 1 with polynomial potential $\Re P$ such that $\mu_N^{(1)}$ \eqref{type1} is well-defined, and let $W,h$ be the compatible seed of $\Tr{\Re P(L)}$ and $\Tr{\Re L^s}$. Assume that $|h|^a e^{-W}$ is bounded for $1\leq a \leq 3$,   so that
\begin{equation}
\label{eq:ration_BE}
    \mathbb{E}_1\left[e^{-it\Tr{\Re L^s}}\right] = \frac{Z_{kM+\ell}^{(1)}(\alpha, \Re P+it\Re x^s)}{Z_{kM+\ell}^{(1)}(\alpha,\Re P)}\,.
\end{equation}
Then, there exists $A\in \R$, $\sigma,C >0$ such that if $Y\sim\cN(0,\sigma^2)$ we have for any interval $J$ of the real line
\begin{equation}
    \left| \mathbb{P}\left( \left[\Tr{L^s} - N A \right]/\sqrt{N} \in J\right) - \mathbb{P}\left(Y\in J\right)\right| \leq \frac{C}{\sqrt{N}}\,.
\end{equation}
\end{theorem}

\textbf{The paper is organized as follows.} In Section \ref{section:application} we deduce a central limit theorem for the systems in table \ref{tab:relationsIntro} by applying Theorem \ref{THM:FINAL_CLT_INTRO}. In Section \ref{sec:transfer_operator}, we use the transfer operator method to explicitly compute the partition functions for the models in table \ref{tab:relationsIntro}. In Section \ref{section:generalTheory}, we prove some generalization of Nagaev--Guivarc'h Theorem, which is the main probability tool that allows us to prove our main result. In Section \ref{section:proof_of_the_result}, we prove the main results of our paper, i.e Theorems \ref{THM:FINAL_CLT_INTRO}-\ref{THM:SUSCEPTIBILITY}. In Section \ref{section:technical}, we prove the technical result that we stated in Section \ref{sec:transfer_operator} and we prove Theorems \ref{thm:decay}-\ref{thm:BerryEsseen_intro}. Finally, in Section \ref{section:conclusion}, we provide concluding remarks and outline potential avenues for future advancements in this field. We defer some of the technical results of the paper in Appendix \ref{app:circ}-\ref{App:proof differentiability}.

\section{Application}
\label{section:application}

In this section, we show how to apply Theorem \ref{THM:FINAL_CLT_INTRO} to obtain a CLT for the integrable systems and for the classical $\beta$-ensembles in the high-temperature regime in Table \ref{tab:relationsIntro}. 

Specifically, we are able to prove that all the integrable systems in Table \ref{tab:relationsIntro} in the periodic case have a Generalized Gibbs ensemble of the form $\mu^{(1)}_{N}$ \eqref{type1}. Meanwhile, the $\beta$-ensembles at high-temperature are characterized by a probability  distribution of the form $\mu^{(2)}_{N}$ \eqref{type2}. In this way, we proved a further connection between the theory of integrable systems and Random Matrix Theory. Indeed, in view of Theorem \ref{THM:FINAL_CLT_INTRO} and Theorem \ref{THM:SUSCEPTIBILITY}, for any integrable system in the previous table, we can relate its free energy, moments, variances and covariances  with the corresponding quantities of the random matrix model on the same line.   Moreover, in the final part of this section, we consider the family of INB lattices that do not have a known $\beta$-ensemble counterpart. Despite that, we are still able to derive the existence of a polynomial central limit theorem. Finally, applying the decay of correlations of Theorem \ref{thm:decay}, we are able to compute the currents of the Toda chain in terms of the limiting variance in the CLT, justifying a computation by Spohn in \cite{Spohn4}.

\subsection{The Toda lattice, and the real $\beta$-ensemble at high-temperature}
\label{sec:toda_gaussian}

\paragraph{The Toda lattice.}

In this subsection, we apply the previous results to deduce the central limit theorem for the empirical measure in the setup of the Toda chain. We then apply the decay of correlations established in Theorem \ref{thm:decay} to deduce a formulation of the \textit{currents} of the Toda chain in terms of the limiting variance in the CLT, justifying a computation by Spohn, \cite{Spohn4}. Finally, we relate the limiting variance and free energy of the Toda chain with the ones of the  $\beta$-ensemble at high-temperature.

To begin with, our aim is to obtain a central limit theorem for the conserved fields when the initial data is sampled according to the Generalized Gibbs ensemble of the Toda chain, which we rewrite here for convenience. \textbf{To ease the notations, we denote it by $\mu_T$ instead of} $\mu_{N,\text{Toda}}^{\alpha,P}$ \textbf{in the sequel, and we denote the associated partition function by $Z_N^T(\alpha,P)$.}
\begin{equation}\label{eq:Gibbs_toda}
	\di \mu_{T} :=    \frac{1}{Z_{N}^T(\alpha,P)} \  \prod_{j=1}^N b_j^{2\alpha -1}\mathbbm{1}_{b_j>0}e^{-\Tr{P(L)} } \di \ba \,  \di \bb \,,
	\end{equation}
where $P$ is a polynomial of even degree with positive leading coefficient, and $\alpha > 0$. Also recall that $L$ is a periodic Jacobi matrix, \textit{i.e.} a matrix of type 1-\textit{i}.

We want to apply Theorem \ref{THM:FINAL_CLT_INTRO} to this model. To do that, we need some preparation. First, we notice that the Lax matrix of the periodic Toda lattice is a circular-like matrix, thus we can apply Lemma \ref{LEM:CIRCULAR_TRACE} to ensure that the traces of powers of $L$ are circular functions. Furthermore, from the following Theorem, whose proof can be found in \cite{Grava2020}, one can deduce that the seeds of even powers of the trace are bounded from below, see Remark \ref{rem: toda local potential}.
	
		\begin{theorem}[cf. Theorem 3.1 \cite{Grava2020}]
	\label{thm:struct}
		For any {$1 \leq m \leq N-1$}, consider the matrix $L$ given by \eqref{eq:L_toda}. One has 
\begin{equation}
		\label{Jm.sum}
		\Tr{L^m}=  \sum_{j=1}^N h_{j}^{(m)} \, ,
		\end{equation}		
		where $ h_{j}^{(m)}:= [L^m]_{jj}$ is given explicitly by
		\begin{equation}\label{eq:general_super_motzkin}
 h_{j}^{(m)} (\bp,\br)= \sum_{(\bn,\bq)\in \cA^{(m)}}  \rho^{(m)}(\bn,\bq) \prod_{i = -\wt m }^{\wt m-1} a_{j+1}^{q_i} \prod_{i = -\wt m+1 }^{\wt m -1} b_{j+i}^{2n_i} \, ,
		\end{equation}
where  it is understood $a_j \equiv a_{j \mod N +1}, \, b_j \equiv b_{j \mod N+1}$ and  $\cA^{(m)}$ is the set

		\begin{equation}
		\label{cAm}
			\begin{split}
			\cA^{(m)} := \Big\{(\bn,\bq) \in \N^{\mathbb{Z}}_0 \times \N^{\mathbb{Z}}_0 \ \colon \ \ \ 
			& \sum_{i= -\wt m }^{\wt m-1} \left(2n_i + q_i\right) = m  , \\
& \forall i \geq 0, \ \ \ n_i = 0 \Rightarrow n_{i+1} = q_{i+1} = 0,  \,  
\\
& \forall i < 0, \ \ \ n_{i+1} = 0 \Rightarrow n_{i}= q_i = 0  
			\Big\}.
			\end{split}
		\end{equation}
		The quantity  $\wt m := \floor{m/2}$, $\N_0=\N\cup\{0\}$
		and $\rho^{(m)}(\bn, \bq) \in \N $ is 
		given by 
		\begin{align}
		\label{rhom}
								\rho^{(m)}(\bn,\bq) := &\binom{n_{-1} + n_0 + q_0}{q_0}\binom{n_{-1} + n_0}{n_0}
	\prod_{i=-\wt m \atop i \neq -1}^{ \wt m -1}\binom{n_i + n_{i+1} +q_{i+1} -1}{q_{i+1}}\binom{n_i + n_{i+1} -1}{n_{i+1}} \, .
\end{align}	

	\end{theorem}
The combination of these results leads to the following corollary
	
	\begin{corollary}
	\label{cor:struct}
	Fix $m\in \N$, and consider the matrix $L$ \eqref{eq:L_toda}. Then for $N$ big enough, there exists some $k=k(m)\in \N$, such that $\Tr{L^m}$ is a $k$-circular function, with seed $V$ and weed $\widetilde V$. Furthermore, both $V,\widetilde V$ are bounded from below, away from $-\infty$. Furthermore, for $m$ even, for any local polynomial $u(\ba,\bb)$ with diameter $\td < k$ the function $|u|^se^{-V},\, s\in \mathbb{N}$ is bounded from above.
	\end{corollary}
	
	\begin{remark}
 \label{rem: toda local potential}
    We notice that one can prove that the function $V,\widetilde V$ are lower bounded away from $-\infty$ using the explicit formula in Theorem \ref{thm:struct} or applying the properties of super-Motzkin paths used for the proof of the Theorem in \cite{Grava2020}.
	\end{remark}

    We apply the previous Corollary to the Gibbs measure of the Toda lattice \eqref{eq:Gibbs_toda}, so it can be written as
    
	\begin{equation}
	\begin{split}
	    	\di \mu_{T} &=    \frac{1}{Z_{N}^T(\beta,P)}   \prod_{j=1}^N b_j^{2\alpha -1}\mathbbm{1}_{b_j>0}\exp\Big(-\sum_{j=1}^M V(\ba_j,\bb_j,\ba_{j+1},\bb_{j+1})\\ & -  \wt V(\ba_M,a_{kM+1},\ldots, a_{kM+\ell},\ba_1, \bb_M, b_{kM+1},\ldots, b_{kM+\ell},\bb_1) \Big) \di \ba \,  \di \bb \,.
	\end{split}
	\end{equation}
    Now we can apply Theorem \ref{THM:FINAL_CLT_INTRO}, see Remark \ref{rem:Toda_proof}, and we deduce the following
	
		\begin{corollary}[CLT for the Toda lattice]
	Consider the Lax matrix L \eqref{eq:L_toda} of the Toda lattice distributed according to the Generalized Gibbs Ensemble \eqref{eq:Gibbs_toda}, and assume that $P(x)$ is a polynomial of even degree with positive leading order coefficient. Then, defining the Free energy $\cF_{T}(\alpha,P)$ as
	
	\begin{equation}
 \label{freeEnergyToda}
	    \cF_{T}(\alpha,P) = -\lim_{N\to\infty}\frac{1}{N}\ln(Z^{T}_N(\alpha,P))\,,
	\end{equation}
	for all $j\in\N$ fixed, we have the following weak limit 
	
	\begin{equation}
	    \lim_{N\to\infty}  \frac{\Tr{L^j} - \meanval{\Tr{L^j}}  }{\sqrt{N}} \rightharpoonup \cN(0,\sigma^2)\,.
	\end{equation}

        Where 
        \begin{equation}
            \meanval{\Tr{L^j}} = iN \partial_t \cF_{T}(\alpha,P+itx^j)_{\vert_{t=0}}\,, \quad \sigma^2 = \vert\partial^2_t \cF_{T}(\alpha,P+itx^j)_{\vert_{t=0}}\vert\,.
        \end{equation}
	
	\end{corollary}

    Moreover, we can also apply Theorem \ref{THM:SUSCEPTIBILITY} to compute the correlation between the conserved fields at time zero, indeed the theorem immediately implies that
	   
	   \begin{equation}
    \label{eq:susc_in_free}
	       \lim_{N\to\infty} \frac{\meanval{\Tr{L^j}\Tr{L^n}} - \meanval{\Tr{L^j}}\meanval{\Tr{L^n}}}{N} = \partial_{t_1}\partial_{t_2} \cF_{T}(\alpha,P+it_1x^j +it_2x^n)_{\vert_{t_1,t_2=0}}\,,
	   \end{equation}
	   where the mean value is taken with respect to the Gibbs measure of the Toda lattice \eqref{eq:Gibbs_toda}. We notice that this implies that we can compute the susceptibility matrix of the Toda lattice \eqref{eq:susceptibility} in terms of the derivative of the Free energy.

\subsubsection{The Toda chain's currents}

Since the conserved fields are local quantities, they must satisfy a local conservation law. Following the notation of \cite{spohn2021hydrodynamic}, we define 

\begin{equation}
    Q^{[n],N}_j = L^n_{j,j} \,,
\end{equation}
where $L\in \textrm{Mat}(N,\R)$ is \eqref{eq:L_toda}. We can easily compute the evolution equation for such quantities as

\begin{equation}
    \frac{d}{d t} Q^{[n],N}_j = (BL^n-L^nB) = b_{j-1}L^n_{j,j-1} - b_{j}L^n_{j+1,j}\,.
\end{equation}
Defining $J^{[n],N}_j = b_{j-1}L^n_{j,j-1} $, we have 
$$ \frac{d}{d t} Q^{[n],N}_j = J^{[n],N}_j - J^{[n],N}_{j+1} $$
and we say that $J^{[n],N}_j$ is the current of the local conserved field $Q^{[n],N}_j$. In particular, defining the matrix $L^{\downarrow}$ as
\begin{equation}
    L^\downarrow_{i,j} = \begin{cases}
        L_{i,j} \quad \text{if} \; j < i \; \text{or}\; i=1,j=N \\
        0 \quad \text{otherwise}
    \end{cases}\,
\end{equation}
we can recast the previous definition as 
\begin{equation}
    J^{[n],N}_j =  (L^nL^\downarrow)_{j,j}\,.
\end{equation}
We notice that both $Q^{[n],N}_j$ and $J^{[n],N}_j$ depend on time, and we adopt the convention that if not explicitly written the evaluation is at time $0$. Furthermore, we define

\begin{equation}
	\label{eq:total}
    Q^{[n],N} = \sum_{j=1}^N Q^{[n],N}_j\,, \qquad  J^{[n],N} = \sum_{j=1}^N J^{[n],N}_j\,,
\end{equation}
and we refer to $Q^{[n],N}$ as the $n^{th}$-conserved field, and to $J^{[n],N}$ as the $n^{th}$-total current.

The evaluation of the expected values of both the currents $J^{[n],N}_j$ and the total current  $J^{[n],N}$ according to the Generalized Gibbs ensemble \eqref{eq:Gibbs_toda} is one of the crucial steps to apply the theory of Generalized Hydrodynamics to the Toda lattice, as it is explained in \cite{spohn2021hydrodynamic}. In this manuscript, the author used some heuristic arguments to explicitly derive the expression for these quantities, here we rigorously justify his argument applying Theorem \ref{thm:decay}. 

First, we extend the definition of $Q_j^{[n],N}$ and $J_j^{[n],N}$ for $n=0$, setting $Q_j^{[0],N} = r_j$, and $J_j^{[0],N} = -p_j = -Q_j^{[1],N}$. We notice that $\sum_{j=1}^N J_j^{[0],N} = -\sum_{j=1}^N Q_j^{[1],N}$ is still a conserved field. We are now in position to show how to compute the limiting Toda average current

\begin{equation}
    \lim_{N\to \infty}\frac{1}{N}\mathbb{E}\left[ J^{[n],N}\right]\,,
\end{equation}
in terms of the susceptibility matrix \eqref{eq:susceptibility} of the Toda chain, so in particular of the derivative of the Free energy \eqref{eq:susc_in_free}. Indeed, we prove the following:

\begin{lemma}
    Consider the Lax matrix L \eqref{eq:L_toda} of the Toda lattice distributed according to the Generalized Gibbs Ensemble \eqref{eq:Gibbs_toda}, and assume that $P(x)$ is a polynomial of even degree with positive leading order coefficient. Then, for any fixed $n\in \N$, and $\alpha \in \R_+$ defining the total currents $J^{[n],N}$  as in \eqref{eq:total} we have the following equality
    
    \begin{equation}
    	\lim_{N\to\infty}\frac{1}{N}\E\left[J^{[n],N}\right] = \int_0^{\alpha}\partial_{t_1}\partial_{t_2} \cF_{T}(s,P+it_1x +it_2x^n)_{\vert_{t_1,t_2=0}}\di s.
    \end{equation}
\end{lemma}

\begin{proof}
    
In view of the cyclic structure of the measure $\mu_T$ and of the total current, we deduce that

\begin{equation}
    \frac{1}{N}\meanval{J^{[n],N}} = \meanval{J^{[n],N}_1}\,.
\end{equation}
Furthermore, for any fixed $N$, we deduce, by differentiating with respect to the parameter $\alpha$, the following equality

\begin{equation}
\label{eq:mean_cur_der}
    \partial_\alpha  \meanval{J^{[n],N}_1} =-\mathrm{Cov}\left(J^{[n],N}_1\ ;\ \sum_{j=1}^N r_j \right)=-\sum_{j=1}^N \mathrm{Cov}\left(J^{[n],N}_1\ ;\ Q^{[0],N}_j \right)\,,
\end{equation}

where we defined for any functions $f,g\in L^2(X^N,\mu_T)$  

\begin{equation}
    \textrm{Cov}(f\,;\,g) = \meanval{fg} - \meanval{f}\meanval{g}\,.
\end{equation}

We show now that the following limits coincide

\begin{equation}
\label{eq:switch}
    \lim_{N\to\infty}\sum_{j=1}^N \mathrm{Cov}\left(J^{[n],N}_1\, ;\, Q^{[0],N}_j \right) = \lim_{N\to\infty}\sum_{j=1}^N \mathrm{Cov}\left(J^{[0],N}_1\, ;\, Q^{[n],N}_{N-j+2} \right)\,.
\end{equation}

Indeed, for any $n,m\geq 0$ and $t\in \R$
\begin{equation}
\label{eq:derivative_trick}
    \begin{split}
           \mathrm{Cov}\left( J_{j+1}^{[n],N}(t)-J_j^{[n],N}(t)\ ;\ Q^{[m],N}_1(0) \right) &= -\frac{d}{dt}\mathrm{Cov}\left(Q^{[n],N}_j(t)\ ;\ Q^{[m],N}_1(0) \right) \\
    &= -\frac{d}{dt}\mathrm{Cov}\left( Q^{[n],N}_1(0)\ ;\ Q^{[m],N}_{N-j+2}(-t)\right)\\
    &=\mathrm{Cov}\left( Q^{[n],N}_1(0)\ ;\ J^{[m],N}_{N-j+3}(-t)-J^{[m],N}_{N-j+2}(-t)\right),
    \end{split}
\end{equation}

where we used that $s\mapsto Q^{[n],N}_j(t+s)Q^{[m],N}_1(s)$ is constant in law under the Toda dynamic, and the periodicity of the matrix $L$ \eqref{eq:L_toda}. Denoting the difference operator $\partial_j f(j) = f(j+1)-f(j)$, equation \eqref{eq:derivative_trick} shows that 

\begin{equation}
    \partial_j\left(\mathrm{Cov}\left(J_j^{[n],N}(t)\ ;\ Q^{[m],N}_1(0)\right)-\mathrm{Cov}\left(Q^{[n],N}_1(0) ;\ J^{[m],N}_{N-j+2}(-t)\right) \right) =0
\end{equation}
Evaluating the previous expression at $t=0$, we deduce that there is some constant $c_N$, independent of $j$, such that
\begin{equation}
    \mathrm{Cov}\left(J_j^{[n],N}\ ;\ Q^{[m],N}_1\right) - \mathrm{Cov}\left(Q^{[n],N}_1 ;\ J^{[m],N}_{N-j+2}\right) = c_N\,.
\end{equation}
Furthermore, since both $Q^{[n],N}_j$, and $J^{[m],N}_j$ are local quantities, in view of Theorem \ref{thm:decay}, we deduce that $\lim_{N\to\infty} Nc_N =0$.
So, evaluating the previous expression for $m=0$, we deduce \eqref{eq:switch}. Thus, in the large $N$ limit, we can recast \eqref{eq:mean_cur_der} as
\begin{equation}
    \lim_{N\to\infty}\partial_\alpha \E\left[J^{[n],N}_1\right]=-\lim_{N\to\infty} \sum_{j=1}^N \mathrm{Cov}\left(J^{[0],N}_1\ ;\ Q_j^{[n],N}\right) = \lim_{N\to\infty} \sum_{j=1}^N \mathrm{Cov}\left(Q^{[1],N}_1\ ;\ Q_j^{[n],N}\right) \,.
\end{equation}
Moreover, in view of the periodicity properties ot the conserved fields and \eqref{eq:susc_in_free}

\begin{equation}
	 \lim_{N\to\infty}\partial_\alpha \E\left[J^{[n],N}_1\right] =  \lim_{N\to\infty}\frac{1}{N} \mathrm{Cov}\left(Q^{[1],N}\ ;\ Q^{[n],N}\right) =  \partial_{t_1}\partial_{t_2} \cF_{T}(\alpha  ,P+it_1x+ it_2x^n)_{\vert_{t_1,t_2=0}}\,.
\end{equation}
Noticing that $\lim_{\alpha\to0}\E\left[J^{[n],N}_1\right]=0$, and that we can always uniformly bound $\E\left[J^{[n],N}_1\right]$ by a constant independent of $N$, the previous equation implies that
\begin{equation}
 \lim_{N\to\infty}\E\left[J^{[n],N}_1\right] = \int_0^{\alpha}\partial_{t_1}\partial_{t_2} \cF_{T}(s,P+it_1x +it_2x^n)_{\vert_{t_1,t_2=0}}\di s.
\end{equation}
So, we conclude.
\end{proof}

\paragraph{The real $\beta$-ensemble in the high-temperature regime.}
	
	We consider the real $\beta$-ensemble introduced in Section \ref{subsection Toda} as the probability measure on the real line 
 \begin{equation}
    \label{eq:gaussianBeta}
    \di \mathbf{P}^{\beta,V}_N(\lambda_1,\ldots, \lambda_N)=\frac{1}{\mathbf{Z}^{\beta,V}_N}\prod_{1\leq i < j \leq N}|\lambda_i-\lambda_j|^\beta e^{-\sum_{j=1}^N V(\lambda_j)}\di \boldsymbol{\lambda}\,,
\end{equation}
where $\beta>0$ and $V$ is a continuous function such that the partition function
	$$ \mathbf{Z}^{\beta,V}_N = \int_{\R^N}\prod_{i<j}\vert\lambda_j -\lambda_i\vert^\beta e^{-\sum_{j=1}^NV(\lambda_j)} \di \boldsymbol{\lambda} $$
	is finite. This is the case if $V$ grows to infinity fast enough, namely if for some $\beta'>\max(1,\beta)$,
\begin{equation}
\label{condition de base sur le potentiel}
    \liminf_{|x|\to \infty}\frac{V(x)}{N\beta'\ln|x|}>1\,,
\end{equation}
see \cite[equation (2.6.2)]{AGZ}.
     
     Dumitriu and Edelman showed in \cite{dued} that the $\beta$-ensemble admits a tridiagonal representation
	\begin{equation}
	\label{eq:gaussian_matrix}
	    H = \begin{pmatrix}
	    a_1 & b_1 & & &0 \\
			b_1 & a_2 & b_2 & &\\
			& \ddots & \ddots & \ddots &\\
			&& \ddots & \ddots &  b_{N-1}  \\
			0 &&&  b_{N-1} &a_N
	    \end{pmatrix}\,,
	\end{equation}
	where the entries of the matrix are distributed according to the following probability measure
	\begin{equation}
	    \label{eq:gaussian_measure}
	    	    \di \mu^{\beta,V}_{N,\text{beta}} = \frac{1}{Z^\text{beta}_N(\beta,V)} \prod_{j=1}^{N-1} b_j^{\beta(N-j)-1}\mathbbm{1}_{b_j\geq 0} \exp\left( - \Tr{V(H)}\right)\di\ba\di\bb\,.
	\end{equation}
	Then, the eigenvalues of $H$ are distributed according to $\di \mathbf{P}^{\beta,V}_N$ \eqref{eq:gaussianBeta}. An important example is the case $V(x)=x^2/2$ for which we recover the classical \textit{Gaussian $\beta$-ensemble} (also called Hermite $\beta$-ensemble), see \cite[Section 2.5]{AGZ}, and the distribution $\mu^{\beta,V}_{N,\text{beta}}$ factorizes in the following way: the entries of $H$ can be seen to be independent (modulo the symmetry of the matrix), Gaussian $\mathcal{N}(0,1)$ on the diagonal, and the law of the off-diagonal elements is given by renormalized chi variables
	$$b_j \sim \frac{1}{\sqrt{2}}\chi_{(N-j)\beta}\,,$$
	where the variable $X$ is $\chi_{\kappa}$-distributed if its law is given by the density function
	$$ f(x)=\frac{x^{\kappa-1}e^{-x^2/2}}{2^{\kappa/2-1}\Gamma(\kappa/2)}\,.$$
	We are interested in the so-called \textit{high-temperature regime} for this model, specifically, we are interested in the infinite size $N$ limit, in such a way that $\beta = \frac{2\alpha}{N}$ for some $\alpha>0$. In this regime the probability distribution \eqref{eq:gaussian_measure} becomes $\mu_{N,\text{HT}}^{\alpha,V}$ \eqref{def:HTbeta}, which we recall:
		\begin{equation}
	    \label{eq:gaussian_measure_ht}
	    	    \di \mu_{N,\text{HT}}^{\alpha,V} = \frac{1}{Z^\text{HT}_N(\beta,V)} \prod_{j=1}^{N-1} b_j^{2\alpha\left(1-\frac{j}{N}\right)-1}\mathbbm{1}_{b_j\geq 0} \exp\left( - \Tr{V(H)}\right)\di\ba\di\bb\,.
	\end{equation}
This regime has drawn a lot of attention from the random matrix and statistical physics communities lately. Introducing the empirical measure by 
$$\di \nu_N = \frac{1}{N}\sum_{i=1}^N \delta_{\lambda_i}\,,$$
this model was first considered in \cite{Allez2012}, where the authors were able to compute the limiting empirical measure for this model when $V(x) = x^2/2$. Recently, Garcia-Zelada showed in \cite{Garcia} that under a general choice of $V$, the sequence of empirical measures satisfies a large deviation principle with strictly convex rate function, ensuring the convergence of $\nu_N$. Although the limiting measure is not explicit, its density $\rho^V_\alpha$ satisfies for almost every $x$ the nonlinear equation 
$$ V(x)-2\alpha\int_\R \log|x-y|\rho^V_\alpha(y)\di y + \log \rho^V_\alpha(x) = \lambda^V_\alpha$$
for some constant $\lambda^V_\alpha$, see \cite[Lemma 3.2]{GMToda} for example.

The fluctuations of the eigenvalues in the bulk and at the edge of a configuration were studied for example in \cite{benaych2015poisson,Nakano2018,Nakano2020,pakzad2018poisson,lambert2021poisson}. These fluctuations were shown to be described by Poisson statistics in this regime. With the choice $V(x)=x^2/2$, Nakano and Trinh proved in \cite{Nakano2018} a Central Limit theorem for this ensemble, namely y they proved that for smooth enough $f:\R\to \R$, the random variables
$$ \sqrt{N}\left(\int_\R f\di\nu_N - \int_\R f\rho^V_\alpha\di x \right)$$
converge towards a centred Gaussian random variable with variance depending both on $\alpha$ and $f$. In \cite{memin2023}, the authors showed this central limit theorem for general confining potentials and smooth enough, decaying at infinity test functions. In this paper, we consider the case where $V$ is a polynomial of even degree greater than $2$ as in Hypotheses \ref{polynomial}. We deduce a central limit theorem for polynomial test functions  from section \ref{section:generalTheory}.

Indeed, we notice that the matrix $H$ \eqref{eq:gaussian_matrix} is a Toeplitz-like matrix (see Appendix \ref{app:circ}), thus, in view of Lemma \ref{LEM:CIRCULAR_TRACE} and Corollary \ref{cor:struct}, we can apply Theorem \ref{THM:FINAL_CLT_INTRO} to the real $\beta$-ensemble in the high-temperature regime, thus we deduce that

	\begin{corollary}[CLT for real, high temperature $\beta$-ensemble]
	Consider the matrix representation \eqref{eq:gaussian_matrix} of the real $\beta$-ensemble in the high-temperature regime with potential $P$, polynomial of even degree with positive leading  order coefficient. Then, defining the Free energy $\cF_\mathrm{HT}(\alpha,P)$ as
	
	\begin{equation}
	    \cF_{\mathrm{HT}}(\alpha,P) = -\lim_{N\to\infty}\frac{1}{N}\ln(Z^\mathrm{HT}_N(\beta,P))\,,
	\end{equation}
	for all $j\in\N$ fixed, we have the following weak limit
	
	\begin{equation}
	    \lim_{N\to\infty}  \frac{\Tr{H^j} - \meanval{\Tr{H^j} }}{\sqrt{N}} \rightharpoonup \cN(0,\sigma^2)\,,
	\end{equation} 

    where 
    \begin{equation}
        \meanval{\Tr{H^j}} = iN\partial_t \cF_{\mathrm{HT}}(\alpha,P+itx^j)_{\vert_{t=0}} \,,\quad \sigma^2 = \vert\partial^2_t \cF_{\mathrm{HT}}(\alpha,P+itx^j)_{\vert_{t=0}}\vert\,.
    \end{equation}
	\end{corollary}
Thus, we obtained a central limit theorem for the real $\beta$-ensemble in the high-temperature regime with polynomial potential. 

Furthermore, all the hypotheses to apply the second part of our results are satisfied, so we deduce the following identities

\begin{equation}
\begin{split}
        &\partial_\alpha(\alpha \partial_t \cF_{\mathrm{HT}}(\alpha,P+itx^j)_{\vert_{t=0}}) = \partial_t \cF_{T}(\alpha,P+itx^j)_{\vert_{t=0}}\,, \\
        &\partial_\alpha(\alpha \partial^2_t \cF_{\mathrm{HT}}(\alpha,P+itx^j)_{\vert_{t=0}}) = \partial^2_t \cF_{T}(\alpha,P+itx^j)_{\vert_{t=0}}\,,
\end{split}
\end{equation}
where $\cF_{T}$ is the free energy of the Toda lattice \eqref{freeEnergyToda}.
So we can compute both the moments and their variances of the Toda lattice starting from the one of the real $\beta$-ensemble at high temperature.

\begin{remark}
    Applying the second part of Theorem \ref{THM:FINAL_CLT_INTRO}, we deduce the following equality valid for the currents of the Toda lattice:

    \begin{equation}
         \lim_{N\to\infty}\E\left[J^{[n],N}_1\right] = \int_0^{\alpha}\partial_{t_1}\partial_{t_2} \cF_{T}(s,P+it_1x +it_2x^n)_{\vert_{t_1,t_2=0}}\di s = \alpha\partial_{t_1}\partial_{t_2} \cF_{\mathrm{HT}}(\alpha,P+it_1x +it_2x^n)_{\vert_{t_1,t_2=0}}.
    \end{equation}
\end{remark}

\subsection{The exponential Toda lattice, and the Laguerre $\beta$-ensemble at high-temperature}
\label{sec:exptoda_laguerre}
In this subsection, we focus on the Exponential Toda lattice and its relation with the Laguerre $\beta$-ensemble in the high-temperature regime \cite{Forrester2021}. These two systems were considered in \cite{mazzuca_int_sys}. In this paper, the authors considered the classical Gibbs ensemble for the Exponential Toda lattice and were able to compute the density of states for this model connecting it to the Laguerre $\alpha$ ensemble \cite{mazzuca2021mean}, which is related to the classical $\beta$ one in the high-temperature regime. Here we consider both the Generalized Gibbs ensemble for the integrable lattice and the Laguerre $\beta$-ensemble at high-temperature with polynomial potential, and we obtain a CLT for both systems, furthermore, we connect the two in the same way as we did for the Toda lattice and the real $\beta$-ensemble.

\paragraph{The exponential Toda lattice.}
The exponential Toda lattice is the Hamiltonian system on  $\mathbb{R}^{2N}$  described by the Hamiltonian
	\begin{equation}   \label{HamExpToda}
		H_E({\mathbf p},{\mathbf q}) = \sum_{j=1}^N e^{-p_j} + \sum_{j=1}^{N}e^{q_j-q_{j+1}}\,, \quad p_j,q_j \in \R\,,
	\end{equation}
	with canonical Poisson bracket.
	Here, we consider periodic boundary conditions
	\begin{equation}
		q_{j+N}=q_{j} + \Omega, \quad p_{j+N} = p_j,\qquad \forall \, j\in \Z,
	\end{equation}
	and $\Omega\geq 0$ is an arbitrary constant. The equations of motion are given in Hamiltonian form as
	\begin{equation} \label{HamiltonianForm}
		\begin{split}
			& \dot{q}_j=\frac{\partial H_E}{\partial p_j} = - e^{-p_j}\,,\\  
			& \dot{p}_j=-\frac{\partial H_E}{\partial q_j} = e^{q_{j-1} -  q_{j}} -  e^{q_{j} -  q_{j+1}}\,.
		\end{split}
	\end{equation}
	
	Following \cite{mazzuca_int_sys}, we perform the non-canonical change of coordinates
	\begin{equation}    \label{ChangeOfVars}
		x_j = e^{-\frac{p_j}{2}},   \quad
		y_j = e^{\frac{q_{j}-q_{j+1}}{2}}= e^{-\frac{r_j}{2}},\quad r_j=q_{j+1}-q_{j},\quad j=1,\dots, N,
	\end{equation}
          to obtain a Lax Pair for this system.
	Indeed, in  these variables, the Hamiltonian \eqref{HamExpToda} transform into 
	
	\begin{equation}
		H_E({\mathbf x},{\mathbf y}) = \sum_{j=1}^N (x_j^2+y_j^2)\,,
	\end{equation}
	and the Hamilton's equations \eqref{HamiltonianForm} become
	\begin{equation}        \label{ExpToda}
		\dot{x_j} = \frac{x_j}{2} \left( y_j^2-y_{j-1}^2 \right),   \quad
		\dot{y_j} = \frac{y_j}{2} \left( x_{j+1}^2-x_j^2 \right),\quad j=1,\dots, N,
	\end{equation}
	where  $x_{N+1}=x_1, \, y_0 = y_N$.

    Let us introduce the matrices $L,A\in\textrm{Mat}(N)$ as
	
	\begin{align}
	\label{eq:L_exp_toda}
	& L= \begin{pmatrix} x_1^2+y_N^2 & x_1 y_1 & & &x_N y_N \\
			x_1 y_1 & x_2^2+y_1^2 & x_2 y_2 & &\\
			& \ddots & \ddots & \ddots &\\
			&& \ddots & \ddots &  x_{N-1}y_{N-1}  \\
			x_N y_N &&&  x_{N-1}y_{N-1} &x_N^2+y_{N-1}^2\end{pmatrix}\,,
			\\
	& A = \frac{1}{2} \begin{pmatrix} 0 & x_1 y_1 & & & - x_N y_N \\
			- x_1 y_1 & 0 & x_2 y_2 & &\\
			& \ddots & \ddots & \ddots &\\
			&& \ddots & \ddots &  x_{N-1}y_{N-1}  \\
			x_N y_N &&&  -x_{N-1}y_{N-1} &0\end{pmatrix}\,,
	\end{align}

	The system of equations \eqref{ExpToda} admits the Lax representation 
	\begin{equation}
		\dot{L} = [A,L].
		\label{eq:ccLA}
	\end{equation}
	Hence,  the quantities $H_m = \Tr{L^{m}}$, $m=1,\ldots,N$ are constants of motion as well as the eigenvalues of $L$.  
	For this integrable model, we define the generalized Gibbs ensemble as
		\begin{equation}
		\label{Gibbs_measure_TodaExp}
		\di\mu_{ET} = \frac{1}{Z^{H_E}_N(\alpha,P)} \prod_{j=1}^N x_j^{2\alpha -1}  y_j^{2\alpha -1} \mathbbm{1}_{x_j \geq 0}\mathbbm{1}_{y_j \geq 0} e^{- \Tr{P(L)} }   \di \bx \di \by\,,
	\end{equation}
	where  $\alpha>0$, and $P$ is a real-valued polynomial with positive leading coefficient. $Z_N^{H_E}(\alpha,P)$ is the normalization constant. We notice that the matrix $L$ is of type 1-\textit{ii}.
	\begin{remark}
	The definition of our Gibbs ensemble is more general than the one given in \cite{mazzuca_int_sys}, indeed there the authors were considering just the classical Gibbs ensemble for this model, so the case $P(x) = x/2$.
	\end{remark}
	We notice that the structure of \eqref{Gibbs_measure_TodaExp} resembles the one of $\mu_{N}^{(1)}$ \eqref{type1}, thus we want to apply Theorem \ref{THM:FINAL_CLT_INTRO}. To do this, we have to identify the functions $F,W$. First, as an application of Lemma \ref{LEM:CIRCULAR_TRACE} and Theorem \ref{thm:struct}, we obtain the following corollary
	
	\begin{corollary}
	Fix $m\in \N$, and consider the matrix $L$ \eqref{eq:L_exp_toda}. Then for $N$ big enough, there exists some $k=k(m)\in \N$, such that $\Tr{L^m}$ is a $k$-circular function, with seed $V$ and weed $\widetilde V$. Furthermore, both $V,\widetilde V$ are bounded from below away from $-\infty$.
	\end{corollary}

	As in the Toda lattice case, if we naively set $F(x,y) = x^{2\alpha -1}  y^{2\alpha -1} $, this would not fit in the hypotheses of our theorem. As in the previous case, we have just to consider a slight modification of the measure:
	
			\begin{equation}
		\label{Gibbs_measure_TodaExp_mod}
		\di\mu_{ET} = \frac{1}{Z^{H_E}_N(\alpha,P)} \prod_{j=1}^N x_j^{2\alpha -1}  y_j^{2\alpha -1} \exp\left(- \varepsilon\frac{x_j^2 + y_j^2}{2}\right)\mathbbm{1}_{x_j \geq 0}\mathbbm{1}_{y_j \geq 0} e^{- \Tr{P(L)} + \varepsilon\frac{x_j^2 + y_j^2}{2}}   \di \bx \di \by\,,
	\end{equation} 
	for fixed $\varepsilon>0$, but small. In this way, defining $F(x,y,\alpha) = x^{2\alpha -1}  y^{2\alpha -1} \exp(- \varepsilon\frac{x^2 + y^2}{2})$, and $W(\bx_1,\by_1,\bx_2,\by_2) = V(\bx_1,\by_1,\bx_2,\by_2) - \frac{\varepsilon}{2}\sum_{j=1}^{2k} x_j^2 + y_j^2 $ we are in the same hypotheses as Theorem \ref{THM:FINAL_CLT_INTRO}, thus we deduce the following corollary
	
	\begin{corollary}[CLT for the Exponential Toda lattice]
	Consider the Lax matrix L \eqref{eq:L_exp_toda} of the Exponential Toda lattice distributed according to the Generalized Gibbs Ensemble \eqref{Gibbs_measure_TodaExp}. Then, defining the Free energy $\cF_{HE}(\alpha,P)$ as
	
	\begin{equation}
	    \cF_{ET}(\alpha,P) = -\lim_{N\to\infty}\frac{1}{N}\ln(Z^{H_E}_N(\alpha,P))\,,
	\end{equation}
	for all $j\in\N$ fixed, we have the following weak limit
	
	\begin{equation}
	    \lim_{N\to\infty}  \frac{\Tr{L^j} - \meanval{\Tr{L^j}}  }{\sqrt{N}} \rightharpoonup \cN(0,\sigma^2)\,,
	\end{equation}
	where
        \begin{equation}
            \meanval{\Tr{L^j}} = iN\partial_t \cF_{ET}(\alpha,P+itx^j)_{\vert_{t=0}}\,, \quad \sigma^2=\vert\partial^2_t \cF_{ET}(\alpha,P+itx^j)_{\vert_{t=0}}\vert
            \end{equation}
	\end{corollary}

	\paragraph{The Laguerre $\beta$-ensemble in the high-temperature regime.}
	
	The Laguerre $\beta$-ensemble is a random matrix ensemble introduced by Dumitriu and Edelman in \cite{dued}. It has the following matrix representation
	
	\begin{equation}
	\label{eq:laguerre_matrix}
	    Q = \begin{pmatrix}
	    x_1^2 & x_1 y_1 & & & \\
			x_1 y_1 & x_2^2+y_1^2 & x_2 y_2 & &\\
			& \ddots & \ddots & \ddots &\\
			&& \ddots & \ddots &  x_{N-1}y_{N-1}  \\
			 &&&  x_{N-1}y_{N-1} &x_N^2+y_{N-1}^2
	    \end{pmatrix}\,,
	\end{equation}
	where the entries of $Q$ are distributed according to 
	\begin{equation}
	\label{eq:laguerre_beta}
	    \di \mu_L = \frac{1}{Z^L_N(\beta)} \prod_{j=1}^N x_j^{\beta(M-j+1)-1}\mathbbm{1}_{x_j\geq 0} \prod_{j=1}^{N-1}y^{\beta(N-j)-1}\mathbbm{1}_{y_j\geq 0}\exp\left( - \Tr{P(Q)}\right)\di\bx\di\by\,,
	\end{equation}
	here we consider the case $M=N$; $P$ can be any continuous function such that the partition function is well-defined. For our purpose we consider $P$ to be a polynomial of degree greater than one with positive leading coefficient, as in Hypotheses \ref{polynomial}.
	
	The remarkable property of this ensemble is that it is possible to explicitly compute the joint eigenvalue density as
	
	\begin{equation}
	    \di \mathbb{P}_L = \frac{1}{\mathfrak{Z}^L_N(\beta,P)} \prod_{j=1}^N \lambda_j^{\frac{\beta}{2}\left(M-N + 1\right) - 1}\mathbbm{1}_{\lambda_j\geq 0}\prod_{j<i}\vert\lambda_j -\lambda_i\vert^\beta e^{-\sum_{j=1}^NP(\lambda_j)} \di \boldsymbol{\lambda}\,.
	\end{equation}
	
	We are interested in the so-called high-temperature limit, i.e. when $\beta = \frac{2\alpha}{N}$, $\alpha\in\R_+$, which was considered in \cite{Allez2013}, where the authors were able to compute the density of states for the particular case $P(x)=x/2$. 
	
	In this regime, the density \eqref{eq:laguerre_beta} takes the form
	
	\begin{equation}
	    \di \mu_L = \frac{1}{Z^L_N(\alpha,P)} \prod_{j=1}^N x_j^{2\alpha\left(1-\frac{j+1}{N}\right)-1}\mathbbm{1}_{x_j\geq 0} \prod_{j=1}^{N-1}y_j^{2\alpha\left(1-\frac{j}{N}\right)-1}\mathbbm{1}_{y_j\geq 0}\exp\left( - \Tr{P(Q)}\right)\di\bx\di\by\,.
	\end{equation}
	
	We notice that the structure of this density resembles the one of $\mu^{(2)}_{N}$ \eqref{type2}. Furthermore, the matrix $Q$ is Toeplitz-like and tridiagonal, thus applying Lemma \ref{LEM:CIRCULAR_TRACE} and Theorem \ref{thm:struct}, we deduce the following corollary
	
	\begin{corollary}[CLT for Laguerre $\beta$-ensemble]
	Consider the matrix representation \eqref{eq:laguerre_matrix} of the Laguerre $\beta$-ensemble in the high-temperature regime, and let $P$ be a real polynomial of degree at least $1$ with positive leading coefficient. Then,   defining the Free energy $\cF_{L}(\alpha,P)$ as
	
	\begin{equation}
	    \cF_{L}(\alpha,P) = -\lim_{N\to\infty}\frac{1}{N}\ln(Z^{L}_N(\alpha,P))\,,
	\end{equation}
	for all $j\in\N$ fixed, we have the following weak limit
	
	\begin{equation}
	    \lim_{N\to\infty}  \frac{\Tr{Q^j} - \meanval{\Tr{Q^j}} }{\sqrt{N}} \rightharpoonup \cN(0,\sigma^2)\,,
	\end{equation} 
    where
    \begin{equation}
        \meanval{\Tr{Q^j}} = i  N\partial_t \cF_{L}(\alpha,P+itx^j)_{\vert_{t=0}} \,,\quad \sigma^2 = \vert\partial^2_t \cF_{L}(\alpha,P+itx^j)_{\vert_{t=0}}\vert
    \end{equation}
	\end{corollary}
Which is the perfect analogue of the result for the Exponential Toda lattice. 
Furthermore, we are ready to apply the second part of our result, indeed we can deduce the following identities 

\begin{equation}
\begin{split}
        &\partial_\alpha(\alpha \partial_t \cF_{L}(\alpha,P+itx^j)_{\vert_{t=0}}) = \partial_t \cF_{ET}(\alpha,P+itx^j)_{\vert_{t=0}}\,, \\
        &\partial_\alpha(\alpha \partial^2_t \cF_{L}(\alpha,P+itx^j)_{\vert_{t=0}}) = \partial^2_t \cF_{ET}(\alpha,P+itx^j)_{\vert_{t=0}}\,,
\end{split}
\end{equation}
thus, we can compute all the quantities involved in the previous theorems from the Free Energy of the Laguerre ensemble.

\begin{remark}
    We notice that in the Laguerre $\beta$-ensemble one can consider a more general regime $M=\gamma N$ with $\gamma\in (0,1]$. It is possible to obtain an analogous result for this situation, but one has to consider a slightly different $F$ and add an extra parameter depending on $\alpha, \gamma$
\end{remark}

\subsection{The Volterra lattice, and the antisymmetric $\beta$-ensemble at high-temperature}
\label{sec:volt_antigauss}
In this subsection, we focus on the Volterra lattice and its relation with the Antisymmetric $\beta$-ensemble \cite{DF10} in the high-temperature regime \cite{Forrester2021}.  These two systems were considered in \cite{mazzuca_int_sys}. In this paper, the authors considered the classical Gibbs ensemble for the Volterra lattice and were able to compute the density of states for this model connecting it to the Antisymmetric $\alpha$ ensemble \cite{Forrester2021}, which is related to the classical $\beta$ one introduced by Dumitriu and Forrester \cite{DF10}.

\paragraph{The Volterra Lattice.}
	The Volterra lattice (or \emph{discrete KdV equation}) is the following systems of $N$ coupled ODEs 
	\begin{equation}    \label{Volterra}
		\dot{a_j} = a_j \left(a_{j+1} - a_{j-1} \right), \qquad j=1,\dots,N,
	\end{equation}
	here $a_j\in\R_+$  for $j=1,\ldots,N$, and we consider  periodic boundary conditions $a_j=a_{j+N}$ for all $j\in\Z$.
	Volterra introduced it to study 
	evolution of populations in a hierarchical system of competing species. This system was considered by Kac and van Moerbeke in \cite{Kac1975}, who solved it explicitly using a discrete version of the inverse scattering transform introduced by  Flaschka \cite{Flashka1974b}.

    Introducing  on the phase space $\R^N_+$ the following Poisson bracket
	\begin{equation}
		\label{eq:poisson_volterra}
		\{ a_j, a_i \}_{\Volt} =  a_ja_i(\delta_{i,j+1} - \delta_{i,j-1})\,,
	\end{equation}
	and defining the Hamiltonian  $H_1 = \sum_{j=1}^N a_j\,$, we can rewrite the equations of motion \eqref{Volterra} in Hamiltonian form as
	\begin{equation}
		\dot{a}_j = \{ a_j, H_1\}_{\Volt}\,.
		\label{eq:hamvoltN}
	\end{equation}
	An elementary constant of motion for the system is $H_0 = \prod_{j=1}^N a_j$ which is  independent of  $H_1$.
	
	The Volterra lattice  is a completely integrable system, and  it admits several equivalents \emph{Lax representations}, see e.g. \cite{Kac1975,Moser75,mazzuca_int_sys}. We use the one presented in \cite{mazzuca_int_sys}. Specifically, we introduce the matrices $L,A\in \textrm{Mat}(\R,N)$ as
	
	\begin{align}
	    \label{eq:lax_volterra}
	    & L=\begin{pmatrix}
	    0& \sqrt{a_1} & & & -\sqrt{a_N} \\
			-\sqrt{a_1} & 0 & \sqrt{a_2} & &\\
			& \ddots & \ddots & \ddots &\\
			&& \ddots & \ddots &  \sqrt{a_{N-1}}  \\
			\sqrt{a_N} &&&  -\sqrt{a_{N-1}} &0
	    \end{pmatrix} \,,\\
     \vspace{10pt}
	    & A = \frac{1}{2}\sum_{j=1}^{N} \sqrt{a_ja_{j+1}} (E_{j,j+2}-E_{j+2,j})\,,
	\end{align}
	where $E_{r,s}$ is defined as $\left(E_{r,s}\right)_{ij}=\delta^i_r \delta^j_s$ and   $E_{j+N,i} =E_{j,i+N} = E_{j,i}$. Then, it follows that the equations of motion \eqref{eq:hamvoltN} are equivalent to 
	
	\begin{equation}
	    \dot{L} = [L;A]\,.
	\end{equation}
	In view of this Lax pair, we know that $\Tr{L^k}$ are constant of motion for the model. 
	
	Following \cite{mazzuca_int_sys}, we introduce the Generalized Gibbs Ensemble of the Volterra lattice \eqref{Volterra} as
	
	\begin{equation}
		\label{VolterraGibbs}
		\di\mu_{\Volt}(\ba) = \frac{e^{-\Tr{P(L)}}\prod_{j=1}^N a_j^{\alpha-1}\mathbbm{1}_{a_j>0}\di \ba }{Z_N^{\Volt}(\alpha,P)}\,,
	\end{equation}
where $\alpha>0$,  $P$ is a polynomial as in Hypotheses \ref{polynomial}  of the form $P(x) = (-1)^{j}x^{2j} + \textrm{l.o.t}$, where $\textrm{l.o.t}$ stands for lower order terms. Moreover, we notice that, given the antisymmetric nature of $L$, $\Tr{L^{2j+1}} = 0$. For this reason, we perform the change of coordinates $\sqrt{a_j} = x_j$, where we take just the positive root, so the previous measure read

	\begin{equation}
		\label{VolterraGibbs_mod}
		\di\mu_{\Volt}(\bx) = \frac{e^{-\Tr{P(L)}}\prod_{j=1}^N x_j^{2\alpha-1}\mathbbm{1}_{x_j>0}\di \bx }{Z_N^{\Volt}(\alpha,P)}\,.
	\end{equation}
	
	This Generalized Gibbs ensemble resembles the structure of $\mu_N^{(1)}$ \eqref{type1}, we have just to identify $F,W$. We notice that it is possible to generalize Theorem \ref{thm:struct} also for the antisymmetric situation, moreover, the matrix $L$ is circular-like. Thus, by applying Lemma \ref{LEM:CIRCULAR_TRACE} and Theorem \ref{thm:struct}, we deduce the following Corollary:
	
	\begin{corollary}
		Fix $m\in \N$, and consider the matrix $L$ \eqref{eq:lax_volterra}. Then for $N$ big enough, there exists some $k=k(m)\in \N$, such that $\Tr{L^m}$ is a $k$-circular function, with seed $V$ and weed $\widetilde V$. Furthermore, both $V,\widetilde V$ are bounded from below away from $-\infty$.
	\end{corollary}
	
	Thus, following the same kind of reasoning as in the Toda lattice, section \ref{sec:toda_gaussian}, and the Exponential Toda lattice, section \ref{sec:exptoda_laguerre}, we deduce the following:
	
	\begin{corollary}[CLT for Volterra lattice]
	Consider the Lax matrix L \eqref{eq:lax_volterra} of the Volterra lattice distributed according to the Generalized Gibbs Ensemble \eqref{VolterraGibbs}. Then, defining the Free energy $\cF_{\Volt}(\alpha,P)$ as
	
	\begin{equation}
	    \cF_{\Volt}(\alpha,P) = -\lim_{N\to\infty}\frac{1}{N}\ln(Z^{\Volt}_N(\alpha,P))\,,
	\end{equation}
	for all $j\in\N$ fixed, we have the following weak limit
	
	\begin{equation}
	    \lim_{N\to\infty}  \frac{\Tr{L^j} - \meanval{\Tr{L^j}} }{\sqrt{N}} \rightharpoonup \cN(0,\sigma^2)\,,
	\end{equation}
	where
    \begin{equation}
        \meanval{\Tr{L^j}} = iN\partial_t \cF_{\Volt}(\alpha,P+ itx^j)_{\vert_{t=0}}\,,\quad \sigma^2 = \vert\partial^2_t \cF_{\Volt}(\alpha,P+tx^j)_{\vert_{t=0}}\vert\,.
    \end{equation}
	\end{corollary}
	
	\paragraph{The Antisymmetric $\beta$-ensemble in the high-temperature regime}
	
	The Antisymmetric $\beta$-ensemble is a random matrix ensemble introduced by Dumitriu and Forrester in \cite{DF10}; it has the following matrix representation
	
	\begin{equation}
	\label{eq:antiguassian_matrix}
	    Q = \begin{pmatrix}
	    0& x_1 & & &  \\
			-x_1 & 0 & x_2 & &\\
			& \ddots & \ddots & \ddots &\\
			&& \ddots & \ddots &  x_{N-1}  \\
		 &&&  -x_{N-1} &0
	    \end{pmatrix}\,,
	\end{equation}
	and the entries of the matrix $Q$ are distributed according to
	
\begin{equation}
\label{eq:antiguassian}
    \di \mu_{AG} = \frac{1}{Z^{AG}_N(\beta,P)}\prod_{j=1}^{N-1}x_j^{\beta(N-j) - 1}\mathbbm{1}_{x_j\geq 0}\exp(-\Tr{P(Q)})\di \bx\,,
\end{equation}
here $P$ can be any function that makes \eqref{eq:antiguassian} normalizable. For our purpose, we consider as in Hypotheses \ref{polynomial}  $P$ polynomial of the form $P(x) = (-1)^{j}x^{2j} + \textrm{l.o.t}$.

As in the previous cases, we are interested in the high-temperature regime for this model, so we set $\beta =\frac{2\alpha}{N}$, and we rewrite the previous density as
\begin{equation}
\label{eq:antiguassian_ht}
    \di \mu_{AG} = \frac{1}{Z^{AG}_N(\alpha,P)}\prod_{j=1}^{N-1}x_j^{2\alpha\left(1-\frac{j}{N}\right) - 1}\mathbbm{1}_{x_j\geq 0}\exp(-\Tr{P(Q)})\di \bx\,.
\end{equation}
This regime was introduced in \cite{mazzuca_int_sys}, where the author computed the density of states for this model in the case $P(x) = -x^2/2$.
The structure of this last density \eqref{eq:antiguassian_ht} resembles the one of $\mu^{(2)}_N$\eqref{type2}, indeed proceeding as in the case of the Volterra lattice, we deduce the following corollary
	
	\begin{corollary}[CLT for Antisymmetric $\beta$-ensemble]
	Consider the matrix representation \eqref{eq:antiguassian_matrix} of the Antisymmetric  $\beta$-ensemble in the high-temperature regime, endowed with the probability distribution $\di \mu_{AG}$ \eqref{eq:antiguassian_ht}, and let $P(x)$ be a polynomial of the form $P(x) = (-1)^{\ell} x^{2\ell} + \textrm{l.o.t.}$. Then, defining the Free energy $\cF_{AG}(\alpha,P)$ as
	
	\begin{equation}
	    \cF_{AG}(\alpha,P) = -\lim_{N\to\infty}\frac{1}{N}\ln(Z^{AG}_N(\alpha,P))\,,
	\end{equation}
	for all $j\in\N$ fixed, we have the following weak limit
	
	\begin{equation}
	    \lim_{N\to\infty}  \frac{\Tr{Q^j} - \meanval{\Tr{Q^j}}  }{\sqrt{N}} \rightharpoonup \cN(0,\sigma^2)\,.
	\end{equation} 
 	where
    \begin{equation}
        \meanval{\Tr{Q^j}} = iN\partial_t \cF_{AG}(\alpha,P+itx^j)_{\vert_{t=0}}\,,\quad \sigma^2 = \vert\partial^2_t \cF_{AG}(\alpha,P+itx^j)_{\vert_{t=0}}\vert\,.
    \end{equation}
	\end{corollary}
Which is the perfect analogue of the result for the Volterra lattice. 

\begin{remark}
\label{rem:derivativeSusceptibility}
In view of Theorem \ref{THM:FINAL_CLT_INTRO}, we deduce the following identities 

\begin{equation}
\begin{split}
        &\partial_\alpha(\alpha \partial_t \cF_{AG}(\alpha,P+itx^j)_{\vert_{t=0}}) = \partial_t \cF_{AG}(\alpha,P+itx^j)_{\vert_{t=0}}\,, \\
        &\partial_\alpha(\alpha \partial^2_t \cF_{AG}(\alpha,P+itx^j)_{\vert_{t=0}}) = \partial^2_t \cF_{AG}(\alpha,P+itx^j)_{\vert_{t=0}}
\end{split}
\end{equation}
\end{remark}

\subsection{The defocusing Ablowitz-Ladik lattice, and the  Circular  $\beta$-ensemble at high-temperature}
\label{sec:AL_circ}
In this subsection, we focus on the defocusing Ablowitz-Ladik lattice, and its relation to the  Circular  $\beta$-ensemble at high-temperature \cite{Hardy2021,mazzuca2021generalized,spohn2021hydrodynamic}.
This relation was highlighted by one of the present authors and T. Grava \cite{mazzuca2021generalized}, and independently by H. Spohn \cite{spohn2021hydrodynamic}. In these papers, the authors were able to characterize the density of states of the Ablowitz-Ladik lattice in terms of the one of the Circular $\beta$-ensemble in the high-temperature regime. Moreover, in \cite{mazzuca2021generalized} the authors were able to compute explicitly the density of states in the case of linear potential in terms of the solution of the Double Confluent Heun Equation \cite{DLMF} highlighting a connection with the Painlev\'e equations \cite{Lisovyy,FIKN}.  In \cite{mazzucamemin}, the two present authors obtained a large deviations principles for the empirical spectral measure for any continuous and bounded potential. 

\paragraph{The defocusing Ablowitz-Ladik lattice.}
	The defocusing Ablowitz-Ladik (dAL) lattice is defined by the following system of nonlinear equations	
	\begin{equation}
		\label{eq:AL}
		i	\dot{a}_j =-(a_{j+1}+a_{j-1}-2a_j)+|a_j|^2(a_{j-1}+a_{j+1})\,,
	\end{equation}
	where $a_j(t)\in\C$. We assume $N$-periodic boundary conditions $a_{j+N}=a_j$, for all $j\in \Z$. 
	The dAL lattice  was introduced by Ablowitz and Ladik  \cite{Ablowitz1974,Ablowitz1975} as the spatial  integrable discretization of the defocusing cubic  nonlinear Schr\"odinger Equation for the complex function 
	$\psi(x,t)$, $x\in S^1$ and $t\in\R$:
	\begin{equation}
		i \partial_t \psi(x,t) = -\partial^2_{x} \psi(x,t) +  2\lvert \psi(x,t) \rvert^2 \psi(x,t).
	\end{equation}
	
	As for the others dynamical systems that we considered, the dAL is an integrable system. Its integrability was proved by Ablowitz and Ladik by discretizing the $2\times 2$ Zakharov-Shabat Lax pair \cite{Ablowitz1973, Ablowitz1974} of the cubic  nonlinear Schr\"odinger equation. 
	Furthermore, Nenciu and Simon \cite{Nenciu2005,Simon2005}  constructed a new Lax pair for this lattice.  Following their construction we double the size of the chain according to the periodic boundary condition, thus we consider a chain of $2N$ particles $a_1, \ldots, a_{2N}$ such that $a_j = a_{j+N}$ for $j=1,\ldots, N$.
	Define the  $2\times2$ unitary matrix  $\Xi_j$ 		
	\begin{equation}
	\label{eq:Xi}
		\Xi_j = \begin{pmatrix}
			\wo a_j & \rho_j \\
			\rho_j & -a_j
		\end{pmatrix}\, ,\quad j=1,\dots, 2N\, ,\quad \rho_j = \sqrt{1-|a_j|^2}
	\end{equation}
	and the $2N\times 2N$ matrices

	\begin{equation}
 \label{eq:ML}
		\cM= \begin{pmatrix}
			-a_{2N}&&&&& \rho_{2N} \\
			& \Xi_2 \\
			&& \Xi_4 \\
			&&& \ddots \\
			&&&&\Xi_{2N-2}\\
			\rho_{2N} &&&&& \wo a_{2N}
		\end{pmatrix}\, ,\qquad 
		\cL = \begin{pmatrix}
			\Xi_{1} \\
			& \Xi_3 \\
			&& \ddots \\
			&&&\Xi_{2N-1}
		\end{pmatrix} \,.
	\end{equation}
Now  let us define the  unitary  Lax matrix 
	\begin{equation}
		\label{eq:Lax_AL}
		\cE  = \cL \cM\,,
	\end{equation}
	that has the structure  of a $5$-band periodic diagonal matrix. 
	The matrix $\cE$ is a periodic  CMV  matrix \cite{Cantero2005}.
The equations of motion \eqref{eq:AL}  are equivalent to  the following Lax equation for the matrix $\cE$:	
	\begin{equation}
		\label{eq:Lax_pair}
		\dot \cE = i\left[\cE, \cE^+ + (\cE^+)^\dagger\right]\,,
	\end{equation}	
	where $^\dagger$ stands for hermitian conjugate and 
	\begin{equation}
		\cE^+_{j,k} = \begin{cases}
			\frac{1}{2} \cE_{j,j} \quad j = k \\
			\cE_{j,k} \quad k = j + 1 \, \mod \, 2N \, \mbox{or} \, k = j + 2 \, \mod \, 2N  \\
			0 \quad \mbox{otherwise}.
		\end{cases}
	\end{equation}	
	The matrix $\cE$ is a circular-like matrix, this can be deduced from the fact that both $\cL,\cM$ are circular-like \eqref{eq:ML}, and the entries are continuous functions on a compact set, so by applying Lemma \ref{LEM:CIRCULAR_TRACE}, we deduce the following:
	
		\begin{corollary}
		\label{cor:struct_CMV}
	Fix $m\in \N$, and consider the matrix $\cE$ \eqref{eq:Lax_AL}. Then for $N$ big enough, there exists some $k=k(m)\in \N$, such that $\Tr{\cE^m}$ is a $k$-circular function, with seed $V$ and weed $\widetilde V$.
	\end{corollary}
	
    Following \cite{spohn2021hydrodynamic,mazzuca2021generalized,mazzucamemin}, we notice that the quantity $K_0 = \prod_{j=1}^N (1-\vert a_j\vert^2)$ is conserved, so this means that if $\vert a_j(0)\vert < 1$ for all $j=1,\dots, N$ then $\vert a_j(t)\vert < 1$ for all $j=1,\dots, N$ for all $t\in \R$, so we can consider $\D^N$ as our phase space, here $\D = \{z\in\C\,\vert\, \vert z\vert < 1$\}. On this phase space, we introduce the Generalized Gibbs ensemble for the defocusing AL lattice as
    
    \begin{equation}
    \label{eq:gibbs_dAL}
        \di\mu_{\dAL} = \frac{\prod_{j=1}^{N} \left(1-|a_j|^2\right)^{\alpha-1}\mathbbm{1}_{a_j\in\D}\exp\left(-\mbox{Tr}\left(\Re P\left(\cE\right)\right)\right)  \di^2 \ba}{Z_N^{\dAL}(\alpha,\Re P)}\,,
    \end{equation}
	where $P\in \C[X]$ is a polynomial. In view of Corollary \ref{cor:struct_CMV}, we are in the hypotheses of Theorem \ref{THM:FINAL_CLT_INTRO}, thus we deduce the following:
	
	\begin{corollary}[CLT for defocusing Ablowitz--Ladik lattice]
	Consider the Lax matrix $\cE$ \eqref{eq:Lax_AL} of the defocusing Ablowitz--Ladik lattice distributed according to the Generalized Gibbs Ensemble \eqref{eq:gibbs_dAL}. Then, defining the Free energy $\cF_{\dAL}(\alpha,\Re P)$ as
	
	\begin{equation}
	    \cF_{\dAL}(\alpha,\Re P) = -\lim_{N\to\infty}\frac{\ln(Z^{\dAL}_N(\alpha,\Re P))}{2N}\,,
	\end{equation}
	for all $j\in\N$ fixed, we have the following weak limits
	
	\begin{equation}
 \begin{split}
     & \lim_{N\to\infty}  \frac{\Tr{\Re \cE^j} - \meanval{\Re \Tr{\cE^j}} }{\sqrt{N}} \rightharpoonup  \cN(0,\sigma_{\text{real}}^2)\,, \\
     & \lim_{N\to\infty}  \frac{\Tr{\Im \cE^j} - \meanval{\Im \Tr{\cE^j}} }{\sqrt{N}} \rightharpoonup  \cN(0,\sigma_{\text{imag}}^2)\,,
 \end{split}
	\end{equation}

        where
        \begin{equation}
        \begin{split}
            & \meanval{\Tr{\Re \cE^j}} =i N\partial_t \cF_{\dAL}(\alpha,\Re P+it \Re x^j)_{\vert_{t=0}} \,,\quad \sigma_{\text{real}}^2 = \vert\partial^2_t \cF_{\dAL}(\alpha,\Re P+it\Re x^j)_{\vert_{t=0}}\vert\,. \\
            & \meanval{\Tr{\Im \cE^j}} =i N\partial_t \cF_{\dAL}(\alpha,\Re P+it \Im x^j)_{\vert_{t=0}} \,,\quad \sigma_{\text{imag}}^2 = \vert\partial^2_t \cF_{\dAL}(\alpha,\Re P+it\Im x^j)_{\vert_{t=0}}\vert\,.
        \end{split}
        \end{equation}
	\end{corollary}
	
	\paragraph{The circular $\beta$-ensemble at high-temperature.} 
	The circular $\beta$-ensemble was introduced by Killip and Nenciu in \cite{Killip2004}; as the other $\beta$-ensembles that we considered, it possesses a matrix representation. Consider the two block diagonal matrices
	
	\begin{equation}
		\label{eq:LME}
		\fM = \mbox{diag}\left(\Xi_1,\Xi_3,\Xi_{5} \ldots,\right) \quad \mbox{ and } \quad \fL = \mbox{diag}\left(\Xi_{0},\Xi_2,\Xi_4, \ldots\right)\,,
	\end{equation}
	where the blocks $\Xi_j$, $j=1,\dots, N-1$ are defined in \eqref{eq:Xi}, while $\Xi_{0} = (1)$ and $\Xi_{N} = (\wo \alpha_{N})$ are $1\times 1$ matrices. 
	Then, we define $\fE$ as follows
	\begin{equation}
	\label{eq:CMV_true}
	\fE = \fL\fM.
	\end{equation}
	
	The entries of this matrix are distributed according to 
	
	\begin{equation}
	    \di\mu_C = \frac{\prod_{j=1}^{N-1} \left(1-|a_j|^2\right)^{\frac{\beta}{2}(N-j)-1}\mathbbm{1}_{a_j\in\D}\exp\left(-\mbox{Tr}\left(\Re P\left(\fE\right)\right)\right)  \prod_{j=1}^{N-1}\di^2a_j \frac{\di a_N}{i a_N}}{Z_N^{C}(\beta,\Re P)}\,.
	\end{equation}
	As for the other $\beta$-ensembles, one can explicitly compute the joint eigenvalue density for this ensemble as
	
	\begin{equation}
	    \di\mathbb{P}_{C} = \frac{1}{\mathfrak{Z}^{C}_N(\beta,\Re P)}  \prod_{j<\ell}\vert e^{i\theta_j} -e^{i\theta_\ell}\vert^\beta \mathbbm{1}_{\theta_j\in \T} e^{-\sum_{j=1}^N\Re P(e^{i\theta_j})} \di \boldsymbol{\theta}\,,
	\end{equation}
	
	here $\T = [-\pi,\pi)$, $e^{i\theta_j}$ are the eigenvalues of $\fE$, and $P$ can be any continuous function. We restrict our attention to the class of polynomials $P\in \C[X]$. 
	
	We are interested in the high-temperature limit for this ensemble \cite{mazzuca2021generalized,spohn2021hydrodynamic}, so we set $\beta =\frac{2\alpha}{N}$, obtaining 
	
		\begin{equation}
		\label{eq:circular_gibbs}
	    \di\mu_C = \frac{\prod_{j=1}^{N-1} \left(1-|a_j|^2\right)^{\alpha\left(1-\frac{j}{N}\right)-1}\mathbbm{1}_{a_j\in\D}\exp\left(-\mbox{Tr}\left(\Re P\left(\fE\right)\right)\right)  \prod_{j=1}^{N-1}\di^2a_j \frac{\di a_N}{i a_N}}{Z_N^{C}(\alpha,P)}\,.
	\end{equation}
	  
	So, in view of Corollary \ref{cor:struct_CMV}, we are in the hypotheses of Theorem \ref{THM:FINAL_CLT_INTRO}, so we deduce the following
	
		\begin{corollary}[CLT for Circular $\beta$-ensemble]
	Consider the matrix representation \eqref{eq:CMV_true} of the Circular  $\beta$-ensemble in the high-temperature regime, endowed with the probability distribution $\di \mu_{C}$ \eqref{eq:circular_gibbs}, and let $P\in \C[X]$ be a polynomial. Then, defining the Free energy $\cF_{C}(\alpha,\Re P)$ as
	
	\begin{equation}
	    \cF_{C}(\alpha,\Re P) = -\lim_{N\to\infty}\frac{\ln(Z^{C}_N(\alpha,\Re P))}{2N}\,,
	\end{equation}
	for all $j\in\N$ fixed, we have the following weak limit
	
	\begin{equation}
 \begin{split}
     & \lim_{N\to\infty}  \frac{\Tr{\Re \fE^j} - \meanval{\Re \Tr{\fE^j}} }{\sqrt{N}} \rightharpoonup  \cN(0,\sigma_{\text{real}}^2)\,, \\
     & \lim_{N\to\infty}  \frac{\Tr{\Im \fE^j} - \meanval{\Im \Tr{\fE^j}} }{\sqrt{N}} \rightharpoonup  \cN(0,\sigma_{\text{imag}}^2)\,,
 \end{split}
	\end{equation}

        where
        \begin{equation}
        \begin{split}
            & \meanval{\Tr{\Re \fE^j}} =i N\partial_t \cF_{C}(\alpha,\Re P+it \Re x^j)_{\vert_{t=0}} \,,\quad \sigma_{\text{real}}^2 = \vert\partial^2_t \cF_{C}(\alpha,\Re P+it\Re x^j)_{\vert_{t=0}}\vert\,. \\
            & \meanval{\Tr{\Im \cE^j}} =i N\partial_t \cF_{C}(\alpha,\Re P+it \Im x^j)_{\vert_{t=0}} \,,\quad \sigma_{\text{imag}}^2 = \vert\partial^2_t \cF_{C}(\alpha,\Re P+it\Im x^j)_{\vert_{t=0}}\vert\,.
        \end{split}
        \end{equation}
	\end{corollary}
	
	\begin{remark}
	\begin{itemize}
	We notice that
	    \item 		Hardy and Lambert in \cite{Hardy2021} already proved a CLT theorem for the Circular $\beta$-ensemble in the high-temperature regime for a wider class of functions and potentials than we can consider with our result. Nevertheless, we highlight the fact that in our case we can explicitly compute the means, and the variances in terms of the Free energy. 
	    
	    \item The following identities hold in view of the last part of Theorem \ref{THM:FINAL_CLT_INTRO}

\begin{equation}
\begin{split}
        &\partial_\alpha(\alpha \partial_t \cF_{C}(\alpha,\Re P+it\Re x^j)_{\vert_{t=0}}) = \partial_t \cF_{\dAL}(\alpha,\Re P+it\Re x^j)_{\vert_{t=0}}\,, \\
        &\partial_\alpha(\alpha \partial^2_t \cF_{C}(\alpha,\Re P+it\Re x^j)_{\vert_{t=0}}) = \partial^2_t \cF_{\dAL}(\alpha,\Re P+it\Re x^j)_{\vert_{t=0}}\,,\\
        &\partial_\alpha(\alpha \partial_t \cF_{C}(\alpha,\Re P+it\Im x^j)_{\vert_{t=0}}) = \partial_t \cF_{\dAL}(\alpha,\Re P+it\Im x^j)_{\vert_{t=0}}\,, \\
        &\partial_\alpha(\alpha \partial^2_t \cF_{C}(\alpha,\Re P+it\Im x^j)_{\vert_{t=0}}) = \partial^2_t \cF_{\dAL}(\alpha,\Re P+it\Im x^j)_{\vert_{t=0}}\,.
\end{split}
\end{equation}
This relation was already proved in \cite{mazzuca2021generalized} with the same kind of argument that we followed.
	\end{itemize}

	\end{remark}

\subsection{The defocusing Schur flow, and the Jacobi $\beta$-ensemble at high-temperature}
\label{sec:schur_jacobi}
In this subsection, we focus on the defocusing Schur flow \cite{Golinskii}, and its relation to the  Jacobi  $\beta$-ensemble at high-temperature \cite{Forrester2021}. This relation was first noticed in \cite{spohn2021hydrodynamic}, and then the two present authors obtained a large deviations principles for the empirical spectral measure for the defocusing Schur flow, and they were able to link it to the one of the Jacobi $\beta$-ensemble in the high-temperature regime \cite{mazzucamemin}.

\paragraph{The defocusing Schur flow.}
The defocusing Schur flow is the system of ODEs \cite{Golinskii}

\begin{equation}
\label{eq:schur}
    \dot{a_{j}} = \rho_j^2(a_{j+1}-a_{j-1})\,, \quad \rho_j = \sqrt{1-|a_j|^2}\,,
\end{equation}
and, as before, we consider periodic boundary conditions, namely $a_j = a_{j+N}$ for all $j\in \Z$.

We notice that, if one chooses an initial data such that $a_j(0)\in\R$ for all $j=1,\ldots,N$, then $a_j(t)\in\R$ for all times. Moreover, it is straightforward to verify that $K_0 = \prod_{j=1}^N \left( 1-|a_j|^2\right)$ is conserved along the Schur flow. This implies that we can choose as phase space for the Schur flow the $N$-cube $\mathbb{I}^N$, where $\I := (-1,1)$.
Furthermore, it was shown in \cite{Golinskii}, that the Schur flow has the same Lax matrix as the focusing Ablowitz--Laddik lattice.

Following \cite{spohn2021hydrodynamic,mazzuca_int_sys}, on $\I^N$ we define the finite volume limit GGE as

    \begin{equation}
    \label{eq:gibbs_dSchur}
        \di\mu_{\dSc}(\ba) = \frac{\prod_{j=1}^{N} \left(1-a_j^2\right)^{\alpha-1}\mathbbm{1}_{a_j\in\I}\exp\left(-\mbox{Tr}\left(P\left(\cE\right)\right)\right)  \di \ba}{Z_N^{\dSc}(\alpha,P)}\,,
    \end{equation}
where $P\,:\, \R \to \R$ is a polynomial. Thanks to Corollary \ref{cor:struct_CMV}, we can apply Theorem \ref{THM:FINAL_CLT_INTRO} obtaining a CLT theorem for the defocusing Schur flow

\begin{corollary}[CLT for defocusing Schur flow]
	Consider the Lax matrix $\cE$ \eqref{eq:Lax_AL} of the defocusing Schur flow distributed according to the Generalized Gibbs Ensemble \eqref{eq:gibbs_dSchur}. Then, defining the Free energy $\cF_{\dSc}(\alpha,P)$ as
	
	\begin{equation}
	    \cF_{\dSc}(\alpha,P) = -\lim_{N\to\infty}\frac{\ln(Z^{\dSc}_N(\alpha,P))}{2N}\,,
	\end{equation}
	for all $j\in\N$ fixed, we have the following weak limit
	
	\begin{equation}
	    \lim_{N\to\infty}  \frac{\Tr{\cE^j} -\meanval{\Tr{\cE^j}}   }{\sqrt{N}} \rightharpoonup \cN(0,\sigma^2 )\,,
	\end{equation}
	where
        \begin{equation}
            \meanval{\Tr{\cE^j} } = i  N\partial_t \cF_{\dSc}(\alpha,P+itx^j)_{\vert_{t=0}}\,,\quad \sigma^2 = \vert\partial^2_t \cF_{\dSc}(\alpha,P+itx^j)_{\vert_{t=0}}\vert\,.
        \end{equation}
	\end{corollary}

\paragraph{The Jacobi $\beta$-ensemble in the high-temperature regime.}

The Jacobi $\beta$-ensemble is a random matrix ensemble introduced by Killip and Nenciu in \cite{Killip2004}. It has two slightly different matrix representations. The first one is the same as the Circular $\beta$-ensemble \eqref{eq:CMV_true}, but the distribution of the entries of the matrix is

	\begin{equation}
	\label{eq:jacobi_distr}
	    \di\mu_J(\ba) = \frac{\prod_{j=1}^{2N-1} \left(1-a_j^2\right)^{\beta(2N-j)/4-1}\prod_{j=1}^{2N-1}(1-a_j)^{\ta +1-\beta/4} (1+ (-1)^ja_j)^{\tb +1-\beta/4}\mathbbm{1}_{a_j\in\I}\exp\left(-\mbox{Tr}\left(P\left(\fE\right)\right)\right)  \di \ba}{Z_N^{J}(\beta,P)}\,,
	\end{equation}
	
	where $\ta,\tb > -1$, $P\in\R[X]$ is a polynomial. We notice that we are considering an \textit{even} number of random variables, and $a_j\in\R$; for these reasons, all the eigenvalues of $\fE$ come in pairs, meaning that if $e^{i\theta}$ is an eigenvalue, then $e^{-i\theta}$ is another one. Exploiting this symmetry, Killip and Nenciu found another matrix representation for this ensemble
\begin{equation}
    J = \begin{pmatrix}
        c_1 & b_1&&&\\
        b_1 & c_2 & b_2 &&\\
        &\ddots & \ddots&\ddots& \\
        && \ddots&\ddots& b_{N-1}\\
        &&&b_{N-1} & c_N
    \end{pmatrix}\,, \quad \begin{cases}
    b_{j} = \left( (1-a_{2j-2}) (1-a_{2j-1}^2)(1+a_{2j})\right)^{1/2} \\
    c_{j} = (1-a_{2j-2})a_{2j-1} - (1+a_{2j-2})a_{2j-3}
    \end{cases}\,,
\end{equation}
where $a_0 = a_{2N} = -1$, and the eigenvalues $\{\lambda_j\}_{j=1}^N$ of $J$ are related to the one of $\fE$ as $\lambda_j = \cos(\theta_j)$.

Also, in this case, it is possible to compute explicitly the joint eigenvalue density for this model as

	\begin{equation}
	    \di\mathbb{P}_{J} = \frac{1}{\mathfrak{Z}^{J}_N(\beta,P)}  \prod_{j<\ell}\vert \cos(\theta_j) -\cos(\theta_\ell)\vert^\beta \mathbbm{1}_{\theta_j\in \T} e^{-2\sum_{j=1}^NP(\cos(\theta_j))} \di \boldsymbol{\theta}\,.
	\end{equation}

As in the previous cases, we are interested in the high-temperature regime for this ensemble, so we wet $\beta = \frac{2\alpha}{N}$, thus the measure \eqref{eq:jacobi_distr} read

\begin{equation}
\label{eq:ht_jacobi}
    	    \di\mu_J(\ba) = \frac{\prod_{j=1}^{2N-1} \left(1-a_j^2\right)^{\alpha\left(1-\frac{j}{2N}\right)}\prod_{j=1}^{2N-1}(1-a_j)^{\ta +1-\frac{\alpha}{2N}} (1+ (-1)^ja_j)^{\tb +1-\frac{\alpha}{2N}}\mathbbm{1}_{a_j\in\I}\exp\left(-\mbox{Tr}\left(P\left(\fE\right)\right)\right)  \di \ba}{Z_N^{J}(\beta,P)}\,.
\end{equation}

This regime was considered in \cite{Trinh2021} and in the recent paper \cite{trinh2023beta}, where the authors established a CLT for polynomial test functions in the absence of external potential ($P=0$ in \eqref{eq:ht_jacobi} ) by considering orthogonal polynomials, obtaining an explicit recurrence relation for the limiting variance.

Again, thanks to Corollary \ref{cor:struct_CMV}, we can apply Theorem \ref{THM:FINAL_CLT_INTRO} deducing the following

\begin{corollary}[CLT for Jacobi $\beta$-ensemble in the high-temperature]
	Consider the matrix representation $\fE$ \eqref{eq:CMV_true} of the Jacobi $\beta$-ensemble in the high-temperature regime \eqref{eq:ht_jacobi} . Then, defining the Free energy $\cF_{J}(\alpha,P)$ as
	
	\begin{equation}
	    \cF_{J}(\alpha,P) = -\lim_{N\to\infty}\frac{\ln(Z^{J}_N(\alpha,P))}{N}\,,
	\end{equation}
	for all $j\in\N$ fixed, we have the following weak limit
	
	\begin{equation}
	    \lim_{N\to\infty}  \frac{\Tr{\fE^j} -  \meanval{ \Tr{\fE^j}}   }{\sqrt{N}} \rightharpoonup \cN(0,\sigma^2)\,,
	\end{equation}
	where
        \begin{equation}
            \meanval{ \Tr{\fE^j}} = i N\partial_t \cF_{J}(\alpha,P+itx^j)_{\vert_{t=0}}\,,\quad \sigma^2 = \vert\partial^2_t \cF_{J}(\alpha,P+itx^j)_{\vert_{t=0}}\vert\,.
        \end{equation}
	\end{corollary}
	
	\begin{remark}
	We notice that for $N$ even, for $\ta+\tb = -1 + \frac{\beta}{4}$ we can apply the final part of Theorem \ref{THM:FINAL_CLT_INTRO}, thus we deduce that 
	
	\begin{equation}
\begin{split}
        &\partial_\alpha(\alpha \partial_t \cF_{J}(\alpha,P+itx^j)_{\vert_{t=0}}) = \partial_t \cF_{\dSc}(\alpha,P+itx^j)_{\vert_{t=0}}\,, \\
        &\partial_\alpha(\alpha \partial^2_t \cF_{J}(\alpha,P+itx^j)_{\vert_{t=0}}) = \partial_t^2 \cF_{\dSc}(\alpha,P+itx^j)_{\vert_{t=0}}
\end{split}
\end{equation}
	\end{remark}

\subsection{The Itoh--Narita--Bogoyavleskii lattices}
\label{sec:INB}
In this section, we apply our results to two families of integrable lattices with short-range interaction that generalize the Volterra one \eqref{Volterra}. These families are described in \cite{Bogoyavlensky1991}  (see also \cite{Bogoyavlensky1988,Itoh1975,Narita1982}).

	One is called  \emph{additive Itoh--Narita--Bogoyavleskii (INB) $r$-lattice} and is defined by the following equations
	\begin{align}
		\dot{a}_{i} &= a_{i} \left( \sum_{j=1}^{r} a_{i+j}- \sum_{j=1}^{r} a_{i-j} \right),
		\quad
		i = 1,\dots,N,\;N\geq r\in \N.
		\label{eq:inbpadd}
		\end{align}
		The second family is called the 	\emph{multiplicative Itoh--Narita--Bogoyavleskii (INB) $r$-lattice} and is defined by the equations
		\begin{align}
		\dot{a}_{i} &= a_{i} \left( \prod_{j=1}^{r} a_{i+j}- \prod_{j=1}^{r} a_{i-j} \right),
		\quad
		i = 1,\dots,N,\;N\geq r\in \N.
		\label{eq:inbpmul}
	\end{align}
	In both cases we consider  the periodicity condition
	$a_{j+N}=a_{j}$. We notice that setting $r=1$, we recover in both cases the Volterra lattice. 
	Moreover, both families admit the KdV equation as continuum limits,	see \cite{Bogoyavlensky1991}. 
	
	In both cases the interaction is short-range, but in the additive case \eqref{eq:inbpadd} the nonlinearity is quadratic as in the Volterra lattice, instead in the multiplicative one \eqref{eq:inbpmul} it is of polynomial order.
	
	As we already mentioned, both families are integrable for all $r\in\N$, indeed both families admits a Lax pair formulation.
	For the additive INB lattice \eqref{eq:inbpadd}, it reads
	\begin{align}
	\label{eq:inbpLadd}
		L^{(+,r)} & = \sum_{i=1}^{N}
		\left( a_{i+r} E_{i+r,i} + E_{i,i+1}\right)\\
		& =	\begin{pmatrix}
	0&1&0&\cdots&\tikz[remember picture]\node[inner sep=0pt] (b) {$a_{N-r}$};&0&0&0\\
	0&0&1&\cdots&0&a_{N-r+1}&0&0\\
	0&0&0&1&\cdots&0&a_{N-r+2}&0\\
	\vdots&\ddots&\ddots&\ddots&\ddots&\ddots&&\\
	a_{r+1}&0&\cdots&\cdots&0&1&0&\tikz[remember picture]\node[inner sep=0pt] (a) {0};\\
	0&a_{r+2}&0\cdots&&\ddots&\ddots&\ddots&\\
	\vdots&\ddots&\ddots&\ddots&\ddots&0&0&1\\
	1&0&\cdots&a_{N-r-1}&\cdots&0&0&0\\
	\end{pmatrix}
		  \begin{tikzpicture}[overlay, remember picture]
		\draw[stealth-] (a.east)++(1,0) -- node[above] {$r+1$ row }++ (1,0);
		\draw[stealth-] (b.north)++(0,0.1) -- node[right] {$N-r$ column }++ (0,0.5);
	\end{tikzpicture}
		\\
		A^{(+,r)} & = 
		\sum_{i=1}^N\left(\sum_{j=0}^{r} a_{i+j}\right) E_{i,i}
		+ E_{i,i+r+1} \, ,
		\label{eq:inbpMadd}
	\end{align}
	we recall that we are always considering periodic boundary conditions, 
	so for all $j \in\Z$, $a_{j+N} = a_j$ and $E_{i,j+N} =E_{i+N,j} = E_{i,j}$.
	In this notation, the equations of motion \eqref{eq:inbpadd} are equivalent to 
	
	\begin{equation}
	    \dot{L}^{(+,r)} = [L^{(+,r)};A^{(+,r)}]\,.
	\end{equation}

	Analogously, the multiplicative INB $r$-lattices have a Lax Pair formulation, which reads
	\begin{align}
		\label{eq:inbpLmul}
		L^{(\times,r)} & = \sum_{i=1}^{N}\left( a_{i} E_{i,i+1} 
		+E_{i+r,i} \right),\\
		& =	\begin{pmatrix}
	0&a_1&0&\cdots&\tikz[remember picture]\node[inner sep=0pt] (b) {$1$};&0&0&0\\
	0&0&a_2&\cdots&0&1&0&0\\
	0&0&0&a_3&\cdots&0&1&0\\
	\vdots&\ddots&\ddots&\ddots&\ddots&\ddots&&\\
	1&0&\cdots&\cdots&0&a_r&0&\tikz[remember picture]\node[inner sep=0pt] (a) {0};\\
	0&1&0\cdots&&\ddots&\ddots&\ddots&\\
	\vdots&\ddots&\ddots&\ddots&\ddots&0&0&a_{N-1}\\
	a_N&0&\cdots&1&\cdots&0&0&0\\
	\end{pmatrix}
		  \begin{tikzpicture}[overlay, remember picture]
		\draw[stealth-] (a.east)++(1,0) -- node[above] {$r+1$ row }++ (1,0);
		\draw[stealth-] (b.north)++(0,0.1) -- node[right] {$N-r$ column }++ (0,0.5);
	\end{tikzpicture}\\
		\label{eq:inbpMmul}
		A^{(\times,r)} & = \sum_{i=1}^{N} \left(\prod_{j=0}^{r} a_{i+j}\right) E_{i,i+r+1}\,.
	\end{align}
	
	Following the construction made in \cite{mazzuca_int_sys}, where the authors numerically computed the density of states for these two families of lattices, we introduce the generalized Gibbs ensemble for these models as
	
		\begin{align}
		    \label{eq:INB_add_gibbs}
		    \di \mu_{+,r} &= \frac{ \exp(- \Tr{P(L^{(+,r))}}) \prod_{j=1}^Na_j^{\alpha-1}\mathbbm{1}_{a_j\geq 0}\di \ba }{Z_N^{(+,r)}(a,P)}\,,\\
		    \label{eq:INB_mul_gibbs}
		    \di \mu_{\times,r} &= \frac{ \exp(- \Tr{P(L^{(\times,r))}}) \prod_{j=1}^Na_j^{\alpha-1}\mathbbm{1}_{a_j\geq 0}\di \ba }{Z_N^{(\times,r)}(\alpha,P)}\,,
		\end{align} 
	where $P$ is a polynomial of degree $j(r+1)$ for some $j\in \N$ as in Hypotheses \ref{polynomial}. We restrict to this potential in view of the following Lemma proved in \cite{mazzuca_int_sys}
		\begin{lemma}
		Fix $\ell\in \N$. Then for $N$ large enough
		\begin{equation}
			\Tr{(L^{(+,r)})^\ell} = \Tr{(L^{(\times,r)})^\ell} = 0\,,
		\end{equation}
		if $\ell$ is not an integer multiple of $r+1$.
	\end{lemma}
	 We notice that both $L^{(+,r)}$ and $L^{(\times,r)}$ are circular-like, and that the entries are all positive, so from Lemma \ref{LEM:CIRCULAR_TRACE}, we deduce the following:

	\begin{lemma}
		Fix $m\in \N$, and consider the matrices $L^{(+,r)},L^{(\times,r)}$ \eqref{eq:inbpLadd}-\eqref{eq:inbpLmul}. Then for $N$ big enough, there exist $k^{(+,r)}=k^{(+,r)}(m), k^{(\times,r)}=k^{(\times,r)}(m)\in \N$, such that 
    \begin{itemize}
        \item $\Tr{(L^{(+,r)})^m}$ is $k^{(+,r)}$-circular with seed $V^{(+,r)}$ and weed $\wt V^{(+,r)}$
        \item $\Tr{(L^{(\times,r)})^m}$ is $k^{(\times,r)}$-circular with seed $V^{(\times,r)}$ and weed $\wt V^{(\times,r)}$
    \end{itemize}

	\end{lemma}
	Thus, proceeding as we have done for the others systems previously considered, we obtain the following:
	
		\begin{corollary}[CLT for INB lattices]
	Consider the Lax matrices $L^{(+,r)},L^{(\times,r)}$ \eqref{eq:inbpLadd}-\eqref{eq:inbpLmul} of the additive and multiplicative INB lattices respectively distributed according to their Generalized Gibbs Ensemble \eqref{eq:INB_add_gibbs}-\eqref{eq:INB_mul_gibbs}. Then, defining the Free energies $\cF_{+,r}(\alpha,P),\cF_{\times,r}(\alpha,P)$ as
	
	\begin{align}
	    \cF_{+,r}(\alpha,P)= & -\lim_{N\to\infty}\frac{1}{N}\ln(Z^{(+,r)}_N(\alpha,P))\,, \\
	    \cF_{\times,r}(\alpha,P)= & -\lim_{N\to\infty}\frac{1}{N}\ln(Z^{(\times,r)}_N(\alpha,P))\,, \\
	\end{align}

	for all $j\in\N$ fixed, we have the following weak limit
	
	\begin{equation}
	    \lim_{N\to\infty}  \frac{\Tr{(L^{(+,r)})^{(r+1)j}} - \meanval{\Tr{(L^{(+,r)})^{(r+1)j}}}  }{\sqrt{N}} \rightharpoonup \cN(0,\sigma^2_{+,r})\,,
	\end{equation}
	
	\begin{equation}
	    \lim_{N\to\infty}  \frac{\Tr{(L^{(\times,r)})^{(r+1)j}} - \meanval{\Tr{(L^{(\times,r)})^{(r+1)j}}} }{\sqrt{N}} \rightharpoonup \cN(0,\sigma^2_{\times,r})\,,
	\end{equation}

        where
        \begin{align}
            &\meanval{\Tr{(L^{(+,r)})^{(r+1)j}}} = i N\partial_t \cF_{+,r}(\alpha,P+itx^{(r+1)j})_{\vert_{t=0}}\,,\quad \sigma^2_{+,r} = \vert\partial^2_t \cF_{+,r}(\alpha,P+itx^{(r+1)j})_{\vert_{t=0}}\vert\,, \\
            & \meanval{\Tr{(L^{(\times,r)})^{(r+1)j}}} = i N\partial_t \cF_{\times,r}(\alpha,P+itx^{(r+1)j})_{\vert_{t=0}}\,,\quad \sigma^2_{\times,r} = \vert\partial^2_t \cF_{\times,r}(\alpha,P+itx^{(r+1)j})_{\vert_{t=0}}\vert\,.
        \end{align}
	
	\end{corollary}
	
	\begin{remark}
	We recall that in \cite{mazzuca_int_sys}, it was shown that the density of states for this model has support on the complex plane, but despite that all the moments of the Generalized Gibbs ensemble are reals. Furthermore, in this case, we lack a $\beta$-ensemble to compare with.
	\end{remark}

\section{Transfer Operator Method}
\label{sec:transfer_operator}
In this section, we develop a transfer operator approach to obtain the asymptotic behaviour of the following integrals

\begin{align}
\label{eq:type_1_int}
    & I_1^{(N)}(H) =  \int_{X^N} \left(\prod_{j=1}^N F(x_j,\alpha)\right) e^{-H(\bx)}\di \bx\, ,\\
    \label{eq:type_2_int}
    & I_2^{(N)}(H) = \int_{X^N} \left(\prod_{j=1}^{N-1}  F\left(x_j,\alpha\left(1-\frac{j}{N}\right)\right)\right) R(x_N)  e^{-H(\bx)}\di \bx \,,
\end{align}
where  $F,H,X$ satisfies the following set of hypotheses, which are closely related to  hypotheses \ref{general_assumptions}

\begin{assumption}
\label{transfer_hp}
In this section, we assume the following hypotheses for $F,H,X$:
\begin{enumerate}
 \item[HP. 1] $X$ is a convex subset of $\R^d$, $d=1,2$ (we identify $\C$ with $\R^2$).
    \item[HP 2.] $F(x,\alpha) \, : \, X \times (0,\infty)\to \R_+$ is such that for any $\alpha>0$ $F(\cdot,\alpha) \in C^1(X)$, and for any $x\in X, F(x,\cdot)\in C^\infty((0,+\infty))$ ;
        \item[HP 3.] $H\,:\, X^N \to \R$ is $k$-circular with seed $h$ and weed $\wt h$ which are both bounded away from $-\infty$.
    \item[HP 4.] for all $\alpha\in (0,\infty)$, $F(\cdot,\alpha) \in L^{1}(X)$, and $[c,d]\subset (0,+\infty)$ one can find $g_{c,d}\in L^2(X)$ such that for all $c\leq \alpha\leq d$, $\left| \sqrt{F(\cdot,\alpha)}\right|,\ \left|\partial_\alpha \sqrt{F(\cdot,\alpha)}\right|\leq g_{c,d}$; moreover,  there exist a $\tc\in\N$ and $\varepsilon_0$ such that for all $\varepsilon<\varepsilon_0$ there exists a compact set $\mathcal{O}_\varepsilon\subseteq X$ and $d_1,d_2,d_3>0$, depending on $\varepsilon$, such that
    \begin{itemize}
        \item $||F(\cdot,\alpha)||_1 = d_1 \alpha^{-\tc}(1 + o(1))$ 
        \item  $||F(\cdot,\alpha)||_{1,\mathcal{O}_\varepsilon} = \int_{\mathcal{O}_\varepsilon }\vert F(x,\alpha) \vert \di x = d_2\alpha^{-\tc}(1 + o(1) )$
        \item  $||F(\cdot,\alpha)||_{1,\mathcal{O}^c_\varepsilon} \leq d_3$ here $\mathcal{O}_\varepsilon^c = X \setminus \mathcal{O}_\varepsilon$
        \item There exists a continuous function $w\, :\, X^k \to \R$ such that for $\bx,\by \in \mathcal{O}_\varepsilon$
        \begin{equation}
            h(\bx,\by) = w(\bx) + w(\by) + o(\varepsilon)\,.
        \end{equation}
        \item If $\varepsilon_1<\varepsilon_2$ then $\mathcal{O}_{\varepsilon_1}\subseteq \mathcal{O}_{\varepsilon_2}$
    \end{itemize}

    \item[HP 5.] $R$ is a distribution defining a probability measure on a subset of $X$.
\end{enumerate}
\end{assumption}

We want to study the asymptotic behaviour of the integrals $I_1^{(N)},I_2^{(N)}$ \eqref{eq:type_1_int}-\eqref{eq:type_2_int} which are a generalization of the partition functions of the type 1-2 systems at hand. In particular, we use the knowledge of the asymptotic behaviour of these integrals to compute the two ratios

\begin{equation}
    \mathbb{E}_1\left[e^{-it\Re\Tr{L^s}}\right] = \frac{Z^{(1)}_{N}(\alpha, \Re P + it \Re z^s)}{Z^{(1)}_{N}(\alpha, \Re P)},\quad \mathbb{E}_2\left[e^{-it \Re\Tr{T^s}}\right] = \frac{Z^{(2)}_{N}(\alpha, \Re P + it\Re z^s)}{Z^{(2)}_{N}(\alpha, \Re P)} \,,
\end{equation}
which is one of the key steps of our proof.
Under these assumptions, we prove the following Theorem:

\begin{theorem}
\label{thm:transfer_operator}
    Under hypotheses \ref{transfer_hp}. Consider a $k$-circular function $U \, :\, X^N \to \R $ with seed $u$ and weed $\wt u$. Then,
    \begin{equation}
        \label{eq:ratios}
        J_1^{(N)} = \frac{I_1^{(N)}(H + it U)}{I_1^{(N)}(H)}  \,,\quad  J_2^{(N)} = \frac{I_2^{(N)}(H+itU)}{I_2^{(N)}(H)}
    \end{equation}
    are finite for all fixed $N$. Furthermore, assume that for some $d\geq 3$, $|u|^a\exp\left(-h\right)$ is bounded for any $1\leq a \leq d$. Form the Euclidean division $N=kM+\ell$. Then, there exists an $\varepsilon>0$, and three complex-valued functions  $\lambda\in C^{1,d}(\R_+^*\times (-\varepsilon,\varepsilon))$, $c_{k,\ell}\in C^{1,d}(\R_+^*\times(-\varepsilon,\varepsilon))$ and $ c_{k,\ell,M}\in C^{1,d}(\R_+^*\times (-\varepsilon,\varepsilon))$ such that for all $q\in \N$ :
\begin{align}
\label{eq:prodPartition_tot}
    & J_1 =  c_{k,\ell}(\alpha,t)\lambda(\alpha,t)^{M-2}\left( 1 + o(M^{-q})\right)\,, \quad \text{as } M\to\infty\,, \\
    & J_2 = c_{k,\ell,M}(\alpha,t)\prod_{j=1}^{M-2} \lambda\left(\alpha\frac{j}{M},t\right)\left( 1+ o_M(1)\right) \,, \quad \text{as } M\to\infty\,, 
\end{align}
for $\vert t\vert <\varepsilon$. Furthermore, we have for all $\alpha>0$
\begin{align}
\label{eq:lambda_easy}
        &\bullet \lambda(\alpha,0) = 1 \,,\\
        \label{eq:first_regime_easy}
        & \bullet c_{k,\ell}(\alpha,0) = 1\,,\\
        \label{eq:first_regime}
        & \bullet \lim_{t \to 0} c_{k,\ell,M}(\alpha,t) = 1 \, \quad \text{uniformly in }M\,, \\
        & \bullet \text{the remainder } o_M(1)\text{ is uniformly bounded in }t\in(-\varepsilon, \varepsilon).
\end{align}
Moreover, there exist three functions $\wt c_{k,\ell}\in C^{1,d}(\R_+^*\times (-\varepsilon,\varepsilon))$, $\wt \lambda\in C^{1,d}(\R_+^*\times (-\varepsilon,\varepsilon))$ and $\wt c_{k,\ell,M}\in C^{1,d}(\R_+^*\times (-\varepsilon,\varepsilon))$ such that there exist four constants $C_1^{(1)},C_2^{(1)},C_1^{(2)},C_2^{(2)} >0$ and $p\in\N$ such that for all $q\in \N$:

\begin{align}
    &\bullet C^{(1)}_1<\wt c_{k,\ell}(\alpha,t) < C^{(1)}_2 \,,\\
    &\bullet C_1^{(1)} N^p < \wt c_{k,\ell,M}(\alpha,t) < C^{(2)}_2 N^p \,, \\
    & \bullet \lambda(\alpha,t) = \frac{\wt \lambda(\alpha,t)}{\wt \lambda(\alpha,0)} \,,\\ 
    \label{eq:Z_periodic}
    &\bullet I_1^{(N)}(H+itU)= \wt c_{k,\ell}(\alpha,t)\wt \lambda(\alpha,t)^{M-2}\left( 1 + o(M^{-q})\right)\,.\\
    &\bullet I_2^{(N)}(H+itU) = \wt c_{k,\ell,M}(\alpha,t)\prod_{j=1}^{M-2}\wt \lambda\left(\alpha\frac{j}{M},t\right)\left( 1 + o_M(1)\right)\,.
\end{align}
\end{theorem}

The previous theorem is the central tool that we use to prove Theorem \ref{THM:FINAL_CLT_INTRO} in Section \ref{section:proof_of_the_result}. Since the proof is rather technical, we defer it to Section \ref{section:technical}.

\section{Nagaev--Guivarc'h theory}
\label{section:generalTheory}

In this section, we adapt some standard results from the fluctuation theory of Nagaev--Guivarc'h \cite{Gouzel2015} to our situation. Specifically, we prove the following:

 \begin{theorem}
\label{thm:nagagui_thm}
Let $(X_n)_{n\geq 1}$ be a sequence of real random variables with partial sums $(S_n)_{n\geq 1}$. 
Assume that there exists $\varepsilon>0$, two functions $\lambda\in  C^1([0,\varepsilon))$, $c\in C^0([0,\varepsilon))$  and $h_n\in C^0([0,\varepsilon))$, such that for all $t\in[-\varepsilon,\varepsilon]$, and all $n\geq 1$ we have

\begin{equation}
\label{eq:startingpoint}
       \meanval{e^{-itS_n}} = c(t)\lambda(t)^m\left(1 + h_n(t) \right)\,,
\end{equation}

Where $\lim_{n\to\infty} n/m = k\in \N_{\geq 1}$.

Furthermore, assume that:
\begin{itemize}
    \item[a.] there exists $A,\,\sigma^2\in \C$ such that $\lambda(t) = \exp\left( -iAt - \sigma^2t^2/2 +o(t^3)\right)$ as $t\to 0$;
    \item[b.] $h_n \xrightarrow{n\to\infty}0$ uniformly in $[-\varepsilon,\varepsilon]$, and $h_n(0) = 0$;
    \item[c.] $c(0) = 1$.
\end{itemize}

Then $A\in \R,\,\sigma^2\geq 0$, and $\left(S_n -nA/k\right)/\sqrt{n}$ converges to a Gaussian distribution $\cN(0,\sigma^2/k)$ as $n$ tends to infinity.
\end{theorem}

\begin{proof}
First, evaluating \eqref{eq:startingpoint} at $t=0$, we deduce that $\lambda(0) = 1$. Then, we use the asymptotic expansion of $\lambda(t)$, and properties b.-c. to prove that 
$$\meanval{\exp\left(-it \frac{S_n -mA}{n} \right)} \xrightarrow[n\to\infty]{} 1\,. $$
Thus, by L\'evy theorem \cite{Williams1991}, we deduce that $S_n/n - A/k$ converges in distribution to $0$. So, since $S_n$ is real, then $A\in \R$. Exploiting again the asymptotic expansion of $\lambda(t)$ and properties b.-c., we show that $\meanval{\exp\left(- it\frac{S_n -mA}{\sqrt{n}} \right)}$ converges to the function $\exp\left(-\frac{\sigma^2t^2}{2k}\right)$. By L\'evy theorem \cite{Williams1991}, this must be the characteristic function of a real random variable, proving that $\sigma^2\geq 0$,  and that $(S_n-nA/k)/\sqrt{n}$ converges to a Gaussian distribution $\mathcal{N}(0,\sigma^2/k)$. 
\end{proof}

Further, we prove the following: 
\begin{theorem}
\label{thm:nagagu_thm_mod}
Let $(X_n)_{n\geq 1}$ be a sequence of random variables with partial sums $(S_n)_{n\geq 1}\in \R$. 
Assume that there exists $\varepsilon>0$ and functions $(x,t)\mapsto\lambda(x,t)\in C^{1,0}([0,1)\times\R)$, $c_n\in C^0(\R)$, and $h_n$ continuous at $0$, such that for all $t\in[-\varepsilon,\varepsilon]$, and all $n\geq 1$ we have

\begin{equation}
\label{eq:startingpoint_mod}
       \meanval{e^{-itS_n}} =    c_n(t)\left(\prod_{j=1}^m \lambda\left(j/m,t\right)\right)\left(1+h_n(t) \right)\,,
\end{equation}
where $\lim_{n\to \infty} n/m = k\in \N_{\geq 1}$.

Furthermore, assume that:
\begin{itemize}
    \item[a.] there exists two continuous functions $A,\,\sigma^2\,: [0,1] \to \C$ such that 
    
    $\lambda(x,t) = \exp\left( -iA(x)t - \sigma^2(x)t^2/2 +o(t^2)\right)$ as $t\to 0$;
    \item[b.] $||h_n||_\infty \xrightarrow{n\to\infty}0$ uniformly in $[-\varepsilon,\varepsilon]$, and $h_n(0) = 0$;
    \item[c.] $c_n(0) = 1$ and $\lim_{n\to\infty} c_n(t/\sqrt{n}) = \lim_{n\to\infty} c_n(t/n) = 1$.
\end{itemize}

Then $\int_0^1A(x)\di x \in \R, \, \int_0^1\sigma^2(x)\di x\in \R_+$, and $\sqrt{n}\left(\frac{S_n}{n} -\frac{\int_0^1 A(x)\di x}{k}\right)$ converges to a Gaussian distribution $\cN\left(0,\frac{\int_0^1\sigma^2(x)\di x}{k}\right)$ as $n$ tends to infinity.
\end{theorem} 

\begin{proof}

First, let $t = \wt t /n$, then by hypothesis c. 
\begin{equation}
       \lim_{n\to \infty} \meanval{e^{-i\wt t\left(S_n/n - \frac{1}{n} \sum_{\ell =1}^m A\left(\ell/m\right)\right) }} = \lim_{n\to \infty} \meanval{e^{-i\wt t\left(S_n/n -\frac{\int_0^1 A(x)\di x}{k}\right) }} = 1   \,.
\end{equation}
Thus, by Levy theorem, $S_n/n \to \frac{\int_{0}^1 A(x)\di x}{k}$ almost surely, thus, since $S_n\in \R$, this implies that $\int_0^1A(x)\di x \in \R$. Consider now $t = \wt t/\sqrt{n}$, following the same reasoning one conclude that

\begin{equation}
\begin{split}
           \lim_{n\to \infty} \meanval{e^{-i\wt t\sqrt{n}\left(S_n/n - \frac{1}{n} \sum_{\ell =1}^n A\left(\ell/m\right)\right) }} & = \lim_{n\to \infty} \meanval{e^{-i\wt t\sqrt{n}\left(S_n/n - \frac{\int_0^1A(x)\di x}{k}\right) }} \\ & =\lim_{n\to \infty} e^{-\frac{\wt t^2}{2n}\sum_{\ell=1}^m \sigma^2(\ell/m)} = e^{-\frac{\wt t^2}{2k}\int_0^1 \sigma^2(x)\di x}\,,
\end{split}
\end{equation}
thus,  by Lévy theorem \cite{Williams1991}, $e^{-\frac{\wt t^2}{2k}\int_0^1 \sigma^2(x)\di x}$ must be the characteristic function of a real random variable, proving that $\int_0^1 \sigma^2(x)\in \R_+$.
\end{proof}

\section{Proof Of Theorem \ref{THM:FINAL_CLT_INTRO}-\ref{THM:SUSCEPTIBILITY} }
\label{section:proof_of_the_result}
We are ready to prove Theorem \ref{THM:FINAL_CLT_INTRO}-\ref{THM:SUSCEPTIBILITY}. For convenience, we split the proof into two Lemmas, which combined give the full proof of our results.

\begin{lemma}
\label{thm:final_clt}
Under hypotheses \ref{general_assumptions}. Consider $\mu^{(1)}_{N},\, \mu^{(2)}_{N}$ \eqref{type1}-\eqref{type2}, and let $s\in \N$.
Assume that $\Re \Tr{P(L)} + it \Re \Tr{L^s}$ is a $k$-circular function, that $W,u$ are the compatible seeds of $\Re \Tr{P(L)}, \Tr{L^s}$ and that $|u|^a e^{-h}$ is bounded for $1\leq a \leq 2$.  Then, there exist four continuous functions 

\begin{align}
    &A\,:\, \R_+ \longrightarrow \R\,, \qquad  \wt A\,:\, \R_+ \longrightarrow \R\,, \\
    &\sigma\, : \, \R_+ \longrightarrow \R_+\,, \qquad \wt \sigma\, : \, \R_+ \longrightarrow \R_+\,,
\end{align}
such that under $\mu^{(1)}_{N}$, 
$$\left(S_{N} -NA(\alpha)/k\right)/\sqrt{N}$$ converges to a Gaussian distribution $\cN(0,\sigma^2(\alpha)/k)$ as $N$ tends to infinity, and under $\mu^{(2)}_{N}$,   $$\left(S_{N} -N\wt A(\alpha)/k\right)/\sqrt{N}$$ converges to a Gaussian distribution $\cN(0,\wt \sigma^2(\alpha)/k)$ as $N$ tends to infinity.
\end{lemma}

The proof of the previous result is a trivial application of  Theorem \ref{thm:transfer_operator}-\ref{thm:nagagui_thm}-\ref{thm:nagagu_thm_mod} and of Lemma \ref{LEM:CIRCULAR_TRACE} where $H$ is the polynomial $\Re P$.
Furthermore, we can interpret the previous relations through the \textit{free energies} of $\mu^{(1)}_{N}, \,\mu^{(2)}_{N}$ \eqref{type1} -\eqref{type2}:

\begin{lemma}
\label{lem:free_energy}
Under the same hypotheses and notation of Lemma \ref{thm:final_clt}. Consider the two measures $\mu^{(1)}_{N}, \,\mu^{(2)}_{N}$  \eqref{type1}-\eqref{type2}, and define the free energies as

\begin{equation}
    \cF^{(1)}(\alpha,\Re P) = - \lim_{N\to \infty} \frac{\ln(Z_{N}^{(1)}(\alpha,\Re P))}{N}\,,\quad  \cF^{(2)}(\alpha,\Re P) = - \lim_{N\to \infty} \frac{\ln(Z_{N}^{(2)}(\alpha,\Re P))}{N}\,,
\end{equation}
then, using the same  notation as in Lemma \ref{thm:final_clt}, the following holds:
\begin{multicols}{2}
\begin{enumerate}
    \item[i.] $ \cF^{(1)}(\alpha,\Re P) =  \partial_\alpha \left( \alpha \cF^{(2)}(\alpha,\Re P)\right)$
    \item[ii.] $A(\alpha) = ik\partial_{t}\cF^{(1)}(\alpha,\Re P +it\Re z^s )_{\vert_{t=0}}$
    \item[iii.] $\wt A(\alpha) = ik\partial_{t}\cF^{(2)}(\alpha,\Re P+it\Re z^s)_{\vert_{t=0}}$ 
    \item[iv.] $\sigma(\alpha) = k\partial^2_{t}\cF^{(1)}(\alpha,\Re P+it\Re z^s)_{\vert_{t=0}}$
    \item[v.] $\wt \sigma(\alpha) = k\partial^2_{t}\cF^{(2)}(\alpha,\Re P +it\Re z^s)_{\vert_{t=0}}$
\end{enumerate}
\end{multicols}

\end{lemma}
\begin{remark}
\label{rem:derivative}
    The previous theorem implies that
    \begin{equation}
        A(\alpha) = \partial_\alpha(\alpha \wt A(\alpha))\,, \qquad \sigma^2(\alpha) = \partial_\alpha(\alpha \wt \sigma^2(\alpha))\,.
    \end{equation}
\end{remark}

\begin{proof}
To prove $i.$, we can just compute the free energy of $\mu^{(1)}_{N},\, \mu^{(2)}_{N}$ using Theorem \ref{thm:transfer_operator}. For $\cF^{(1)}(\alpha,\Re P)$ we deduce immediately that

\begin{equation}
\label{eq:free_energies}
\begin{split}
        \cF^{(1)}(\alpha,\Re P)  & = - \lim_{M\to \infty} \frac{1}{N} \ln(Z^{(1)}_{N}(\alpha, \Re P)) \stackrel{\eqref{eq:Z_periodic}}{=} - \lim_{M\to \infty} \frac{M}{N} \ln(\wt \lambda(\alpha,0)) \\ & = - \frac{1}{k}\ln(\wt \lambda(\alpha,0))\,.
\end{split}
\end{equation}
Where we recall that $N = kM + \ell$.
The computation for $\cF^{(2)}(\alpha,\Re P)$ follows in the same way, establishing \textit{i.}. 

We now prove $ii.-iv$. First, we notice that following the notation of Theorem \ref{thm:transfer_operator}  - \ref{thm:nagagui_thm}:

\begin{equation}
    c(t) = c_{k,\ell}(\alpha,t)\,, \quad h_n = 1+o(M^{-q})\,, \quad \lambda(t) = \lambda(\alpha,t)\,,
\end{equation}
thus to compute $A(\alpha),\, \sigma(\alpha)$ we expand $\lambda(\alpha,t)$ around $t=0$ 
\begin{equation}
\label{eq:eig_expansion}
         \lambda(\alpha,t) = 1 + t\partial_t\lambda(\alpha,0) + \frac{t^2}{2}\partial_t^2\lambda(\alpha,0) + o(t^3)\,,
\end{equation}
which implies that
\begin{equation}
\label{eq:exp_expansion}
    \lambda(\alpha,t) = \exp\left(t\partial_t \ln(\lambda(\alpha,t))_{\vert_{t=0}} + \frac{t^2}{2}\partial^2_t \ln(\lambda(\alpha,t))_{\vert_{t=0}} + o(t^3)\right)\,.
\end{equation}
This implies that $A(\alpha) = -i\partial_t \ln(\lambda(\alpha,t))_{\vert_{ t=0}}\,, \sigma^2(\alpha) = -\partial^2_t \ln(\lambda(\alpha,t))_{\vert_{ t=0}} $. From \eqref{eq:Z_periodic}, we deduce that $\partial_t \ln(\lambda(\alpha,t))_{\vert_{ t=0}} = \partial_t \ln(\wt \lambda(\alpha,t))_{\vert_{ t=0}}$, thus from the previous expressions and the explicit form of the free energy \eqref{eq:free_energies} we conclude.

To prove $iii.-v.$ we proceed in the same way, thus following the notation of Theorem \ref{thm:transfer_operator}-\ref{thm:nagagu_thm_mod}: 

\begin{equation}
    c_n(t) =  c_{k,\ell,M}(\alpha,t)\,, \quad h_n = 1+o(1)\,,\quad \lambda\left( j/M,t\right) = \lambda\left(\alpha\frac{j}{M},t\right) \,.
\end{equation}

Thus, as in \eqref{eq:exp_expansion} except that $\alpha \to \alpha\frac{j}{M}$, we expand $\lambda\left(\alpha\frac{j}{M},t\right) $ around $t=0$, leading to 
\begin{equation}
\begin{split}
    \wt A(\alpha) & = \partial_t  \left(
    \int_0^1\ln(\lambda(\alpha x,t))\di x \right)_{\vert_{t=0}}\,,\\
    \wt \sigma^2(\alpha) &= \partial^2_t \left( \int_0^1\ln(\lambda(\alpha x,t))\di x\right)_{\vert_{t=0}}
\end{split}
\end{equation}
which concludes the proof.
\end{proof}

\begin{remark}
We notice that the Lemma \ref{lem:free_energy}, and Lemma \ref{thm:final_clt}  imply  that we can compute the expected values, and the variances of $\Tr{L^s}$ according to $\mu^{(1)}_{N},\, \mu^{(2)}_{N}$ just computing derivatives of the corresponding free energy. This property is broadly used in the physics literature, but we lacked a precise proof of  the general result. Furthermore, we can compute the expected value, and the variance of $\Tr{L^s}$ according to $\mu^{(1)}_{N}$ starting from the corresponding values for $\mu^{(2)}_{N}$. Thus, we have reduced all this problem to the computation of the free energy of $\mu^{(2)}_{N}$. 
\end{remark}
The proof of both Theorem \ref{THM:FINAL_CLT_INTRO} and Theorem \ref{THM:SUSCEPTIBILITY} follows from the four previous lemmas. Thus, we have completed the proof of our main theorems.

\section{Technical Results}
 \label{section:technical}
 
 In this section, we prove the technical results that we used to prove our main Theorems \ref{THM:FINAL_CLT_INTRO}-\ref{THM:SUSCEPTIBILITY}, the proof follows the same line as the proof of \cite[Proposition 4.2]{mazzuca2021generalized}, and we prove Theorem \ref{thm:decay}, whose proof uses the same machinery as the latter proofs. In the last part, we prove a Berry-Esseen type bound for the type 1 measure $\mu_N^{(1)}$. We start by proving  Theorem \ref{thm:transfer_operator}.

To prove these results, we follow the same ideas as in \cite[Theorem 2.4]{Gouzel2015}, which we apply, for $k\geq 1$, to a family $(\mathcal{L}_t)_{t\in I}$ of compact operators $\cL_t:L^2(X^k)\to L^2(X^k)$, where $I$ is an interval containing a neighbourhood of $0$. Denoting for $u\in L^p(X^k)$ $\|u\|_{p}= \left(\int_{X^k} |u(\bx)|^p\di \bx\right)^{\frac{1}{p}}$, 
we endow the space of linear bounded operators of $L^2(X^k)$, denoted by $\mathbf{L}(L^2(X^k))$ with the norm
$$ \|\cL\| := \sup_{u\in L^2(X^k),\ u\neq 0}\frac{\|\cL u\|_{2}}{\|u\|_{2}}\,.$$

\begin{definition}
    Let $I\subset \R$ be open. We say that the $\mathbf{L}(L^2(X^k))$-valued function $t\in I \mapsto \cL_t$ is $a$-times continuously differentiable, and denote it by $\cL\in C^{a}(\mathbf{L}(L^2(X^k)))$ if $\cL$ has a Taylor expansion of order $a$ around any $t\in I$, $\textit{i.e.}$
    $$ \cL_{t+h} = \cL_t + hD^{(1)}_t + \frac{h^2}{2}D^{(2)}_t + \ldots + \frac{h^a}{a!}D^{(a)}_t + o(h^a) \qquad \text{as } h\to 0\,,$$
    with $t \in I \mapsto D_t^{(a)} \in \mathbf{L}(L^2(X^k))$ continuous, where the $o(h^a)$ is with respect to the norm $\|\cdot \|$. The operator $D^{(j)}_t$ is then by definition the $j$-th derivative of $\cL_t$: $\partial_t^{(j)} \cL_t :=D_t^{(j)}$.
\end{definition}
\begin{remark}
    We define in the same way $C^a$ functions with values in $L^2(X^k)$.
\end{remark}

We enforce the following proposition, which can be easily deduced from \cite[Proposition 2.3]{Gouzel2015}: 

\begin{proposition}
\label{prop:gauezel}
Let $\lambda(0)$ be an isolated eigenvalue of the operator $\cL_0$ with multiplicity one, and assume that the family of operators $t\in I\mapsto \cL_{t}\in \left(\mathbf{L}(L^2(X^k)),\|\cdot\|\right)$ depends on $t$ in a $C^d$ way. Then, for $t$ small enough, $\cL_t$ has an isolated eigenvalue $\lambda(t)$ with multiplicity one, and $\lambda(t)$, the corresponding eigenprojection $\pi_t$ and its eigenfunction $\varphi_t$ are $C^d$ with respect to $t$.

Moreover, assume that the rest of the spectrum of $\cL_0$ is contained in a disk of radius $|\lambda(0)| - \delta $. Writing $\cQ_t = (I-\pi_t)\cL_{t}$, so that for $n\geq 1$ $\cL_{t}^n = \lambda(t)^n\pi_t + \cQ_t^n$. For any $r> |\lambda(0)|-\delta$, there exists a constant $C>0$ independent of $t,n$ such that  $\norm{\cQ^n_t} \leq Cr^n $ for all $n\in \mathbb{N}$.
\end{proposition}

\begin{remark}
    We notice that in the previous proposition the size of the set of $t$ such that the statements holds is not specified: it is an existence result. One of the crucial step of the following analysis is to  ensure that such set is uniform with respect to some parameters.
\end{remark}

In the sequel, for $I,J \subset \R$ open, and for $E$ a normed space, we will denote by $C^{a,b}(I\times J, E)$ the set of functions with values in $E$ which are $C^a$ with respect to the first parameter, and $C^b$ with respect to the second parameter. We will denote $C^{a,b}(I\times J)$ for functions with values in $\C$.

The following families of kernel operators play a crucial role in the proof of the main theorems of this section.

\begin{definition}
    \label{def:operators}
    Under Hypotheses \ref{transfer_hp}. Consider a $k$-circular function $U \, :\, X^N \to \R $ with seed $u$, such that $u$ is compatible with $h$, the seed of $H$.
    
    For $t\in \R$, $\alpha>0$ define
    \begin{itemize}
    \item $\cL_{t,\alpha}\in \mathbf{L}(L^2(X^k)$ as
    \begin{equation}
    \label{def:kernelOp}
    \cL_{\alpha,t}f(\by) = \int_{X^k} f(\bx)\sqrt{\prod_{q=1}^kF(x_q,\alpha) F(y_q,\alpha)}e^{-h(\by,\bx)-itu(\by,\bx)}\di \bx=: \int_{X^k}k_{\alpha,t}(\bx,\by)f(\bx) \di \bx\,.
    \end{equation}
    \item With the writing $N=kM+\ell$, $0\leq \ell< k$, for $1\leq j \leq M-1$ define $\cL_{t,\alpha}^{(j)}\in \mathbf{L}(L^2(X^k)$ as
    \begin{equation*}
    \cL_{\alpha,t}^{(j)}f(\by) = \wt F (\by,j+1)\int_{X^k} f(\bx)\wt F (\bx,j)e^{-h(\by,\bx)-itu(\by,\bx)}\di \bx\,,
    \end{equation*}
    where 
    \begin{equation*}
        \wt F (\bx,j) =  \sqrt{ \prod_{q=1}^kF\left(x_q,\alpha\left(1-\frac{(j-1)k+q}{N}\right)\right)}\,.
    \end{equation*}
    \end{itemize}
\end{definition}

For a kernel operator $T\in \mathbf{L}(L^2(X^k))$
$$ (Tu)(\by) = \int_{X^k} k(\bx,\by)u(\bx)\di \bx\,$$
where $k:X^k\times X^k \to \R$, define its Hilbert-Shmidt norm as 
$$ \| T \|_{\mathrm{HS}}^2 = \iint_{X^{2k}} k(\bx,\by)^2\di \bx\di \by\,.$$
Then, $\|T\|\leq \| T \|_\mathrm{HS}$. Furthermore, if $T$ is a kernel operator with finite Hilbert-Schmidt norm, then it is compact, see \cite{KatoBook}.

\begin{lemma}
\label{lem:differentiability}
    Under assumptions \ref{transfer_hp}. For $\alpha>0$, $t\in \R$, consider the operators $\cL_{\alpha,t},\cL_{t,\alpha}^{(j)}$, $1\leq j \leq M-1$ given by Definition \ref{def:operators}. Assume that for some $d\geq 2$, $|u|^a\exp\left(-h\right)$ is bounded for any $1\leq a \leq d$.
    
    Then, those operators have finite Hilbert-Schmidt norm, and $\cL_{\alpha,t},\cL_{t,\alpha}^{(j)} \in C^{1,d}((0,+\infty)\times \R, \mathbf{L}(L^2(X^k)))$, \textit{i.e.}
    \begin{itemize}
        \item[1)] For any fixed $t\in \R$, the $\mathbf{L}(L^2(X^K))$-valued functions $\alpha\mapsto \cL_{\alpha,t}$ and $\alpha\mapsto\cL_{\alpha,t}^{(j)}$ are continuously differentiable. 
        \item[2)] For any fixed $\alpha>0$, $t\mapsto \cL_{\alpha,t}$ and $\alpha\mapsto \cL_{\alpha,t}^{(j)}$ are $d$-times continuously differentiable.
    \end{itemize}
    Furthermore, for all $\alpha>0$, define 
    $$ \cB_{\alpha}=\frac{1}{\|\cL_{\alpha,0}\|_\mathrm{HS}}\cL_{\alpha,0} $$
    and 
    $$ \cB_\alpha^{(j)}=\frac{1}{\|\cL_{\alpha,0}^{(j)}\|_\mathrm{HS}}\cL_{\alpha,0}^{(j)},\ 1\leq j\leq M-1\,.$$
    Then, $\cB_{\alpha}$ (resp. $\cB_\alpha^{(j)}$)
    has a dominant, simple eigenvalue $\mu(\alpha)>0$ (resp. $\mu^{(j)}(\alpha)$). Moreover, for all $A>0$ there exists $\delta>0$ such that and all $\alpha\in (0,A]$, any other eigenvalue $\mu$ satisfies $|\mu-\mu(\alpha)|\geq \delta$ (resp. $|\mu-\mu^{(j)}(\alpha)|\geq \delta$). We say that those operators have a \textbf{uniform spectral gap} in $\alpha\in (0,A]$.
\end{lemma}

We defer the proof of this lemma to Appendix \ref{App:proof differentiability}.

Applying the previous proposition, we can prove both asymptotic estimates of Theorem \ref{thm:transfer_operator}. For convenience, we split the theorem into two; the first one is related to the proof of the result involving $I_N^{(1)}$, and the second $I_N^{(2)}$:

\begin{theorem}
    \label{thm:easy_asymptotic}
    Under hypotheses \ref{transfer_hp}. Consider a $k$-circular function $U \, :\, X^N \to \R $ with seed $u$ and weed $\wt u$, such that $u$ is compatible with $h$, the seed of $H$. Define
    \begin{equation}
        \label{eq:ratios}
        J_1^{(N)} = \frac{I_1^{(N)}(H + it U)}{I_1^{(N)}(H)}
    \end{equation}
     where $I_1^{(N)}$ is defined as
    
    \begin{equation}
    \label{eq:type_1_partition}
         I_1^{(N)}(H) =  \int_{X^N} \left(\prod_{j=1}^N F(x_j,\alpha)\right) e^{-H(\bx)}\di \bx\, .
    \end{equation}
    Assume that for some $d\geq 2$, $|u|^a\exp\left(-h\right)$ is bounded for any $1\leq a \leq d$. Form the Euclidean division $N=kM+\ell$. Then, there exists an $\varepsilon>0$, and two complex-valued functions  $\lambda\in C^{1,d}(\R_+^*\times (-\varepsilon,\varepsilon))$, $c_{k,\ell}\in C^{1,d}(\R_+^*\times(-\varepsilon,\varepsilon))$ such that for all $q\in \N$ :
\begin{equation}
    \label{eq:prodPartition_1}
    J_1 =  c_{k,\ell}(\alpha,t)\lambda(\alpha,t)^{M-2}\left( 1 + o(M^{-q})\right)\,, \quad \text{as } M\to\infty\,, \\
\end{equation}
for $\vert t\vert <\varepsilon$. Furthermore, for any $\alpha >0$
\begin{align}
\label{eq:lambda_easy}
        &\bullet \lambda(\alpha,0) = 1 \,,\\
        \label{eq:first_regime_easy}
        & \bullet c_{k,\ell}(\alpha,0) = 1\,.
\end{align}
Moreover, there exist two functions $\wt c_{k,\ell}\in C^{1,d}(\R_+^*\times (-\varepsilon,\varepsilon))$, and $\wt \lambda\in C^{1,d}(\R_+^*\times (-\varepsilon,\varepsilon))$  such that there exist two constants $C_1^{(1)},C_2^{(1)} >0$ such that for all $q\in \N$:
\begin{align}
    &\bullet C^{(1)}_1< |\wt c_{k,\ell}(\alpha,t)| < C^{(1)}_2 \,,\\
    & \bullet \lambda(\alpha,t) = \frac{\wt \lambda(\alpha,t)}{\wt \lambda(\alpha,0)} \,,\\ 
    \label{eq:Z_periodic}
    &\bullet I_1^{(N)}(H+itU)= \wt c_{k,\ell}(\alpha,t)\wt \lambda(\alpha,t)^{M-2}\left( 1 + o(M^{-q})\right)\,.\\
\end{align}
\end{theorem}

\begin{proof}
For $\alpha>0$ and $t\in \R$, consider $\cL_{t,\alpha}\,:\,L^2(X^k) \to L^2(X^k)$ given by Definition \ref{def:operators}. By Lemma \ref{lem:differentiability}, each $\cL_{t,\alpha}$ is Hilbert-Schmidt, and in particular, compact. Moreover, since the kernel $k_{\alpha,0}$ \eqref{def:kernelOp} is positive, we apply a generalization of Jentzsch's theorem (showing that $\cL_{\alpha,0}$ has a simple, positive, dominant eigenvalue) \cite[Theorem 137.4]{ZaanenBook} in combinations with Lemma \ref{lem:differentiability} and Proposition \ref{prop:gauezel} to deduce that:

\begin{itemize}
    \item There exists $\varepsilon>0$ such that for $|t|<\varepsilon$, $\cL_{\alpha,t}$ has a simple, dominant eigenvalue $\wt\lambda(\alpha,t)$ with associated eigenfunction $\varphi(\cdot, \alpha,t)\in L^2(X^k)$, $\|\varphi(\cdot,\alpha,t)\|_{2}=1$,
    \item Its dominant eigenvalue $\wt \lambda$ varies smoothly with respect to $(\alpha,t)$: $\wt \lambda \in C^{1,d}(\R_+^* \times (-\varepsilon,\varepsilon))$,
    \item The associated eigenfunction $\varphi(\cdot,\alpha,t)$ varies smoothly: $(\alpha,t)\mapsto\varphi(\cdot,\alpha,t) \in C^{1,d}(\R_+^*\times (-\varepsilon,\varepsilon),L^2(X^k))$,
\end{itemize}
and there exists an operator $\cQ_t\,:\, L^2(X^k)\to L^2(X^k) $ such that

\begin{equation}
\label{eq:decompositon_j_1}
    \cL_{\alpha,t}\phi(\by) = \wt \lambda\left( \alpha, t \right) \left\langle \phi, \varphi\left(\cdot, \alpha, t\right) \right\rangle\varphi\left(\by, \alpha, t\right) + \cQ_t\phi(\by)\,, \quad \forall\,\phi\in L^2(X^k)\,,
\end{equation}
denoting by $\la\cdot, \cdot\ra$ the standard scalar product in $L^2(X^k)$. 

Furthermore, at $t=0$ we have $\wt  \lambda\left( \alpha, 0\right) >0$, $\varphi(\cdot,\alpha,t)>0$ almost everywhere, and  there exists a $\delta >0$ such that for all $t\in (-\varepsilon,\varepsilon)$, one has $\vert\vert \cQ_t \vert \vert \leq \vert\wt \lambda\left( \alpha, t \right)\vert - \delta$.

For $\bz \in X^k$ define $G_{\bz}(\by)$ as 

\begin{equation}
\label{eq:G}
    G_\bz(\by) = \begin{cases} 
    \begin{aligned}
        &\sqrt{\prod_{q=1}^k F(y_q,\alpha)}
    \int_{X^\ell} \prod_{j=kM+1}^{kM+\ell}F(x_j, \alpha)\\ &\times \exp\left(-(\wt h+it\wt u)(\by, x_{kM+1},\ldots, x_{kM+\ell},\bz)\right)\prod_{j=kM+1}^{kM+\ell}\di x_j\,,
    \end{aligned} & \quad \ell > 0\,,\\
    \sqrt{\prod_{q=1}^k F(y_q,\alpha)}\exp\left(-h(\by,\bz)-itu(\by,\bz)\right)\,, &\quad \ell =0\,,
    \end{cases}
\end{equation}
and the linear operator $\cS \;: L^2(X^k) \to \C$ as

\begin{equation}
\label{eq:S}
    \cS \varphi = \int_{X^k} \prod_{j=1}^{k}\sqrt{F(x_j, \alpha)} \varphi(\bx) \di \bx\,.
\end{equation}
We notice that 
$$\norm{\cS}:= \sup_{\|\varphi\|_{2}=1}\left|\cS\varphi \right| \leq \norm{F}_{1}^{k} <+\infty$$
therefore $\cS$ is continuous.

In this notation, we can recast \eqref{eq:type_1_partition}, applying $\cS$ to  $(\bx,\by) \mapsto \left(\cL^{M-1}_{\alpha,t}G_\bx\right)(\by)$, as

\begin{equation}
\begin{split}
         I_1^{(N)}(\alpha,H +itU) &= \cS\left( \left(\cL_{\alpha,t}^{M-1}G_{\bx_M}\right) (\bx_M)\right) \\ &= \wt \lambda^{M-1}(\alpha,t)\cS\big(\left \langle\varphi\left(\cdot, \alpha, t\right) ; G_{\bx_M}(\cdot)\right\rangle\varphi(\bx_M,\alpha,t)\big) + \cS\big(\cQ_t^{M-2}G_{\bx_M}(\bx_M)\big)\,,
\end{split}
\end{equation}
where here and in the sequel, if $f\in L^2(X^k\times X^k)$, using a slight notational liberty, we denoted $\cS(f(\bx)$ for $\cS(f)$.
Defining 
\begin{equation}
\begin{split}
    c_{k,\ell}(\alpha,t) &= \frac{\cS\big(\left \langle\varphi\left(\cdot, \alpha, t\right) ; G_{\bx_M}(\cdot)\right\rangle\varphi(\bx_M,\alpha,t)\big)}{\cS\big(\left \langle\varphi\left(\cdot, \alpha, 0\right) ; G_{\bx_M}(\cdot)\right\rangle\varphi(\bx_M,\alpha,0)\big)}\,, \\
    \lambda(\alpha, t) & = \frac{ \wt \lambda^{M-1}(\alpha,t)}{ \wt \lambda^{M-1}(\alpha,0)}\,.
\end{split}
\end{equation}
Because $\varphi(\cdot,\alpha,0)>0$ almost everywhere, the denominator in $c_{k,\ell}$ is positive, and by direct computation, using that $h,\wt h$ are bounded from below, one sees that 
$$ \cS\big(\left \langle\varphi\left(\cdot, \alpha, t\right) ; G_{\bx_M}(\cdot)\right\rangle\varphi(\bx_M,\alpha,t)\big)\leq \|F(\cdot,\alpha)\|_{1}^{2k+\ell}\,.$$
In view of Proposition \ref{prop:gauezel}, $\norm{\cQ_t^n}\leq (\vert \wt\lambda(t) \vert - \delta)^n$, and we conclude. 
\end{proof}

We prove now another theorem related to the expansion of $I_2^{(N)}$, which, combined with the previous one, conclude the proof of Theorem \ref{thm:transfer_operator}

\begin{theorem}
    \label{thm:hard_asymptotic}
    Under hypotheses \ref{transfer_hp}. Consider a $k$-circular function $U \, :\, X^N \to \R $ with seed $u$ and weed $\wt u$, such that $u$ is compatible with $h$, the seed of $H$. Define
    \begin{equation}
        \label{eq:ratios_2}
        J_2^{(N)} = \frac{I_2^{(N)}(H +it U)}{I_2^{(N)}(H)}
    \end{equation}
    where $I_2^{(N)}$ is defined as
    
    \begin{equation}
    \label{eq:type_1_partition}
         I_2^{(N)}(H) = \int_{X^N} \left(\prod_{j=1}^{N-1}  F\left(x_j,\alpha\left(1-\frac{j}{N}\right)\right)\right) R(x_N)  e^{-H(\bx)}\di \bx \, .
    \end{equation}
    Assume that for some $d\geq 2$, $|u|^a\exp\left(-h\right)$ is bounded for any $1\leq a \leq d$. Form the Euclidean division $N=kM+\ell$. Then, there exists an $\varepsilon>0$, and a complex-valued $ c_{k,\ell,M}\in C^{1,d}(\R_+^*\times (-\varepsilon,\varepsilon))$ such that \textbf{with the same} $\boldsymbol{\lambda(\alpha,t)}$\textbf{ as in Theorem \ref{thm:easy_asymptotic}},
\begin{equation}
\label{eq:prodPartition_2}
     J_2^{(N)}= c_{k,\ell,M}(\alpha,t)\prod_{j=1}^{M-2} \lambda\left(\alpha\frac{j}{M},t\right)\left( 1+ o_M(1)\right) \quad \text{as } M\to\infty\,, 
\end{equation}
for $\vert t\vert <\varepsilon$. Furthermore, for any $\alpha >0$,
\begin{align}
\label{eq:lambda_easy}
        &\bullet \lambda(\alpha,0) = 1 \,,\\
        \label{eq:first_regime}
        & \bullet \lim_{t \to 0} c_{k,\ell,M}(\alpha,t) = 1 \, \quad \text{uniformly in }M\,, \\
        & \bullet \text{the remainder } o_M(1)\text{ is uniformly bounded in }t\in(-\varepsilon, \varepsilon).
\end{align}
Moreover, there exist two functions  $\wt \lambda\in C^{1,d}(\R_+^*\times (-\varepsilon,\varepsilon))$ and $\wt c_{k,\ell,M}\in C^{1,d}(\R_+^*\times (-\varepsilon,\varepsilon))$ such that there exist two constants $C_1^{(2)},C_2^{(2)} >0$ and $p\in\N$ such that:

\begin{align}
    &\bullet C_1^{(1)} N^p < |\wt c_{k,\ell,M}(\alpha,t)| < C^{(2)}_2 N^p \,, \\
    & \bullet \lambda(\alpha,t) = \frac{\wt \lambda(\alpha,t)}{\wt \lambda(\alpha,0)} \,,\\ 
    &\bullet I_2^{(N)}(H+itU) = \wt c_{k,\ell,M}(\alpha,t)\prod_{j=1}^{M-2}\wt \lambda\left(\alpha\frac{j}{M},t\right)\left( 1 + o_M(1)\right)\,.
\end{align}
\end{theorem}

\begin{proof}
	Let $W(\bx) = H(\bx) + it U(\bx)$, and let $w\, :\, X^{2k}\to \C$ be its seed, and  $\wt w\, :\, X^{2k+\ell}\to \C$ be  its weed.
    For $1\leq j\leq M-1$, define the function $\wt F (\bx,j) \,:\, X^k\times \N \to \R_+$ as

    \begin{equation}
        \wt F (\bx,j) =  \sqrt{ \prod_{q=1}^kF\left(x_q,\alpha\left(1-\frac{(j-1)k+q}{N}\right)\right)}\,,
    \end{equation}
    we notice that $\wt F (\bx,j)\in L^2(X^k)$. 
    Define the family of kernel operators,  $\cL_{t,\alpha}^{(j)}\,:\,L^2(X^k) \to L^2(X^k)$ as 
\begin{equation}
    \cL_{t,\alpha}^{(j)}f(\by) = \wt F (\by,j+1)\int_{X^k} f(\bx)\wt F (\bx,j)e^{-w(\by,\bx)}\di \bx\,.
\end{equation}

By Lemma \ref{lem:differentiability}, $(\alpha,t)\mapsto \cL_{\alpha,t}^{(j)}$ is $C^{1,d}$. As in the proof of Theorem \ref{thm:easy_asymptotic}, since for $t=0$ the kernel is positive,  apply a generalization of Jentzsch's theorem \cite[Theorem 137.4]{ZaanenBook} in combinations with Proposition \ref{prop:gauezel} and Lemma \ref{lem:differentiability} deducing that:
\begin{itemize}
    \item By the existence of a uniform spectral gap (Lemma \ref{lem:differentiability}), there exists some $\varepsilon>0$ \textbf{independent of }$\boldsymbol{j}$ (but depending on $\alpha$) such that for $|t|<\varepsilon$, $\cL_{\alpha,t}^{(j)}$ has a simple, dominant eigenvalue $\widehat\lambda\left(\alpha\left(1-\frac{j}{M}\right),t\right)$ with associated eigenfunction $\varphi\left(\cdot,\alpha\left(1-\frac{j}{M}\right),t\right)\in L^2(X^k)$ with unit $L^2$ norm.
    \item $\hat{\lambda}\in C^{1,d}(\R_+^*\times (-\varepsilon,\varepsilon))$
    \item The associated eigenfunction varies smoothly: $(\alpha,t)\mapsto \varphi\left(\cdot, \alpha\left(1-\frac{j}{M}\right),t\right)\in C^{1,d}(\R_+^*\times(-\varepsilon,\varepsilon), L^2(X^k))$\,,
\end{itemize}
and there exists an operator $\cQ_t^{(j)}\,:\, L^2(X^k)\to L^2(X^k) $ such that $\forall\,\phi\in L^2(X^k)$, $\forall |t|<\varepsilon$,
\begin{equation}
\label{eq:decompositon_j}
    \cL_{t,\alpha}^{(j)}\phi(\by) = \widehat \lambda\left(\alpha\left(1-\frac{j}{M}\right),t\right)\pi_t^{(j)}\phi(\by)+\cQ_t^{(j)}\phi(\by)
\end{equation}
with
$$ \pi_t^{(j)}\phi(\by) = \left\langle\phi; \varphi\left(\cdot,\alpha\left( 1 - \frac{j}{M}\right),t\right)\right\rangle\varphi\left(\by,\alpha\left( 1 - \frac{j}{M}\right),t\right)\,, $$
recalling that we denote by $\la\cdot, \cdot\ra$ the standard scalar product in $L^2(X^k)$. 

Furthermore, we have $\widehat\lambda\left(\alpha\left(1-\frac{j}{M}\right),0\right)>0$, $\varphi\left(\cdot,\alpha\left(1-\frac{j}{M}\right),0\right)>0$ almost everywhere,
and there exists $\delta>0$ such that $\norm{ \cQ_t^{(j)}} \leq \left\vert \widehat\lambda\left( \alpha\left( 1 - \frac{j}{M}\right), t \right)\right\vert - \delta$.

Now, with $\wt\lambda$ the function of Theorem \ref{thm:easy_asymptotic}, we have
\begin{equation}
    \label{eq:approxEigen}
    \widehat\lambda\left(\alpha\left(1-\frac{j}{M}\right),t\right)=\wt\lambda\left(\alpha\left(1-\frac{j}{M}\right),t\right) + O\left(\frac{1}{M}\right)\,,
\end{equation}
Where the $O\left(\frac{1}{M}\right)$ term is uniform in $t\in (-\varepsilon,\varepsilon)$. Indeed, recalling that  $\cL_{t,\alpha}$ is defined in \eqref{def:kernelOp},  by the integrability assumptions on $F(\cdot,\alpha)$, $\partial_\alpha \sqrt{F}$ (see Assumptions \ref{transfer_hp}), we have
\begin{equation}
\label{ineq:approxOperator}
    \norm{\cL^{(j)}_{t,\alpha}-\cL_{t,\alpha(1-j/M)} } \leq C\frac{\alpha}{M}
\end{equation}
where $C\geq 0$ is independent of $t$. We then deduce \eqref{eq:approxEigen} by applying Proposition \ref{prop:gauezel}.

Analogously as the previous proof, define the function $G_{\bz}(\by ,t)$ on $X^k$ by


\begin{equation}
\label{eq:G_2}
    G_\bz(\by,t) = \begin{cases} 
    \begin{aligned}
        &\wt F(\by,M)
    \int_{X^{\ell}} \prod_{j=1}^{\ell-1} F\left(x_{kM+j},\alpha\left(1 - \frac{j+kM}{N}\right) \right)\\ &\times \exp\left(-\wt w(\by, x_{kM+1},\ldots, x_{N},\bz)\right)R(x_N)\prod_{j=kM+1}^{kM+\ell}\di x_j\,,
    \end{aligned} & \quad \ell > 0\,,\\
   \wt F(\by,M)\int_X \exp\left(-w(\by,\bz)\right)R(z_k)\di z_k \,, &\quad \ell =0\,,
    \end{cases}
\end{equation}

We recall that $W=H+itU$ so $G_\bz(\by,t)$ depends on $t$ , and the linear operator  $\cS \, :\, L^2(X^k) \to \C$ as

\begin{equation}
     \cS(\psi) = \int_{X^k} \wt F(\bx,1)\psi(\bx) \prod_{j=1}^k\di x_{j}\,.
\end{equation}

 Given the assumptions \ref{transfer_hp}, the operator $\cS$ is uniformly bounded in $M$.
 In this notation, we can rewrite $I_2^{(N)} (H+itU)$ as

 \begin{equation}
     I_2^{(N)} (H+itU )= \cS\left(\cL_{t,\alpha}^{(2)}\ldots\cL_{t,\alpha}^{(M-1)}G_{\bx_1}(\bx_M,t)\right)=\cS\left(\prod_{j=2}^{M-1}\cL_{t,\alpha}^{(j)}G_{\bx_1}(\bx_M,t)\right)
 \end{equation}

Applying the decomposition \eqref{eq:decompositon_j}, it follows that we can decompose the previous expression as

\begin{equation}
    \label{eq:decompL}
    \begin{split}
        I_2^{(N)}(H+itU) &=\prod_{j=2}^{M-1} \widehat\lambda_t\left(\alpha\left(1-\frac{j}{M}\right)\right) \cS\left( \pi_t^{(2)}\ldots\pi_t^{(M-1)}G_{\bx_1}(\bx_M,t) \right) 
        \\&+ \cS\left( \cL_{t,\alpha}^{(2)}\ldots\cL_{t,\alpha}^{(M-2)}\cQ_t^{(M-1)}G_{\bx_1}(\bx_M,t) \right) 
        \\&+ \sum_{n=2}^{M-2} \prod_{j=n+1}^{M-1}\widehat\lambda_t\left(\alpha\left(1-\frac{j}{M}\right)\right) \cS\left( \cL_{t,\alpha}^{(2)}\ldots\cL_{t,\alpha}^{(n-1)}\cQ_t^{(n)}\pi_t^{(n+1)}\ldots\pi_t^{(M-1)}G_{\bx_1}(\bx_M,t)\right)\,,
    \end{split}
\end{equation}
where we arranged the terms of the product of the $\cL^{(j)}_{t,\alpha}$'s by order of the first appearance from the right of a factor $\cQ^{(j)}_t$ (the first term being the product where no $\cQ^{(j)}_t$ appears). We notice that
$$ \cS\left( \pi_t^{(2)}\ldots\pi_t^{(M-1)}G_{\bx_1}(\bx_M,t) \right)= \braket{\varphi_t^{1/M};G_{\bx_1}(\bx_M,t)}\prod_{i=2}^{M-2}\braket{\varphi_t^{\alpha(1-i/M)};\varphi_t^{\alpha(1-(i+1)/M)}}\cS\left(\varphi_t^{\alpha\left(1-\frac{2}{M}\right)}\right)\,,$$
where we set $\varphi_t^{\alpha(1-i/M)} \equiv \varphi\left(\cdot,\alpha(1-i/M),t \right)$ to shorten the notation.
Furthermore, the ratio 
$$ \prod_{i=2}^{M-1} \frac{\langle \varphi_t^{\alpha\left(1-\frac{i}{M}\right)};\varphi_t^{\alpha\left(1-\frac{i+1}{M}\right)} \rangle}{\langle \varphi_0^{\alpha\left(1-\frac{i}{M}\right)};\varphi_0^{\alpha\left(1-\frac{i+1}{M}\right)} \rangle} $$
stays bounded as $M\to \infty$  uniformly in $t\in (-\varepsilon,\varepsilon)$.  This is due to the fact that 
$$ \langle \varphi_t^{\alpha\left(1-\frac{i}{M}\right)};\varphi_t^{\alpha\left(1-\frac{i+1}{M}\right)} \rangle = 1+O\left(\frac{\alpha}{M}\right),$$
because of \eqref{ineq:approxOperator} and Proposition \ref{prop:gauezel}, thus the product
$$ \prod_{i=2}^{M-1} \langle \varphi_t^{\alpha\left(1-\frac{i}{M}\right)};\varphi_t^{\alpha\left(1-\frac{i+1}{M}\right)} \rangle$$
stays from bounded below and above uniformly on $M\geq 1$, $t\in (-\varepsilon,\varepsilon)$.

Denoting the first term of \eqref{eq:decompL} by $f(\alpha,t)$, and the second and third terms by $g_1(\alpha,t)$ and $g_2(\alpha,t)$,
we can rewrite  $J_2^{(N)}$ \eqref{eq:ratios_2} as

\begin{equation}
   J_2^{(N)}(H+itU)= \frac{f(\alpha,t)}{f(\alpha,0)}\left( \frac{1 + \frac{g_1(\alpha,t)}{f(\alpha,t)} + \frac{g_2(\alpha,t)}{f(\alpha,t)}}{1 + \frac{g_1(\alpha,0)}{f(\alpha,0)} + \frac{g_2(\alpha,0)}{f(\alpha,0)}}\right)\,.
\end{equation}
Thus, to prove our result we need to show that there exist $3$ constants $c_1,c_2,c_3$ independent of $M$ such that for all $t\in (-\varepsilon,\varepsilon)$,

\begin{align}
\label{eq:estimate1}
    &\left \vert \frac{g_1(\alpha,t)}{f(\alpha,t)} \right\vert \leq c_1\,, \\
    \label{eq:estimate2}
    &\left \vert \frac{g_2(\alpha,t)}{f(\alpha,t)} \right\vert \leq c_2\,, \\
    \label{eq:estimate3}
    & \left \vert \frac{ \left\langle \varphi^{(1/M)}_t; G_{\bx_1}(\bx_M,t) \right\rangle\cS\varphi_t^{\alpha\left(1 - \frac{2}{M}\right)}}{\left\langle \varphi_0^{(1/M)}; h_0 \right\rangle\cS\varphi_0^{\alpha\left(1 - \frac{2}{M}\right)} } \right \vert \leq c_3\,.
\end{align}
If we are able to show this, then defining

\begin{equation}
\begin{split}
    c_{k,\ell,M}(\alpha,t) &= \frac{\left\langle \varphi^{(1/M)}_t; G_{\bx_1}(\bx_M,t) \right\rangle \cS\varphi_t^{\alpha\left(1 - \frac{2}{M}\right)}}{\left\langle \varphi_0^{(1/M)}; h_0 \right\rangle \cS\varphi_0^{\alpha\left(1 - \frac{2}{M}\right)}}\prod_{i=2}^{M-1} \frac{\langle \varphi_t^{\alpha\left(1-\frac{i}{M}\right)};\varphi_t^{\alpha\left(1-\frac{i+1}{M}\right)} \rangle}{\langle \varphi_0^{\alpha\left(1-\frac{i}{M}\right)};\varphi_0^{\alpha\left(1-\frac{i+1}{M}\right)} \rangle}\,, \\
    \lambda(y,t) &= \frac{\wt\lambda\left( y, t \right)}{\wt\lambda\left( y, 0 \right)}\,,
\end{split}
\end{equation}
we obtain \eqref{eq:prodPartition_2} with the wanted properties. Notice that in the definition of $\lambda$ we took $\dfrac{\wt\lambda(y,t)}{\wt\lambda(y,0)}$ instead of $\dfrac{\widehat\lambda(y,t)}{\widehat\lambda(y,0)}$. This is indeed possible because of equation \eqref{eq:approxEigen}.

First, we focus on \eqref{eq:estimate2}. The term $g_2(\alpha,t)$ is given by
$$\braket{\varphi^{(1/M)};G_{\bx_1}(\bx_M,t)}\sum_{n=2}^{M-2}\prod_{j=n+1}^{M-1}\widehat\lambda\left( \alpha\left(1-\frac{j}{M}\right),t \right)\prod_{i=1}^{M-n-2}\braket{\varphi_t^{\frac{\alpha i}{M}};\varphi_t^{\frac{\alpha (i+1)}{M}}}\cS\left(\cL_{t,\alpha}^{(2)}\ldots\cL_{t,\alpha}^{(n-1)}\cQ_t^{(n)}\varphi_t^{\alpha\left(1-\frac{n+1}{M}\right)}\right)\,.$$

Because $\varphi\left(\bx, y, t\right)$ is regular with respect to $y$, we deduce that there exists a function $(y,t)\mapsto\psi(\cdot, y,t)\in C^\infty(\R^+\times[-\varepsilon,\varepsilon],L^2(X^k))$ with $\norm{\psi_t^{\alpha(1-n/M)}}_2$ uniformly bounded in $n,M$ and $t$ such that

\begin{equation}
    \cQ_t^{n}\left(\varphi_t^{\alpha\left(1-\frac{n+1}{M}\right)}\right) = Q_t^{(n)}\left(\varphi_t^{\alpha\left(1-\frac{n}{M}\right)}\right)+\frac{1}{M}\cQ_t^{(n)}\psi_t^{\alpha(1-n/M)}= \frac{1}{M}\cQ_t^{(n)}\psi_t^{\alpha(1-n/M)} \,.
\end{equation}
given this equality, it is trivial to prove \eqref{eq:estimate2}, recalling that for any $t,j$, $\dfrac{1}{\widehat\lambda_t^{(j)}}\cL_{t,\alpha}^{(j)}$ has operator norm smaller than one.

For \eqref{eq:estimate1},  it suffices to show that there exists a constant $c_2$ independent of $M$ such that

\begin{equation}
\label{eq:small_bound}
    \left\| \frac{\cQ_t^{(M-1)}G_{\bx_1}(\bx_M,t)}{ \left\langle \varphi^{(1/M)}_t, G_{\bx_1}(\bx_M,t) \right\rangle}\right\| \leq c_2\,.
\end{equation}

From the assumptions, \eqref{eq:decompositon_j} and the definition of $G_{\bx_1}(\bx_M,t)$ \eqref{eq:G_2}, we deduce that there exists a constant $d_1$ such that

\begin{equation}
    \left\|\cQ_j^{(M-1)}G_{\bx_1}(\bx_M,t)\right\| \leq d_1\left( \lambda\left(\frac{\alpha}{M} \right) -\delta\right)M^{\tc (k+\ell)}\,.
\end{equation}
On the other hand, given the previous proof and the assumptions, we conclude that, for $t$ small enough, there exists a constant $d_2$ such that

\begin{equation}
\label{ineq:lowerbound}
    \left|\left\langle \varphi^{(1/M)}_t, G_{\bx_1}(\bx_M,t) \right\rangle\right| \geq d_2 M^{\tc (k+\ell)}\,.
\end{equation}
Indeed, for $t=0$ the previous inequality follows from the definition of $ G_{\bx_1}(\bx_M,t)$, and we deduce that for $t$ small enough the same holds true

Combining the two previous estimates, and setting $p=\tc(k+\ell)$ we deduce \eqref{eq:small_bound}, which leads to \eqref{eq:estimate2}. The proof of \eqref{eq:estimate3} is analogous, thus we conclude.
\end{proof}

We now turn on the proof of Theorem \ref{thm:decay}, which we rewrite here for convenience.
\begin{theorem}[Decay of correlations]
    Under hypotheses  \ref{general_assumptions}. Let $W$ be the seed of $\Tr{\Re P(L)}$  and $I,J : X^k \to \R $ two local functions such that   $\int_{X^k\times X^k}\left|I(\bx)\prod_{i=1}^k \sqrt{F(x_i,\alpha)F(y_i,\alpha)}e^{-W(\bx,\by)}\right\vert^2\di \bx \di \by < \infty $, and analogously for $J(\bx)$. Write $N=kM+\ell$, and let $j\in \{1,\ldots,M\}$. Then there exists some $0<\mu<1$ such that
    $$ \E_{1}\left[I(\bx_1)J(\bx_j)\right]-\E_{1}\left[I(\bx_1)\right]\E_{1}\left[J(\bx_j)\right]=O(\mu^{M-j}+\mu^j)\,.$$
\end{theorem}

\begin{proof}
Let $\cL=\cL_{0,\alpha}$ with $\cL_{t,\alpha}$ given by \eqref{def:kernelOp}.  
Furthermore, define $\mathcal{L}^{(J)}$
$$\mathcal{L}^{(J)}\phi(\by)=\int_{X^k}\phi(\bx)\prod_{i=1}^k \sqrt{F(x_i,\alpha)F(y_i,\alpha)} J(\bx)e^{-W(\by,\bx)}= \cL(J\phi)(\by)\,,$$
    and $\mathcal{L}^{(I)}$ analogously. 
With $G^{(I)}_\bx(\by)=I(\bx)G_\bx(\by)$, $G_\bx(\by)$ given in \eqref{eq:G}, we have for $j\geq 3$

\begin{align*}
    \E_{1}\left[I(\bx_1)J(\bx_j)\right]&=\frac{\mathcal{S}_0\left(\left(\mathcal{L}^{M-j}\mathcal{L}^{(J)}\mathcal{L}^{j-3}G^{(I)}_{\bx_M}\right)(\bx_M)\right)}{\mathcal{S}_0((\mathcal{L}^{M-2}G_{\bx_M})(\bx_M))}\\
    &=\frac{\wt\lambda^{M-j}(\alpha,0)\mathcal{S}_0((\pi_0\mathcal{L}^{(J)}\mathcal{L}^{j-3}G^{(I)}_{\bx_M})(\bx_M))+O(\wt\lambda^{j-3}r^{M-j}))}{\wt\lambda^{M-2}(\alpha,0)\mathcal{S}_0((\pi_0 G_{\bx_M})(\bx_M))+O(r^{M-2})}\,,
\end{align*}
where $\cS_t$ is defined in \eqref{eq:S}, and we used the decomposition
$$\mathcal{L}^k_0=\wt \lambda^k(\alpha,0)\pi_0+\mathcal{Q}_0^k\,,$$
where $\pi_0$ is the orthogonal projection on the (one dimensional) eigenspace associated with $\wt \lambda(\alpha,0)$, and $\cQ_0$ is an operator such that
$\|\mathcal{Q}_0^k\|\leq Cr^k$ for some $0<r<\wt\lambda$. Similarly, 
$$ \mathcal{S}_0((\pi_0\mathcal{L}^{(J)}\mathcal{L}^{j-3}G^{(I)}_{\bx_M})(\bx_M)) = \wt\lambda^{j-3}(\alpha,0)\mathcal{S}_0((\pi_0\mathcal{L}^{(J)}\pi_0 G^{(I)}_{\bx_M})(\bx_M)) + O(r^{j-3}).$$
We deduce
$$\E_{1}\left[I(\bx_1)J(\bx_j)\right]=\frac{\mathcal{S}_0((\pi_0\mathcal{L}^{(J)}\pi_0 G^{(I)}_{\bx_M})(\bx_M)) + O((r/\wt\lambda)^{M-j}+(r/\wt\lambda)^{j-3})}{\wt\lambda \left(\mathcal{S}_0((\pi_0 G_{\bx_M})(\bx_M)) + O((r/\wt\lambda)^{M-2})\right)}\,.$$
Similarly, we deduce

\begin{equation}
    \E_{1}[I(\bx_1)]\E_{1}[J(\bx_j)]= \frac{\mathcal{S}_0((\pi_0 G^{(I)}_{\bx_M})(\bx_M))\cS_0\left( (\pi_0\cL^{(J)}\pi_0 G_{\bx_M})(\bx_M)\right)+O((r/\wt\lambda)^{M-j}+(r/\wt\lambda)^{j-3})}{\wt\lambda\left(\mathcal{S}_0((\pi_0 G_{\bx_M})(\bx_M))^2+O((r/\wt\lambda)^{M-2})\right)}\,.
\end{equation}

By a direct computation, recalling that $\pi_0 \phi = \braket{\varphi_1,\phi}\varphi_1$ where $\varphi_1$ is the eigenfunction associated with $\wt\lambda$, we deduce the following 
$$ \cS_0\left(( \pi_0\cL^{(J)}\pi_0G_{\bx_M}^{(I)})(\bx_M)\right)=\braket{\cL^{(J)}\varphi_1,\varphi_1}\int\braket{G_{\bx},\varphi_1}I(\bx)F(\bx)\varphi_1(\bx)\di \bx\,,$$

$$ \cS\left( (\pi_0G_{\bx_M})(\bx_M) \right) = \int\braket{G_{\bx},\varphi_1}\varphi_1(\bx)F(\bx)\di \bx\,,\quad \cS_0\left( (\pi_0G^{(I)}_{\bx_M})(\bx_M) \right)=\int\braket{G_\bx,\varphi_1}I(\bx)F(\bx)\varphi_1(\bx)\di \bx\,,$$
and
$$ \cS_0\left(( \pi_0\cL^{(J)}\pi_0 G_{\bx_M})(\bx_M)\right) =  \braket{\cL^{(J)}\varphi_1,\varphi_1}\int\braket{G_{\bx},\varphi_1}F(\bx)\varphi_1(\bx)\di \bx\,.$$

These formulas imply that
$$\mathcal{S}_0((\pi_0\mathcal{L}^{(J)}\pi_0 G^{(I)}_{\bx_M})(\bx_M))=\frac{\mathcal{S}_0((\pi_0 G^{(I)}_{\bx_M})(\bx_M))\cS_0\left( (\pi_0\cL^{(J)}\pi_0 G_{\bx_M})(\bx_M)\right)}{\cS_0\left((\pi_0G_{\bx_M})(\bx_M)\right)}\,,$$
and so
$$ \E_{1}\left[I(\bx_1)J(\bx_j)\right]-\E_{1}\left[I(\bx_1)\right]\E_{1}\left[J(\bx_j)\right]=O((r/\wt\lambda)^{M-j}+(r/\wt\lambda)^{j-3})\,.$$
    
\end{proof}

Finally, we prove a Berry-Esseen bound type theorem for the measure $\mu_{N}^{(1)}$: 
\begin{theorem}
\label{thm:BerryEsseen}
Under Hypotheses \ref{general_assumptions}. Consider  $s\in \N$ and the measure $\mu_{N}^{(1)}$ \eqref{type1}, let $W,h$ be the common seeds of $\Tr{\Re P(L)}$ and $\Tr{\Re L^s}$ respectively and let $\wt W, \wt h$ be their weeds. Assume that for some $d\geq 3$, the functions $|u|^a e^{-h}$, $1\leq a \leq d$ are bounded.
Then, there exists $A\in \R$, $\sigma,C >0$ such that if $Y\sim\cN(0,\sigma^2)$ we have for any interval $J$ of the real line
\begin{equation}
    \left| \mathbb{P}\left( \left[\Tr{\Re P(L)} -N A \right]/\sqrt{N} \in J\right) - \mathbb{P}\left(Y\in J\right)\right| \leq \frac{C}{\sqrt{N}}\,.
\end{equation}
\end{theorem}

\begin{proof}

We adapt the arguments of \cite[Theorem 3.7]{Gouzel2015}. By \cite[Lemma XVI.3.2]{feller2}, there exists a constant $C$ such that for any $X$ real random variable, and $Y$ Gaussian random variable, for any interval $J\subset \R$ and for any $T>0$, we have 
$$ \left|\mathbb{P}\left(X\in J\right)-\mathbb{P}\left(Y\in J\right) \right| \leq C\int_0^T \frac{|\mathbb{E}[e^{-itX}]-e^{-\sigma^2t^2/2}|}{t}\di t +  \frac{C}{T}.$$
We take $X=\left(\Tr{\Re(L^s)} - N A \right)/\sqrt{N}$. We are going to show that, taking $T=\varepsilon \sqrt{N}$ for some small enough $\varepsilon$, the last integral remains bounded by $\frac{C_{k,\ell}}{\sqrt{N}}$, where $C_{k,\ell}$ is a constant depending on $k,\ell$. Recall $N=kM+\ell$.
By Theorem \ref{THM:FINAL_CLT_INTRO}, there exists an $A\in\R,\sigma >0$ such that as $N$ goes to infinity $X$ converges to $\cN(0,\sigma^2)$.  
Since $t^{-1}$ is not integrable at $0$, we consider the special interval $[0,N^{-1}]$. In this interval, we have the following estimate:

\begin{equation}
    \begin{split}
        \left| \meanval{e^{-itX}} - e^{-it\sqrt{N}A}\right| &\stackrel{\eqref{eq:ration_BE}}{=} \frac{\lvert Z_{N}^{(1)}\left(\alpha, \Re P +i\frac{t}{\sqrt{N}}\Re z^s\right) - Z_{N}^{(1)}\left(\alpha, \Re P\right)\rvert  }{Z_{N}^{(1)}\left(\alpha, \Re P\right)} \\ 
        &= \frac{1}{Z_{N}^{(1)}\left(\alpha, \Re P\right)}\Big|\sum_{p=1}^{M-1}\int_{X^{kM+\ell}}\mathfrak{F}(\bx)\prod_{j=1}^{p-1}e^{i\frac{t}{\sqrt{N}}h(\bx_j,\bx_{j+1})}\left(e^{i\frac{t}{\sqrt{N}}h(\bx_p,\bx_{p+1})} -1\right)\prod_{j=1}^{N}\di x_j
        \\ &+ \int_{X^{kM+\ell}}\mathfrak{F}(\bx)\left(e^{i\frac{t}{\sqrt{N}}\wt h(\bx_M,x_{kM+1},\ldots,x_{kM+\ell},\bx_1)}-1\right)\Big|\,,
    \end{split}
\end{equation}
with the convention that the empty product is equal to one.
Here we defined 

\begin{equation}
\begin{split}
    \mathfrak{F}(\bx) =& \prod_{j=1}^{N}F(x_j, \alpha)\exp\left(-\sum_{j=1}^{M-1}W(\bx_j,\bx_{j+1}) -\wt W(\bx_M,x_{kM+1},\ldots, x_{kM+\ell},\bx_1) \right)\\ 
    & \times \exp\left(-\frac{it}{\sqrt{N}} \wt h(\bx_M,x_{kM+1},\ldots, x_{kM+\ell},\bx_1) \right)\,,
\end{split}
\end{equation}
Thus, since $\lvert e^{i\frac{t}{\sqrt{N}}h(\bx_p,\bx_{p+1})} -1\rvert \leq \vert h(\bx_p,\bx_{p+1}) \vert N^{-1/2}t$, we deduce the following inequality 

\begin{equation}
\label{eq:small_set_2}
\left| \mathbb{E}_1[e^{-itX}] - e^{-it\sqrt{N}A}\right| \leq \mathbb{E}_1\left[\vert h(\bx_1,\bx_{2}) \vert\right] t\sqrt{N} + \frac{t}{\sqrt{N}}\mathbb{E}_1\left[\vert \wt h(\bx_M,x_{kM+1},\ldots,x_{kM+\ell},\bx_{1}) \vert\right]\,,
\end{equation}

and this last term is by assumption bounded by $Ct\sqrt{N}$ for some $C$ independent of $N$ and $t$. Thus integrating for $t\in[0,N^{-1}]$ we deduce the following
\begin{equation}
    \begin{split}
        \int_0^{\frac{1}{N}}& \frac{|\mathbb{E}_1[e^{-itX}]-e^{-\sigma^2t^2/2}|}{t}\di t \\
        &\leq \int_0^{\frac{1}{N}} \frac{\left| \mathbb{E}_1\left[e^{-itX}\right] - e^{-it\sqrt{N}A}\right| + \left| e^{-it\sqrt{N}A} - 1\right| + \left| 1- e^{-\sigma^2t^2/2}\right|}{t}\di t \\
        & \stackrel{\eqref{eq:small_set_2}}{\leq} \int_0^{\frac{1}{N}} \frac{C\sqrt{N} t + t\sqrt{N}A + \sigma^2t^2/2}{t}\di t \leq \frac{C_1}{\sqrt{N}}\,,
        \end{split}
\end{equation}
for some constant $C_1$.

We now consider the integral on $[1/N, \varepsilon\sqrt{N}]$. Here we use the spectral decomposition of $\mathbb{E}_1[e^{it X}]$. Since $|h|^a e^{-W}$ is bounded for $1\leq a \leq 3$, we deduce by Proposition \ref{prop:gauezel}, Remark \ref{rem:operatorRegularity} and Theorem \ref{THM:FINAL_CLT_INTRO}  the following third order expansion and that there exist two continuous functions $p(t)\in C^0([-\varepsilon,\varepsilon])$ and $c_{k,\ell}(y,t)\in C^{1,d}(\R\times[-\varepsilon,\varepsilon])$  for some $\varepsilon>0$, such that $c_{k,\ell}(y,0)=1$ and $\|p\|_{\infty}<+\infty$, and for $q\geq 1$

\begin{equation}
            \int_\frac{1}{N}^{\varepsilon\sqrt{N}} \frac{|\mathbb{E}_1[e^{-itX}]-e^{-\sigma^2t^2/2}|}{t}\di t = \int_\frac{1}{N}^{\varepsilon\sqrt{N}} \frac{\left| c_{k,\ell}(\alpha,t/\sqrt{N})e^{-\sigma^2t^2/2 + t^3 p(t/\sqrt{N})/\sqrt{N} }(1+o(N^{-q})) - e^{-\sigma^2 t^2/2}\right|}{t} \di t \,,
\end{equation}
 thus we have the following estimate

\begin{equation}
    \begin{split}
        \int_\frac{1}{N}^{\varepsilon\sqrt{N}} \frac{\left|\mathbb{E}_1[e^{-itX}]-e^{-\sigma^2t^2/2}\right|}{t}\di t & \leq \left\vert\left\vert c_{k,\ell}(\alpha,\cdot) (1+o(M^{-q})) \right\vert\right\vert_{\infty,[0,\varepsilon]}\int_\frac{1}{N}^{\varepsilon\sqrt{N}} \frac{\left|\left(1-e^{ t^3 p(t/\sqrt{N})/\sqrt{N} }\right)e^{-\sigma^2t^2/2}\right|}{t}\di t \\
        & + \int_{\frac{1}{N}}^{\varepsilon \sqrt{N}} \frac{\left|1-c_{k,\ell}(\alpha,t/\sqrt{N})(1+o(M^{-q}))\right|e^{-\sigma^2t/2}}{t}\di t\,,
    \end{split}
\end{equation}
where $||\cdot||_{\infty,[0,\varepsilon]}$ in the $L^\infty$ norm on $[0,\varepsilon]$.

We notice that $\left\vert \left\vert c_{k,\ell}(\alpha,\cdot) (1+o(M^{-q})) \right\vert \right\vert_{\infty,[0,\varepsilon]}$ is uniformly bounded in $N$. Moreover,
\begin{equation}
    \begin{split}
    \int_\frac{1}{N}^{\varepsilon\sqrt{N}} \frac{e^{-\sigma^2t^2/2}}{t}\left| e^{t^3 p(t/\sqrt{N})/\sqrt{N}} - 1 \right|\di t 
        & \leq \int_\frac{1}{N}^{\varepsilon\sqrt{N}} \frac{e^{-\sigma^2t^2/2}}{t\sqrt{N}} e^{t^3\|p\|_{\infty,[0,\varepsilon]}/\sqrt{N}} t^3 \|p\|_{\infty,[0,\varepsilon]}\di t  \\ & \leq  \int_\frac{1}{N}^{\varepsilon\sqrt{N}} \frac{e^{-\sigma^2t^2/2}}{\sqrt{N}}e^{t^2\varepsilon\|p\|_{\infty,[0,\varepsilon]}}t^2 \|p\|_{\infty,[0,\varepsilon]}\di t\,,
    \end{split}
\end{equation}
where in the first inequality we used the bound $|e^x -1 | \leq |x|e^{|x|}$.
Since for $\varepsilon$ small enough $||p||_\infty\varepsilon < \sigma^2/4$, thus integrating, we deduce that

\begin{equation}
    \int_\frac{1}{N}^{\varepsilon\sqrt{N}} \frac{\left| e^{-\sigma^2t^2/2 + t^3 p(t/\sqrt{N})/\sqrt{N} } - e^{-\sigma^2 t^2/2}\right|}{t} \di t  = O\left( \frac{1}{\sqrt{N}}\right)\,.
\end{equation}

To conclude, we have to show that the last integral is of order $N^{-1/2}$. Since $c_{k,\ell}(\alpha,t)$ is $C^1$ in $t$, and $c_{k,\ell}(\alpha,0) = 1$, it is easy to deduce that there exists a constant $C$ such that

\begin{equation}
    \int_{\frac{1}{N}}^{\varepsilon \sqrt{N}} \frac{\left|1-c_{k,\ell}(\alpha,t/\sqrt{N})(1+o(M^{-q}))\right|e^{-\sigma^2t/2}}{t}\di t \leq \frac{C}{\sqrt{N}}\,
\end{equation}
so we conclude.
\end{proof}

\section{Conclusion and Outlooks}
\label{section:conclusion}
In this paper, we proved a general Central Limit Theorem type result and we apply it to several models in random matrix theory and integrable systems. By doing this, we strengthen the connection between these two subjects. Specifically, we could connect the expected values and the variances of the moments of each classical $\beta$-ensemble in the high-temperature regime with one specific integrable model, see Table \ref{tab:relationsIntro}.

The results that we have obtained are relevant for two main reasons. Under the random matrix theory perspective, we were able to develop a general framework to prove polynomial central limit theorems for the classical $\beta$-ensemble in the high-temperature regime, based on their band matrix representation and on the transfer operator technique. Under the integrable systems' theory point of view, our result enables the explicit computation of the so-called \textit{susceptibility} matrix, which is a fundamental object in the theory of Generalized Hydrodynamics to compute the correlation functions for integrable models. Furthermore, we can prove rigorously the exponential decay of correlation for short-range interacting systems with polynomial potential.

It would be fascinating to generalize our result to a wider class of potential and functions and to obtain a Berry-Esseen bound for the classical $\beta$-ensemble in the high-temperature regime. Furthermore, defining a new $\beta$-ensemble related to the INB lattice would be interesting. Finally, we point out that it would be interesting to obtain large deviation principles for the Exponential Toda lattice and the Volterra one in the spirit of \cite{GMToda,mazzucamemin}.

\paragraph{Acknowledgements}\hfill
\newline
The authors warmly thank Alice Guionnet for giving the initial idea of the proof of the polynomial CLT, and Herbert Spohn for the useful discussions about the applications of our result. We also aknowledge Sébastien Gouëzel and two anonymous referees for useful remarks.

\noindent R.M. is supported by ERC Project LDRAM : ERC-2019-ADG Project 884584.

\noindent G.M. is financed by the KAW grant number 2018.0344

\appendix
\label{appendix}

\section{Proof of Lemma \ref{LEM:CIRCULAR_TRACE}}
\label{app:circ}
In this section we prove Lemma \ref{LEM:CIRCULAR_TRACE}, we report here the statement for reader convenience:

\begin{lemma}
    Consider any type 1-2 matrix $M$. Then for any polynomial $P\in \C[X]$, $\Tr{P(M)}$ is circular.
\end{lemma}

To prove this lemma, we need some preparation.
First, we introduce the family of shift operators $S_\ell$,  for any $\ell \in \Z$,  and $\bx=(x_1, x_2, \ldots, x_{N})\in X^{N}$  we define the {\em cyclic shift of order $\ell$} as the map  
\begin{equation}
	\label{shift}
	S_\ell \colon X^{N} \to X^{N}, \qquad (S_\ell x)_j := x_{((j+\ell-1)\mod N) +1} . 
\end{equation}
For example $S_1$ and $S_{-1}$ are  the left respectively right shifts:
$$
S_1(x_1, x_2, \ldots, x_{N}) := (x_2, \ldots, x_{N},  x_{1}), \qquad
S_{-1}(x_1, x_2, \ldots, x_{N}) := (x_{N}, x_1, \ldots, x_{N-1}).
$$
It is immediate to check that for  any $\ell, \ell' \in \Z$, cyclic shifts fulfills:
\begin{equation}
	\label{prop:shift}
	S_{\ell} \circ S_{\ell'} = S_{\ell + \ell'}, \qquad S_{\ell}^{-1} = S_{- \ell} , \qquad S_0 = \uno , \qquad
	S_{\ell + 2N } = S_{\ell} . 
\end{equation}
Consider now a  a function $H\colon X^N \to \C$;  we shall denote  $S_\ell H\colon X^N \to \C$  as the operator\begin{equation}
	\label{cyc.func}
	(S_\ell H)(\bx) := H(S_\ell \bx) , \qquad \forall  \bx \in X^N\,.
\end{equation}
Clearly, $S_\ell$ is a linear operator.
We can now define cyclic functions:
\begin{definition}
	\label{def:cyclic}
	A function $H\colon X^{N} \to \C$ is called {\em cyclic} if $S_1 H = H$.
\end{definition}
It follows from the definition that a cyclic function fulfils $S_\ell H = H$ $\, \forall \ell \in \Z$.\\
On can construct  cyclic functions as follows: given a function $h \colon X^N \to \C$  we define the  new function $H$ by
\begin{equation}
	\label{seed}
	H(\bx)  := \sum_{\ell = 0}^{2N-1} (S_{\ell} h)(\bx) .
\end{equation}
$H$ is clearly  cyclic, and  we say  that  $H$ is {\em generated } by $h$. We notice that these definitions were previously introduced in \cite{Giorgilli2014,Grava2020}

Our first aim is to show that the trace of powers of the matrices of type $1$ are cyclic functions and that the ones of the matrices of type 2 are cyclic up to a certain function. To this end, we introduce the companion matrices $P,T$ as

\begin{equation}
    P= \begin{pmatrix}
        0 & 1 &    \\
        0 & 0 &  1 & \\
        \vdots & \ddots & \ddots & \ddots \\
        0 &  & \ddots & \ddots  &  1 \\
        1 & 0 & \ldots & \ldots& 0 
    \end{pmatrix}\, \qquad T = \begin{pmatrix}
        0 & 1 & \\
        0 & 0 & 1 \\
        & \ddots & \ddots & \ddots
    \end{pmatrix}\,,
\end{equation}

so $T$ is a sparse matrix, with just the first upper diagonal filled with ones, and $P$ is similar to $T$, except that there is an extra $1$ in the bottom left corner. 

Using this notation, we can prove the following lemma  \textit{Toeplitz}-like and \textit{circular}-like matrices as follows

\begin{lemma}
\label{lem:cyclic}
    Fix $m,N \in \mathbb{N}$. Given the $2m+1$ sequences $\{a_j^{(k)}\}_{j=1}^N$ for $k=-m, m$ and defining the two matrices 

    \begin{equation}
    \begin{split}
          & A_1 = \sum_{j=1}^m\left( \diag{\ba^{(j)}}  P^j  + \diag{\ba^{(-j)}} P^{N-j}\right) + \diag{\ba^{(0)}}\,, \\
          & A_2 = \sum_{j=1}^m\left( \diag{\ba^{(j)}} T^j   +  \diag{\ba^{(-j)}} (T^\intercal)^{j}\right) + \diag{\ba^{(0)}}\,.
    \end{split}
    \end{equation}
    Here we denoted by $\diag{\ba}$ the diagonal matrices with entries $a_1, \ldots,a_N$.
    Then for all $\ell\in \mathbb{N}$, $\Tr{A_1^\ell}$ is cyclic, and $\Tr{A_2^\ell} = \Tr{A_1^\ell} + K(\ba)$, where $K(\ba)$ depends on a number of variables which is uniformly bounded in $N$. We call the matrices having the shape of $A_1$ \textit{circular}-like matrices, and the one having the shape of $A_2$ \textit{Toeplitz}-like matrices
\end{lemma}

\begin{proof}
    First, we focus on the matrix $A_1$, we expand the trace as

    \begin{equation}
        \Tr{A_1^\ell} = \sum_{i_1, \ldots,i_{\ell} = 1}^N \prod_{j=1}^{\ell-1} (A_1)_{i_j,i_{j+1}}  (A_1)_{i_\ell,i_1}  =  \sum_{i_1, \ldots,i_{\ell} = 1}^N \prod_{j=1}^{\ell-1} a^{( i_j - i_{j+1})}_{i_j} a^{( i_\ell - i_{1})}_{i_\ell}\,,
    \end{equation}
    where in the second equality, we used the structure of $A_1$, and we use the convention that if $\vert i_j - i_{j+1}\vert > m$ then $a^{( i_j - i_{j+1})}_{i_j} = 0$; from the previous expression is easy to see that $\Tr{A_1^\ell}$ is cyclic. Regarding the matrix $A_2$, we notice that $T$ is a rank $1$ perturbation of $P$, thus by developing the trace as before, one can easily obtain the claim.
\end{proof}

\begin{remark}
    We notice that the matrices of type $1-2$ are Toeplitz-like or circular-like. 
\end{remark}

        Using this Lemma, we can prove the following Corollary

        \begin{corollary}
        \label{cor:nice}
             Fix $m,N \in \mathbb{N}$. Let $\bx$ be the coordinates of the manifold $X^N$. Given the $2m+1$ sequences $\{a_j^{(k)}(\bx)\}_{j=1}^N$ for $k=-m, \ldots, m$ such that $a_j^{(k)}(\bx)$ is local for all $j,k$, then the trace of powers of  Toepliz-like or circular-like matrices are circular in the sense of Definition \ref{def:circular}
        \end{corollary}

    \begin{proof}
        From Lemma \ref{lem:cyclic} we know that the trace of powers of Circular-like matrices are cyclic, and they can be written as

        \begin{equation}
        \label{eq:trace_power}
            \Tr{A_1^\ell}  = \sum_{i_1, \ldots,i_{\ell} = 1}^N \prod_{j=1}^{\ell-1} a^{(i_j - i_{j+1})}_{i_j}(\bx) a^{( i_\ell - i_{1})}_{i_\ell}(\bx)\,.
        \end{equation}
        
    In the previous expression, gather all the terms involving $x_1$, and compute the diameter $d$ of the resulting function, given the locality property of $a_j^{(k)}(\bx)$ and the structure of the single terms of the product in \eqref{eq:trace_power}, we deduce that $0\leq d \leq D(m,\ell)$, thus do not depend on $N$, moreover, it is an even number because of the cyclicity property.

    Consider the function $h^{(\ell)}(x_1,\ldots,x_{d+1}) =\frac{1}{2} h^{(\ell)}_1(x_1,\ldots,x_{d/2+1}) +  \frac{1}{2} h^{(\ell)}_1(x_{d/2+1},\ldots,x_d) + h^{(\ell)}_2(x_1,\ldots,x_{d+1}) + h_3(x_{d/2 + 1})$ where
    \begin{itemize}
        \item $h^{(\ell)}_1(x_1,\ldots,x_{d/2 + 1})$ contains all the terms in \eqref{eq:trace_power} that involves just $x_1,\ldots,x_{d/2+1}$
        \item $h^{(\ell)}_2(x_1,\ldots,x_{d+1})$ contains all the terms in \eqref{eq:trace_power} that involves just $x_1,\ldots,x_{d+1}$, at least one variable among $x_1,\ldots,x_{d/2}$, and at least one variable among $x_{d/2 + 2},\ldots,x_{d+1}$
        \item $h_3(x_{d/2 + 1})$ contains terms involving just $x_{d/2 + 1}$
    \end{itemize}

    Given this construction, it is easy to see that $h^{(\ell)}(x_1, \ldots, x_{d+1})$ is the seed of $\Tr{A_1^\ell}$.

    Regarding the Toeplitz-like matrices, we get the claim because of their relation with the circular-like proved in Lemma \ref{lem:cyclic}.

    \end{proof}

        \begin{remark}
            The previous lemma implies that the trace of powers of type $1-2$ matrices are circular.  
        \end{remark}

    As an example, we consider $L$ to be a periodic Jacobi matrix, then  $$\Tr{L^4} = \sum_{i=1}^{N} \left[ a_i^4 + 4(a_i^2 + a_i a_{i+1} + a_{i+1}^2)b_i^2 + 2 b^4_i + 4 b_i^2b_{i+1}^2\right]\,,$$ then $d=2$ and

    \begin{equation}
        \begin{split}
            & h_1(a_1,b_1,a_2,b_2) = 4 a_1 a_2b_1^2 + 4a_2^2b_1^2 + 4b^2_1b^2_2\,,\\
            & h_1(a_N,b_N,a_1,b_1,a_2,a_3) = 0 \,,\\ 
            & h_3(a_1,b_1) = a_1^4 + 4a_1^2b_1^2 + 2b_1^4\,.
        \end{split}
    \end{equation}

\section{Proof of Lemma \ref{lem:differentiability}}
\label{App:proof differentiability}

\begin{lemma}
    Under assumptions \ref{transfer_hp}. For $\alpha>0$, $t\in \R$, consider the operators $\cL_{\alpha,t},\cL_{t,\alpha}^{(j)}$, $1\leq j \leq M-1$ given by Definition \ref{def:operators}. Assume that for some $d\geq 2$, $|u|^a\exp\left(-h\right)$ is bounded for any $1\leq a \leq d$, then
    
    \begin{itemize}
    \item[1)] Both $\cL_{\alpha,t}$ and $\cL_{t,\alpha}^{(j)}$ are Hilbert-Schmidt.        \item[2)] For any fixed $t\in \R$, the $\mathbf{L}(L^2(X^K))$-valued functions $\alpha\mapsto \cL_{\alpha,t}$ and $\alpha\mapsto\cL_{\alpha,t}^{(j)}$ are continuously differentiable. 
        \item[3)] For any fixed $\alpha>0$, $t\mapsto \cL_{\alpha,t}$ and $\alpha\mapsto \cL_{\alpha,t}^{(j)}$ are $d$-times continuously differentiable.
    \item[4)] For all $\alpha>0$, define 
    $$ \cB_{\alpha}=\frac{1}{\|\cL_{\alpha,0}\|_\mathrm{HS}}\cL_{\alpha,0} $$
    and 
    $$ \cB_\alpha^{(j)}=\frac{1}{\|\cL_{\alpha,0}^{(j)}\|_\mathrm{HS}}\cL_{\alpha,0}^{(j)},\ 1\leq j\leq M-1\,.$$
    Then, $\cB_{\alpha}$ (resp. $\cB_\alpha^{(j)}$)
    has a dominant, simple eigenvalue $\mu(\alpha)>0$ (resp. $\mu^{(j)}(\alpha)$). Furthermore, for all $A>0$ there exists $\delta>0$ such that and all $\alpha\in (0,A]$, any other eigenvalue $\mu$ satisfies $|\mu-\mu(\alpha)|\geq \delta$ (resp. $|\mu-\mu^{(j)}(\alpha)|\geq \delta$). We say that those operators have a \textbf{uniform spectral gap} in $\alpha\in (0,A]$.
    \end{itemize}
\end{lemma}

\begin{remark}
    We notice that $1)$ and $2)$ imply that $\cL_{\alpha,t},\cL_{t,\alpha}^{(j)} \in C^{1,d}((0,+\infty)\times \R, \mathbf{L}(L^2(X^k)))$
\end{remark}
\begin{proof}\ \\
 \textit{Proof of 1)} Since $F\in L^1(X)$ and that $h$ is bounded away from $-\infty$, we easily see that $k_{\alpha,t}$ is $L^2(X^k\times X^k)$. Therefore $\cL_{\alpha,t}$ is Hilbert Schmidt, and the same goes for $\cL_{\alpha,t}^{(j)}$, $1\leq j \leq M-1$.

\noindent \textit{Proof of 2) \& 3)} We show these points for $\cL_{\alpha,t}$, the proofs for $\cL_{t,\alpha}^{(j)}$ being similar. Let us show that for any fixed $t$, $\alpha \mapsto \cL_{\alpha,t}$ is continuously differentiable with derivative given by 
$$ \left(\partial _\alpha \cL_{\alpha,t}\right) u(\by) = \int_{X^k} \partial_\alpha k_{\alpha,t}(\bx,\by)u(\bx)\di \bx\,.$$
Let $\alpha>0$ and $h\in \R$ such that $\alpha+h>0$. For $\varphi\in L^2(X^k)$ we have
\begin{align*}
    \left\|\frac{1}{h}\left( \cL_{\alpha+h,t}-\cL_{\alpha,t} \right) \varphi - \left(\partial _\alpha \cL_{\alpha,t}\right)\varphi \right\|_{2}^2 &\leq \int_{X^k} \left(\int_{X^k}\left|\frac{1}{h}\big(k_{\alpha+h,t}(\bx,\by)-k_{\alpha,t}(\bx,\by) \big) - \partial_\alpha k_\alpha(\bx,\by) \right||\varphi(\bx)|\di \bx\right)^2 \di \by \\
                          &\leq \|\varphi\|_{2}^2\int_{X^k\times X^k} \left| \frac{1}{h}\big( k_{\alpha+h,t}(\bx,\by)-k_{\alpha,t}(\bx,\by)\big) - \partial_\alpha k(\bx,\by) \right|^2\di \bx \di\by\,,
\end{align*}
where $\|\cdot \|_2$ is the standard $L^2$ norm, and we used Cauchy-Schwartz inequality in the last line.

Now,
\begin{align*}
    \partial_\alpha k(\bx,\by) &= e^{-h(\bx,\by)}\partial_\alpha \prod_{q=1}^k \sqrt{F(x_q,\alpha)F(y_q,\alpha)} \\
                              &= e^{-h(\bx,\by)}\left(\left(\sum_{q=1}^k \partial_\alpha \sqrt{F(x_q,\alpha)}\prod_{j\neq q} F(x_j\alpha)\right)\prod_{q=1}^k \sqrt{F(y_q,\alpha)} + \prod_{q=1}^k \sqrt{F(x_q,\alpha)}\times \partial_\alpha \prod_{q=1}^k \sqrt{F(y_q,\alpha)}\right)\,.
\end{align*}
Because of the domination hypothesis on $\partial_\alpha \sqrt{ F(\cdot,\alpha)}$ HP.3 and using that $e^{-h}\leq c_k\in \R_+$ since $h$ is bounded away from $-\infty$ HP.4, we can bound for $\alpha\in [c,d]\subset (0,+\infty)$
$$ |\partial_\alpha k(\bx,\by)| \leq 2 c_k k\prod_{q=1}^k g(x_q)\prod_{q=1}^k g(y_q) \in L^2(X^k\times X^k)\,,$$
where $g=g_{[c,d]}\in L^2(X)$ is a function given by HP.3, independent of $\alpha$. We deduce by dominated convergence, taking $h\to 0$ that 
$$ \frac{1}{h}\left(\cL_{\alpha+h,t} - \cL_{\alpha,t}\right)= \partial_\alpha \cL_{\alpha,t} + o_h(1)\,,$$
where the $o(1)$ is with respect to the operator norm $\|\cdot\|$ on $\mathbf{L}(L^2(X^k))$, establishing that $\alpha\mapsto \cL_{\alpha,t}$ is differentiable. By a similar dominated convergence argument, one shows that $\alpha\mapsto \partial_\alpha\cL_{\alpha,t}$ is continuous, establishing that $\cL_{\alpha,t}$ is continuously differentiable with respect to $\alpha$.\\
Similarly, let $\alpha>0$ be fixed and let $1\leq a \leq d$. Using that $|u^a|e^{-h}$ is bounded by some $D_k\in \R_+$,
\begin{align*}
    \partial_t^a k_{\alpha,t}(\bx,\by) &= (-i)^a u(\bx,\by)^a \prod_{q=1}^k \sqrt{F(x_q,\alpha)F(y_q,\alpha)} e^{-h(\by,\bx)-itu(\by,\bx)} \\
          &\leq D_k \prod_{q=1}^k \sqrt{F(x_q,\alpha)F(y_q,\alpha)} \in L^2(X^k\times X^k)\,,
\end{align*}
obtaining an $L^2$ bound independent of $t$ for $\partial_t^a k_{\alpha,t}$. Therefore, by dominated convergence, $t\mapsto \cL_{\alpha,t}$ is at least $d$ times differentiable, with continuous derivative of order $d$.
\ \\

\noindent\textit{Proof of 4)} We now establish the uniform spectral gap property for $\cB_{\alpha}$. The proof for $\cB_{\alpha}^{(j)}$ is similar. First, by a generalization of Jentzsch's theorem, because the kernel $k_\alpha(\bx,\by)$ is positive, $\cB_{\alpha,0}$ has a simple, positive, dominant eigenvalue \cite[Theorem 137.4]{ZaanenBook}, which we denote by $\mu(\alpha).$ By continuity of $\alpha\mapsto \cB_\alpha$, for each $a>0$ we can find $\delta>0$ as in Proposition \ref{prop:gauezel} applied to the family $(\cB_\alpha)_{\alpha>0}$ which is uniform in $\alpha \in [a,A]$. Thus we only need to check that for some $a>0$, $\cB_{\alpha}$ has a uniform spectral gap in $\alpha\in (0,a]$.

Our strategy is based on the observation that for a kernel of the form $k(\bx,\by)=f(\bx)g(\by)$ for $0< f,g \in L^2(X^k)$, the associated kernel operator $T$ is of rank one, so it has a simple dominant eigenvalue $\lambda=\int_{X^k}g(\bx)f(\bx)\di \bx>0$ associated with the eigenfunction $g$.

With this in mind, we want to use the fact that in $\mathcal{O}_\varepsilon$, defined in \ref{transfer_hp} HP 4 the expression $\exp(-h(\bx,\by))$ factorizes into $\exp(-w(\bx))\exp(-w(\by))$. 

\textbf{Step 1: }Define, for $\varepsilon>0$, 
$$ \cL^\varepsilon_{\alpha}u(\by)=\int_{\mathcal{O}_\varepsilon}k_{\alpha,0}(\bx,\by)u(\bx)\di \bx\,,$$
where we recall that $k_{\alpha,0}$ is the kernel of $\cL_{\alpha}$

Consider the following operator
$$\cB_\alpha^\varepsilon=\frac{1}{\|\cL_\alpha\|_\mathrm{HS}}\cL^\varepsilon_\alpha\,.$$
We first show that for all $\varepsilon>0$, $ \left\| \cB_\alpha - \cB_\alpha^\varepsilon \right\|_\mathrm{HS} \underset{\alpha\to 0}\to 0\,.$

Indeed, with $c_\alpha= \| \cL_\alpha \|_\mathrm{HS}$, and denoting for simplicity 
$$ \mathsf{F}_\alpha(\bx) = \prod_{j=1}^k F(x_j,\alpha)\,,$$
we have
\begin{equation}
        \left\| \cB_\alpha - \cB_\alpha^\varepsilon \right\|_\mathrm{HS} \leq \left\|\frac{\cL_\alpha}{c_\alpha}-\frac{\cL_\alpha^\varepsilon}{c_\alpha} \right\|=\frac{1}{c_\alpha}\left(\iint_{\mathcal{O}_\varepsilon^c\times \mathcal{O}_\varepsilon^c} \mathsf{F}_\alpha(\bx)\mathsf{F}_\alpha(\by)e^{-2h(\by,\bx)}\di \bx \di \by\right)^{1/2} \,.
\end{equation}
Applying \ref{general_assumptions} HP 4, we can conclude that for any fixed $\varepsilon>0$, 
\begin{equation}
    \lim_{\alpha\to0}\left\| \cB_\alpha - \cB_\alpha^\varepsilon \right\|_\mathrm{HS} =0\,.
\end{equation}

\textbf{Step 2:} Define $\mathcal{T}_\alpha^\varepsilon$ as
$$ \mathcal{T}_\alpha^\varepsilon u(\by)=\frac{1}{c_\alpha}\int_{\mathcal{O}_\varepsilon}\sqrt{\mathsf{F}_\alpha(\bx)}\sqrt{\mathsf{F}_\alpha(\by)}e^{-w(\bx)-w(\by)}\di x\,.$$

Then, by the previous remark, $\mathcal{T}_\alpha^\varepsilon$ is a rank one operator with a leading simple eigenvalue, and it has a spectral gap independent of $\alpha$. Meaning that there exists a $\delta$ independent of $\alpha$ such that $\|\mathcal{T}_\alpha^\varepsilon\| > \delta$. Furthermore, as in step 1, we show that $\| \cB_\alpha^\varepsilon - \mathcal{T}_\alpha^\varepsilon\|_\mathrm{HS}\underset{\varepsilon\to 0}{\to} 0$ uniformly in $\alpha>0$. Therefore, proposition \ref{prop:gauezel} applies with a uniform $\delta >0$: i.e. there exists $\varepsilon_0>0$ such that for all $\varepsilon<\varepsilon_0$ and for all $\alpha>0$, $\cB_\alpha^\varepsilon$ has a simple, dominant eigenvalue $\wt \mu(\alpha)$, such that any other eigenvalue $\mu$ of $\cB_\alpha^\varepsilon$ satisfies $|\mu-\wt \mu(\alpha)|>2\delta$.

\textbf{Conclusion: }We now apply the result of step 1 with $\varepsilon=\varepsilon_0$: we know that $\| \cB_\alpha - \cB_\alpha^{\varepsilon_0} \|_\mathrm{HS} \underset{\alpha\to 0}{\to} 0$. Therefore, by Proposition \ref{prop:gauezel}, there exists $\alpha_0$ such that for all $\alpha<\alpha_0$, $\cB_\alpha$ has a simple, dominant eigenvalue separated from the rest of the spectrum of at least $\delta$. 

\end{proof}

\bibliographystyle{siam}

\bibliography{bibAbel3}
\end{document}